\documentclass[smallextended]{svjour3}

\usepackage{xparse}

\usepackage{amsmath}
\usepackage{amssymb}
\usepackage{stmaryrd}
\usepackage{mathrsfs}
\usepackage{xfrac}

\usepackage{siunitx}
\usepackage{booktabs}
\usepackage[shortlabels]{enumitem}
\usepackage{algorithm}
\usepackage{algpseudocodex}

\usepackage[group-separator={,}]{siunitx}

\makeatletter
\let\cl@chapter\undefined
\makeatletter
\usepackage{hyperref}
\usepackage[nameinlink, noabbrev]{cleveref}

\usepackage{graphicx}
\usepackage{tikz}
\usepackage{pgfplots}
\usepackage{subcaption}

\usepackage[mathlines]{lineno}

\pgfplotsset{compat=1.18}

\definecolor{qCluster}{RGB}{128, 0, 128}
\definecolor{qNear}{RGB}{0, 206, 209}
\definecolor{qFar}{RGB}{205, 91, 69}
\definecolor{qProxy}{RGB}{0, 0, 0}

\definecolor{akColor}{RGB}{255, 95, 5}
\definecolor{afColor}{RGB}{30, 56, 119}

\usepackage{xspace}

\NewDocumentCommand \stageone {} {\emph{stage-1}}
\NewDocumentCommand \stagetwo {} {\emph{stage-2}}
\NewDocumentCommand \stagequad {} {\emph{stage-2-quad}}

\NewDocumentCommand \rtgt {} {\ensuremath{r_{\text{tgt}, i}}}

\NewDocumentCommand \rpxy {} {\ensuremath{r_{\text{pxy}, i}}}

\NewDocumentCommand \ctgt {} {\ensuremath{\vect{c}_{\text{tgt}, i}}}

\NewDocumentCommand \coloneqq {o} {\mathrel{\mathop:}\mathrel{\mkern-1.2mu}=}
\NewDocumentCommand \dx { O{x} } {\mathop{}\!\mathrm{d} #1}
\NewDocumentCommand \vect { m } {\boldsymbol{#1}}
\NewDocumentCommand \matr { m } {\boldsymbol{#1}}

\newtheorem{method}{Method}

\AddToHook{env/method/begin}{\crefalias{theorem}{method}}

\crefname{algocf}{algorithm}{algorithms}
\Crefname{algocf}{Algorithm}{Algorithms}

\AddToHook{cmd/appendix/before}{\def\cref@section@alias{appendix}}

\def\papertitle{
  A Fast Direct Solver for Boundary Integral Equations Using
  Quadrature By Expansion}
\def\papershorttitle{QBX Direct Solver}

\def\paperauthora{Alexandru Fikl}
\def\papershortauthora{A. Fikl}
\def\paperemaila{alexandru.fikl@e-uvt.ro}
\def\paperaddressa{Institute for Advanced Environmental Research (ICAM), West University of Timişoara, Bd. V. Pârvan nr. 4, 300223}

\def\paperauthorb{Andreas Klöckner}
\def\papershortauthorb{A. Klöckner}
\def\paperemailb{andreask@illinois.edu}
\def\paperaddressb{Department of Computer Science, University of Illinois at Urbana-Champaign, 201 N. Goodwin Ave, Urbana, IL 61801}

\smartqed

\newif\ifreview
\reviewfalse

\ifreview
    \usepackage[mathlines]{lineno}

    \linenumbers

    \NewDocumentCommand \afnote { O{TODO} m } {
        \textcolor{afColor}{\textbf{Alex} (\textbf{#1})}: #2 }
    \NewDocumentCommand \aknote { O{TODO} m } {
        \textcolor{akColor}{\textbf{Andreas} (\textbf{#1})}: #2}
\else
    \NewDocumentCommand \afnote { O{TODO} m } {}
    \NewDocumentCommand \aknote { O{TODO} m } {}
\fi

\definecolor{ReviewerA}{HTML}{0072BD}
\definecolor{ReviewerB}{HTML}{D95319}

\ifreview
    \NewDocumentCommand \reva { +m } {\textcolor{ReviewerA}{#1}}
    \NewDocumentCommand \revb { +m } {\textcolor{ReviewerB}{#1}}
\else
    \NewDocumentCommand \reva { +m } {#1}
    \NewDocumentCommand \revb { +m } {#1}
\fi

\newcommand{\doi}[1]{\href{https://doi.org/#1}{doi: #1}}

\begin{document}


\journalname{Journal of Scientific Computing}
\date{Receive: 23 May 2025 / Revised: ... / Accepted: ...}

\title{\papertitle}
\titlerunning{\papershorttitle}

\author{\paperauthora$^1$ \and \paperauthorb$^2$}
\authorrunning{\papershortauthora \and \papershortauthorb}

\institute{
    \paperauthora$^{*, 1}$
    \at \paperaddressa \\
    \email{\paperemaila} $^*$ corresponding author
    \and
    \paperauthorb$^{2}$
    \at \paperaddressb \\
    \email{\paperemailb}
}

\maketitle


\begin{abstract}
We construct and analyze a hierarchical direct solver for linear systems arising
from the discretization of boundary integral equations using the Quadrature by
Expansion (QBX) method. Our scheme builds on the existing theory of Hierarchical
Semi-Separable (HSS) matrix operators that contain low-rank off-diagonal submatrices.
We use proxy-based approximations of the far-field interactions and the Interpolative
Decomposition (ID) to construct compressed HSS operators that are used as fast
direct solvers for the original system. We describe a number of modifications to
the standard HSS framework that enable compatibility with the QBX family of
discretization methods. We establish an error model for the direct solver that
is based on a multipole expansion of the QBX-mediated proxy interactions and
standard estimates for the ID\@. Based on these theoretical results,
we develop an automatic approach for setting scheme parameters based on user-provided
error tolerances. The resulting solver seamlessly generalizes across two- and
three-dimensional problems and achieves state-of-the-art asymptotic scaling. We
conclude with numerical experiments that support the theoretical expectations for
the error and computational cost of the direct solver.
\end{abstract}



\section{Introduction}
\label{sc:introduction}

Linear elliptic boundary value problems frequently appear in models from
physics (e.g.,\ electrostatics), biology (e.g.,\ vesicle transport), engineering
(e.g.,\ wave scattering), and many others. For large families of homogeneous elliptic
equations the fundamental Green function solution is known, allowing them to be
solved by means of boundary integral methods. These methods lead to dense linear
systems that can, in many cases, be solved efficiently~\cite{Greengard1987,Greengard1997}
through the use of iterative algorithms. However, in cases where this is not possible,
we can take advantage of the special structure of these matrices to
construct efficient direct solvers. For example, Laplace and Helmholtz problems
produce matrices with low-rank off-diagonal blocks. These matrices can benefit
from compression and improved asymptotic cost, as has been shown through the
development of techniques such as $\mathcal{H}$~\cite{Hackbusch1999},
$\mathcal{H}^2$~\cite{Hackbusch2002}, hierarchical semi-separable
(HSS)~\cite{Chandrasekaran2006,Ho2014}, hierarchical off-diagonal low-rank
(HODLR)~\cite{Ambikasaran2013}.

We focus on the hierarchical semi-separable (HSS) class of matrices, which are
defined as being recursively low-rank in their off-diagonal blocks. This allows
constructing compression schemes that can achieve linear time in both two and three
dimensions~\cite{Cheng2005,Greengard2009,Martinsson2005}. Among these, an
efficient family of methods is based on the so-called ``proxy skeletonization'',
which makes use of the underlying PDE and the Interpolative Decomposition (ID)
to achieve $O(N)$ time algorithms~\cite{Gillman2012} in 2D and $O(N^{\frac{3}{2}})$
algorithms in 3D~\cite{Ho2012,Ho2013}. Three-dimensional decompositions have
also been extended to linear $O(N)$ algorithms in~\cite{Ho2016}. These algorithms
compete with iterative solvers driven by classically known Fast Multipole
Methods (FMM)~\cite{Greengard1987,Greengard1997} that achieve linear
time scaling, but tend to struggle for systems that have poorly conditioned discrete
forms.

In this work, we make the following contributions:
\begin{itemize}
\item
We develop an HSS compression scheme based on proxy skeletonization
for first- and second-kind operators discretized using the Quadrature by Expansion
(`QBX') method. In order to maintain accuracy, geometry processing and
formation of proxy expansions need to be adapted.
\item
In the development of our method, we introduce a novel near/far-field weighting scheme
that, empirically, reduces \revb{solution} error by about an order of magnitude (see \Cref{fig:results:target_weight_matrix})%
\revb{, compared to no usage of this additional weighting factor}.
This improvement is applicable independently of the use of QBX quadrature.
\item
Building on~\cite{Ye2020} for the two-dimensional case and~\cite{Xing2020a,Xing2020b}
for the three-dimensional case, we develop an end-to-end error bound for forward
compressed operator application. Our model accounts for the system matrix's block structure,
the separation for near  and far field, and the rank-based compression of far
field-contributions.
\item
We demonstrate the applicability of our modifications in the context of a recursive
direct solver in both two and three dimensions, for arbitrary geometries.
\item
To improve practical usability of the method, we develop a novel, theory-based approach
for choosing a number of free parameters (proxy count, proxy radius). We
provide empirical validation of these parameter choices on a number of operators across
multiple geometries in two and three dimensions.
This scheme is applicable independently of the use of QBX quadrature.
\item
We present extensive numerical evidence that supports our theoretical claims,
examining both details of the error theory and end-to-end estimates
for the forward operator application and system solution process, again
across dimensions, operators, and geometries.
\item
We provide open-source software implementing our methods as part of the
\texttt{pytential} library \cite{Pytential_2025_7395b97f} for the benefit of
the community.
\end{itemize}

The remainder of the paper is organized as follows. In~\Cref{sc:prelim}, we briefly
present the model problem, discuss its discretization using the QBX method, and
describe standard aspects of the Interpolative Decomposition used by the
direct solver. In~\Cref{sc:schemes}, we present the proxy-based skeletonization
of the linear system and the construction of the recursive direct solver.
In~\Cref{sc:errors}, we develop an error model for the single-level direct
solver approximation based on the multipole expansion and standard matrix
inequalities. Then, in~\Cref{sc:results}, we numerically validate the error
model on several representative problems in both two and three dimensions. Finally,
in~\Cref{sc:conclusions} we summarize the results and additional extensions and
open questions.

\section{Preliminaries}
\label{sc:prelim}

We focus on solving linear boundary integral equations of the form
\begin{equation} \label{eq:prelim:lp}
a(\vect{x})
\sigma(\vect{x}) +
\int_\Sigma K(\vect{x}, \vect{y}) \sigma(\vect{y}) \dx[S] = b(\vect{x}),
\end{equation}
that commonly appear in boundary integral methods. In the above equation, we take
$\vect{x}, \vect{y} \in \Sigma \subset \mathbb{R}^d$, for a smooth closed
hypersurface $\Sigma$, where $d \ge 2$. The boundary data $b: \Sigma \to \mathbb{R}$
is prescribed, and the equation is solved for the density $\sigma: \Sigma \to \mathbb{R}$.
For all equations of interest, $K: \mathbb{R}^d \times \mathbb{R}^d \to \mathbb{R}$
is a singular kernel and $a: \Sigma \to \mathbb{R}$ is a scaling factor that depends
on the choice of kernel $K$ (see~\cite{Kress2013}), typically via its jump relations.
The theoretical results presented in~\Cref{sc:errors} are based on the Laplace
Green function, which is given by
\begin{equation} \label{eq:prelim:laplace_kernel}
G^{\text{2D}}(\vect{x}, \vect{y}) \coloneqq
  -\frac{1}{2 \pi} \log \|\vect{x} - \vect{y}\|_2
\quad \text{and} \quad
G^{\text{3D}}(\vect{x}, \vect{y}) \coloneqq
  \frac{1}{4 \pi} \frac{1}{\|\vect{x} - \vect{y}\|_2},
\end{equation}
where $\|\cdot\|_2$ is the Euclidean norm. Representations of the
Dirichlet and Neumann boundary value problems for an elliptic operator
can be chosen as layer potentials \cite[Section 6.3]{Kress2013} in the
form~\eqref{eq:prelim:lp}. For example, the solutions to the Dirichlet problem
for the Laplace equation can be represented by a single-layer potential, so that
$a(\vect{x}) =0$ and $K(\vect{x}, \vect{y}) =G(\vect{x}, \vect{y})$,
or by a double-layer potential, where $a(\vect{x}) =\pm \sfrac{1}{2}$ and
$K(\vect{x}, \vect{y}) =\vect{n}_{\vect{y}} \cdot \nabla_{\vect{y}}
G(\vect{x}, \vect{y})$. \revb{In this case}, we write the
\reva{left-hand side} of the integral equations~\eqref{eq:prelim:lp} as
\begin{equation} \label{eq:prelim:layer_potential}
\begin{aligned}
\mathcal{S}[\sigma](\vect{x}) & \coloneqq
  \int_{\Sigma} G(\vect{x}, \vect{y}) \sigma(\vect{y}) \dx[S], \\
\pm \frac{1}{2} \sigma(\vect{x}) + \mathcal{D}[\sigma](\vect{x}) & \coloneqq
\pm \frac{1}{2} \sigma(\vect{x}) + \int_\Sigma
  \vect{n}_{\vect{y}} \cdot \nabla_{\vect{y}} G(\vect{x}, \vect{y}) \sigma(\vect{x}) \dx[S].
\end{aligned}
\end{equation}

The double-layer representation is generally preferred \revb{for Dirichlet problems}
because the resulting equation \eqref{eq:prelim:lp} becomes a Fredholm integral equation
of the second kind that is uniquely solvable for all right-hand sides
$b$~\cite[Theorem 6.23]{Kress2013}. Furthermore, Nyström-based methods result in
discrete operators with good convergence properties. On the other hand, the
single-layer representation results in a Fredholm integral equation of the first
kind that is known to be ill-posed \cite[Chapter 15]{Kress2013} and is solvable
only under specific conditions on the right-hand side $b$
\cite[Theorem 15.18]{Kress2013}. \revb{On the other hand, a single-layer
representation is preferred for the Neumann problem, where the double-layer representation
results in a hypersingular integral equation.} Assuming the continuous BVP is uniquely
solvable, this type of equation poses numerical issues and standard iterative
methods encounter significant difficulties.

Upon discretization, the linear system~\eqref{eq:prelim:lp} becomes finite-dimensional
and can be written generally as
\begin{equation} \label{eq:prelim:axb}
  \matr{A} \vect{\sigma} = \vect{b},
\end{equation}
where $\matr{A} \in \mathbb{R}^{n \times n}$ is a dense matrix. To obtain the
finite-dimensional system, we use a Nyström discretization, where the Quadrature
by Expansion (QBX)~\cite{Kloeckner2013} method is applied for singularity
handling. For most kernels of interest, the resulting matrix $\matr{A}$ is known
to be low-rank in its off-diagonal entries, but not necessarily well-conditioned,
e.g., in the case for certain high aspect ratio geometries. This motivates the
construction of blockwise direct solvers. A block decomposition is required to
take advantage of the off-diagonal low-rank structure of $\matr{A}$ and a direct
solver may be preferred (compared to iterative solvers) when $\matr{A}$ is
ill-conditioned.

Below we present a brief review of the QBX method and the necessary linear
algebra tools used to describe and analyze the direct solver. This includes
error estimates for the Interpolative Decomposition and multipole expansions.

\subsection{Kernel Expansions}
\label{ssc:prelim:addition}

A key component of local and multipole expansions in potential theory is an
appropriate addition theorem. We focus here on the Laplace kernel, for which
addition theorems are known in two dimensions~\cite{Greengard1987} and
three dimensions~\cite[Equation 3.16]{Greengard1997}. We expand here on the
three-dimensional case, for which we have that
\begin{equation} \label{eq:prelim:laplace_addition}
P_\ell(\hat{\vect{x}} \cdot \hat{\vect{y}}) =
  \frac{4 \pi}{2 \ell + 1} \sum_{m = -\ell}^\ell
  Y^m_\ell(\hat{\vect{x}}) Y^{-m}_\ell(\hat{\vect{y}}),
\end{equation}
where $\hat{\vect{x}}, \hat{\vect{y}} \in \mathbb{S}^2$ are two points on the
unit sphere and $P_\ell$ are the standard Legendre polynomials. The spherical
harmonics $Y^m_\ell$ of degree $\ell$ and order $m$ are given by
\begin{equation} \label{eq:prelim:spherical_harmonic}
Y^m_\ell(\hat{\vect{x}}) = Y^m_\ell(\theta, \phi) =
\sqrt{\frac{2 \ell + 1}{4 \pi} \frac{(\ell - m)!}{(\ell + m)!}}
e^{\imath m \phi} P^m_\ell(\cos \theta),
\end{equation}
where $P^m_\ell$ are the associated Lagrange functions~\cite{Greengard1997}. The
spherical harmonics are normalized such that they form an orthonormal basis of
$L^2(\mathbb{S}^2)$. With this normalization, we have that~\cite[Corollary 2.9]{Stein1971}
\begin{equation} \label{eq:prelim:sph_bound}
\left|
\sum_{m = -\ell}^\ell Y^m_\ell(\hat{\vect{x}}) Y^{-m}_\ell(\hat{\vect{y}})
\right|
= \left|\frac{2 \ell + 1}{4 \pi} P_\ell(\hat{\vect{x}} \cdot \hat{\vect{y}})\right|
\le \frac{2 \ell + 1}{4 \pi}.
\end{equation}

Then, we write the series expansion of the free-space Laplace Green function as
\begin{equation} \label{eq:prelim:laplace_expansion}
G(\vect{x}, \vect{y}) =
\frac{1}{4 \pi} \sum_{\ell = 0}^\infty
  \frac{r_y^\ell}{r_x^{\ell + 1}} P_\ell(\hat{\vect{x}} \cdot \hat{\vect{y}}) =
\sum_{\ell = 0}^\infty \sum_{m = -\ell}^\ell
  \frac{1}{2 \ell + 1}
  \frac{r_y^\ell}{r_x^{\ell + 1}}
  Y^m_\ell(\hat{\vect{x}}) Y^{-m}_\ell(\hat{\vect{y}}),
\end{equation}
where $\hat{\vect{x}} \coloneqq \vect{x} / r_x$ and $r_x \coloneqq \|\vect{x}\|_2$
(equivalently for $\vect{y}$). This expansion is valid for all $\vect{x}, \vect{y}
\in \mathbb{R}^3$ as long as $r_y < r_x$. On the sphere of radius $R$, the
Green's function solution to the exterior Dirichlet Laplace problem is given by
the Poisson kernel~\cite[Theorem 1.9]{Stein1971}
\begin{equation} \label{eq:prelim:poisson}
P(\vect{x}, \vect{p}) =
  \frac{1}{4 \pi R} \frac{R^2 - \|\vect{x}\|_2^2}{\|\vect{x} - \vect{p}\|_2^3},
\end{equation}
where $\vect{x} \in \mathbb{R}^3 \setminus \mathbb{B}^2_{R}$ (the exterior of the
ball) and $\vect{p} \in \mathbb{S}^2_{R}$. Following~\cite[Proposition 1.9]{Atkinson2012}, the Poisson
kernel can also be expanded in terms of spherical harmonics as
\begin{equation} \label{eq:prelim:poisson_expansion}
P(\vect{x}, \vect{p}) =
\sum_{\ell = 0}^\infty
  \frac{2 \ell + 1}{4 \pi} \frac{R^{\ell - 1}}{r_y^{\ell + 1}}
  P_\ell(\hat{\vect{x}} \cdot \hat{\vect{p}})
=
\sum_{\ell = 0}^\infty \sum_{m = -\ell}^\ell
  \frac{R^{\ell - 1}}{r_x^{\ell + 1}}
  Y^m_\ell(\hat{\vect{x}}) Y^{-m}_\ell(\hat{\vect{p}}),
\end{equation}
where $\vect{p}$ is positioned on a sphere of radius $R$ and $r_x > R$.
Moreover, we have the \emph{Poisson integral identity}~\cite[Theorem 1.10]{Stein1971}
which uniquely defines a harmonic function $g: \mathbb{R}^3 \setminus
\mathbb{B}^2_{R} \to \mathbb{R}$ by
\begin{equation} \label{eq:prelim:poisson_integral}
g(\vect{x}) = \int_{\mathbb{S}^2_{R}}
  P(\vect{x}, \vect{p}) g(\vect{p}) \dx[S_{\vect{p}}],
\end{equation}
as obtained by solving the Dirichlet problem on the exterior of the
unit sphere. These definitions are used in~\Cref{ssc:errors:multipole} to
construct an error estimate for the direct solver based on standard multipole
estimates.

\subsection{Quadrature by Expansion (QBX)}
\label{ssc:prelim:qbx}

The central principle of the QBX method is the observation that the layer
potential is locally smooth away from the boundary. On the one hand, this allows the
use of standard high-order quadrature to accurately evaluate the solution at
points $\vect{x} \notin \Sigma$. More importantly, it allows the construction of
an appropriate ``local expansion'' (e.g., a Taylor expansion for the Laplace kernel)
that can approximate the solution in a neighborhood of the boundary. This
expansion can be used to evaluate the layer potential both near the surface
\cite{Barnett2014} and on-surface~\cite{Kloeckner2013} to high-order accuracy
and in a manner that generalizes straightforwardly across kernels and dimensions.

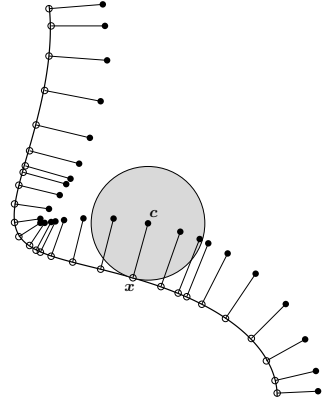
\begin{figure}[ht!]
  \centering
  \resizebox{0.35\linewidth}{!}{
  \begin{tikzpicture}
  \draw[ultra thick, domain=0.0:51.6, samples=512]
    plot ({5.00*(4+cos(6*\x)) * cos(\x)},
          {5.00*(4+cos(6*\x)) * sin(\x)});

  \draw[fill=gray!30, opacity=10]
      (19.0, 8.03) circle [radius=2.64];
  \node at (19.2, 8.23)
        [above] {{ \huge $\boldsymbol{{c}}$ }};
  \node at (18.1, 5.3)
      [below] {{ \huge $\boldsymbol{{x}}$ }};
  \draw[thick] (25.0, 0.143) -- (26.9, 0.231);
  \draw[black] (25.0, 0.143) circle [radius=0.15];
  \fill[black] (26.9, 0.231) circle [radius=0.15];

  \draw[thick] (24.9, 0.725) -- (26.8, 1.17);
  \draw[black] (24.9, 0.725) circle [radius=0.15];
  \fill[black] (26.8, 1.17) circle [radius=0.15];

  \draw[thick] (24.5, 1.64) -- (26.3, 2.64);
  \draw[black] (24.5, 1.64) circle [radius=0.15];
  \fill[black] (26.3, 2.64) circle [radius=0.15];

  \draw[thick] (23.8, 2.68) -- (25.4, 4.28);
  \draw[black] (23.8, 2.68) circle [radius=0.15];
  \fill[black] (25.4, 4.28) circle [radius=0.15];

  \draw[thick] (22.6, 3.61) -- (24.0, 5.69);
  \draw[black] (22.6, 3.61) circle [radius=0.15];
  \fill[black] (24.0, 5.69) circle [radius=0.15];

  \draw[thick] (21.5, 4.27) -- (22.7, 6.63);
  \draw[black] (21.5, 4.27) circle [radius=0.15];
  \fill[black] (22.7, 6.63) circle [radius=0.15];

  \draw[thick] (20.8, 4.62) -- (21.8, 7.1);
  \draw[black] (20.8, 4.62) circle [radius=0.15];
  \fill[black] (21.8, 7.1) circle [radius=0.15];

  \draw[thick] (20.4, 4.79) -- (21.4, 7.3);
  \draw[black] (20.4, 4.79) circle [radius=0.15];
  \fill[black] (21.4, 7.3) circle [radius=0.15];

  \draw[thick] (19.6, 5.09) -- (20.5, 7.64);
  \draw[black] (19.6, 5.09) circle [radius=0.15];
  \fill[black] (20.5, 7.64) circle [radius=0.15];

  \draw[thick] (18.3, 5.5) -- (19.0, 8.03);
  \draw[black] (18.3, 5.5) circle [radius=0.15];
  \fill[black] (19.0, 8.03) circle [radius=0.15];

  \draw[thick] (16.8, 5.9) -- (17.4, 8.25);
  \draw[black] (16.8, 5.9) circle [radius=0.15];
  \fill[black] (17.4, 8.25) circle [radius=0.15];

  \draw[thick] (15.5, 6.23) -- (16.0, 8.26);
  \draw[black] (15.5, 6.23) circle [radius=0.15];
  \fill[black] (16.0, 8.26) circle [radius=0.15];

  \draw[thick] (14.5, 6.5) -- (15.1, 8.18);
  \draw[black] (14.5, 6.5) circle [radius=0.15];
  \fill[black] (15.1, 8.18) circle [radius=0.15];

  \draw[thick] (14.0, 6.69) -- (14.7, 8.12);
  \draw[black] (14.0, 6.69) circle [radius=0.15];
  \fill[black] (14.7, 8.12) circle [radius=0.15];

  \draw[thick] (13.8, 6.78) -- (14.5, 8.09);
  \draw[black] (13.8, 6.78) circle [radius=0.15];
  \fill[black] (14.5, 8.09) circle [radius=0.15];

  \draw[thick] (13.5, 7.0) -- (14.2, 8.05);
  \draw[black] (13.5, 7.0) circle [radius=0.15];
  \fill[black] (14.2, 8.05) circle [radius=0.15];

  \draw[thick] (13.0, 7.41) -- (14.0, 8.05);
  \draw[black] (13.0, 7.41) circle [radius=0.15];
  \fill[black] (14.0, 8.05) circle [radius=0.15];

  \draw[thick] (12.8, 8.07) -- (14.0, 8.24);
  \draw[black] (12.8, 8.07) circle [radius=0.15];
  \fill[black] (14.0, 8.24) circle [radius=0.15];

  \draw[thick] (12.8, 8.93) -- (14.4, 8.7);
  \draw[black] (12.8, 8.93) circle [radius=0.15];
  \fill[black] (14.4, 8.7) circle [radius=0.15];

  \draw[thick] (13.0, 9.8) -- (14.9, 9.34);
  \draw[black] (13.0, 9.8) circle [radius=0.15];
  \fill[black] (14.9, 9.34) circle [radius=0.15];

  \draw[thick] (13.2, 10.4) -- (15.2, 9.85);
  \draw[black] (13.2, 10.4) circle [radius=0.15];
  \fill[black] (15.2, 9.85) circle [radius=0.15];

  \draw[thick] (13.3, 10.7) -- (15.4, 10.1);
  \draw[black] (13.3, 10.7) circle [radius=0.15];
  \fill[black] (15.4, 10.1) circle [radius=0.15];

  \draw[thick] (13.5, 11.4) -- (15.8, 10.8);
  \draw[black] (13.5, 11.4) circle [radius=0.15];
  \fill[black] (15.8, 10.8) circle [radius=0.15];

  \draw[thick] (13.8, 12.6) -- (16.3, 12.0);
  \draw[black] (13.8, 12.6) circle [radius=0.15];
  \fill[black] (16.3, 12.0) circle [radius=0.15];

  \draw[thick] (14.2, 14.2) -- (16.8, 13.7);
  \draw[black] (14.2, 14.2) circle [radius=0.15];
  \fill[black] (16.8, 13.7) circle [radius=0.15];

  \draw[thick] (14.4, 15.8) -- (17.1, 15.6);
  \draw[black] (14.4, 15.8) circle [radius=0.15];
  \fill[black] (17.1, 15.6) circle [radius=0.15];

  \draw[thick] (14.5, 17.2) -- (17.0, 17.2);
  \draw[black] (14.5, 17.2) circle [radius=0.15];
  \fill[black] (17.0, 17.2) circle [radius=0.15];

  \draw[thick] (14.4, 18.0) -- (16.9, 18.2);
  \draw[black] (14.4, 18.0) circle [radius=0.15];
  \fill[black] (16.9, 18.2) circle [radius=0.15];
  \end{tikzpicture}
  }

  \caption{Expansion centers $\vect{c}$ (full) and corresponding target
    points $\vect{x}$ (empty) with the corresponding expansion disk (gray).}
  \label{fig:prelim:qbx}
\end{figure}

The expansions are constructed in a neighborhood of the boundary
(see~\cite{Kloeckner2013} for details) for every target point $\vect{x} \in \Sigma$
(see~\Cref{fig:prelim:qbx}). For evaluation, the following steps are taken:
\begin{enumerate}
  \item Pick an expansion center $\vect{c}$ and an expansion radius $r$.
  In practice, the radius $r$ is chosen based on the local geometry by
  satisfying~\cite[Theorem 1]{Kloeckner2013}. Using the exterior normal vector
  $\vect{n}$, we can then write
  \[
  \vect{c} \coloneqq \vect{x} \pm r \vect{n}.
  \]
  \item Compute an expansion of the potential at $\vect{c}$ up to a prescribed order
  $p_{\text{qbx}}$.
  \item Evaluate the layer potential at $\vect{x}$, mediated through the
  expansion.
\end{enumerate}

\reva{
An explicit expression for the resulting local expansion can be found
in~\cite{Wala2020}. We use this expression to write out the elements of the
matrix operator $\matr{A}$ using the notation from~\eqref{eq:prelim:laplace_expansion} for the case of the single-layer
potential used in \Cref{sc:errors}. It is given by
\begin{equation} \label{eq:matrix_entries_slp}
A_{ij} = \tilde{G}(\vect{x}_i, \vect{y}_j) w_j \coloneqq
  \sum_{\ell = 0}^{p_{\text{qbx}}} \sum_{m = -\ell}^\ell
    L^m_\ell(\vect{y}_j)
    \|\vect{x}_i - \vect{c}_i\|^n
    Y^m_\ell(\widehat{\vect{x}_i - \vect{c}_i}),
\end{equation}
where $w_j$ are the quadrature weights and the \emph{local coefficients} $L^m_\ell$
are given by
\[
L^m_\ell(\vect{y}_j) = \frac{1}{2 \ell + 1}
  \frac{1}{\|\vect{y}_j - \vect{c}_i\|^{n + 1}} Y^{-m}_\ell(\widehat{\vect{y}_j - \vect{c}_i}) w_j.
\]
The above makes use of spherical harmonic expansions of the potential, which matches
our experiments in the three-dimensional case in \Cref{sc:results}.
For other use cases, QBX itself as well as the direct solver described here
are compatible with many other expansion types, including Fourier-Bessel and Cartesian
Taylor.
}

This is known as a ``global'' QBX method, since all interactions with the
target point $\vect{x}$ are mediated through the expansion. An analysis of
the method is presented in~\cite{Kloeckner2013} and extensions
to three dimensions are given in~\cite{Wala2019}. To obtain
accurate results at each discretization node, the newest variant of the
method employs adaptive algorithms over multiple grids obtained by selective refinement.
They are the \stageone{} grid, which is refined to eliminate errors from interfering
source geometry (see~\cite[Section 3.3]{Wala2019}), the \stagetwo{} grid, which is
refined to obtain sufficient quadrature resolution (see~\cite[Section 3.4]{Wala2019}),
and \stagequad{}, which optionally oversamples the quadrature of \stagetwo{}.

For a standard application of the QBX method, the target points $\vect{x}$ are
taken on the \stageone{} discretization and the source points $\vect{y}$ are
taken on the \stagequad{} discretization. This ensures that the solution is
obtained on a consistent mesh and that the layer potential is evaluated with
sufficient accuracy. However, the resulting discrete operator on these grids
is non-square and would require a least squares solver, such as~\cite{Ho2014}. In
\Cref{ssc:schemes:qbx} we discuss the construction of a square operator for the
QBX method and the necessary trade-offs.

\subsection{Interpolative Decomposition}
\label{ssc:prelim:ds}

The direct solver is based on a recursive compression (or ``skeletonization'')
of the off-diagonal blocks of the matrix $\matr{A}$ from~\eqref{eq:prelim:axb}. We
give here the basic notion that is required to construct the compressed
representation. First, the matrix entries are grouped into $N + 1$ contiguous
blocks of size $n_i$, such that $n = n_0 + \cdots + n_N$. For each block,
we construct index sets $\{I_i\}_{i = 0}^N$ and $\{J_i\}_{i = 0}^N$ for the rows
\reva{and} columns, respectively. An element $I_i \coloneqq (k_0, \dots, k_{n_i - 1})$
(or $J_i$) represents an ordered tuple of indices into the rows (or columns,
respectively) of the matrix $\matr{A}$. In this block form, the
system~\eqref{eq:prelim:axb} can be written as
\[
\sum_{j = 0}^N \matr{A}_{I_i, J_j} \vect{\sigma}_{J_j} = \vect{b}_{I_i},
\]
for $i \in \{0, \dots, N\}$, where, e.g., $\vect{b}_{I_i} = [b_{k_0}, \dots, b_{k_{n_i - 1}}]$
denotes the subset of the entries of $\vect{b}$ from the index tuple $I_i$. In the
following, we adopt the convention that the matrix blocks $\matr{A}_{ij} \equiv
\matr{A}_{I_i, J_j}, \vect{\sigma}_j \equiv \vect{\sigma}_{J_j}$ and $\vect{b}_i
\equiv \vect{b}_{I_i}$. To construct a direct solver for this system, we will
make use of the block separability of the $\matr{A}$ matrix, as defined below.

\begin{definition}[{Block Separability \cite[Section 3]{Gillman2012}}]
\label{def:prelim:separable}
A matrix $\matr{A}$ is said to be block separable if every off-diagonal block
$\matr{A}_{ij}$, with $i \ne j$, can be decomposed as a product of three low-rank
matrices
\begin{equation} \label{eq:prelim:separable}
\matr{A}_{ij} = \matr{L}_i \matr{S}_{ij} \matr{R}_j,
\end{equation}
where $\matr{L}_i \in \mathbb{R}^{n_i \times k_i}, \matr{S}_{ij} \in
\mathbb{R}^{k_i \times k_j}$ and $\matr{R}_j \in \mathbb{R}^{k_j \times n_j}$
for $k_i, k_j \ll \min(n_i, n_j)$.
\end{definition}

To obtain such a decomposition, we consider the Interpolative Decomposition (ID)
as a means of reconstructing $\matr{A}_{ij}$ as linear combinations
of outer products of a column subset in $\matr{L}_i$ and a row subset
in $\matr{R}_j$, with coefficients in the $\matr{S}_{ij}$ factors.
The existence of a single $\matr{L}_i$ and $\matr{R}_j$ for each entire block row and
entire block column of $\matr{A}$ is not given for other types of block structured matrices,
e.g.,\ $\mathcal{H}$ matrices, and is a central assumption for block separability.
The interpolative decomposition is defined in the following
(see~\cite{Cheng2005,Gillman2012,Martinsson2005} and the references therein).

\begin{theorem}[{Interpolative Decomposition \cite[Theorem 1]{Martinsson2005}}]
\label{thm:prelim:id}
Let $\matr{A} \in \mathbb{R}^{m \times n}$ be an arbitrary matrix and
$1 \le k < \min(m, n)$. Then, there exist a matrix $\matr{R} \in
\mathbb{R}^{k \times n}$ (not necessarily unique) and a matrix $\matr{S} \in
\mathbb{R}^{m \times k}$, which contains a subset of the columns of $\matr{A}$,
and a permutation matrix $\matr{P} \in \mathbb{R}^{n \times n}$, such that
\[
\|\matr{A} - \matr{S} \matr{R} \matr{P}\|_2 \le \sigma_{k + 1} \sqrt{1 + k (n - k)},
\]
where $|R_{ij}| \le 1$ and $\sigma_{k + 1}$ is the $(k + 1)$-th largest singular
value of $\matr{A}$.
\end{theorem}

Theoretical guarantees on the use of the ID to directly obtain a factorization of the
form~\eqref{eq:prelim:separable} can be found in~\cite[Theorem 3]{Cheng2005}.
However, in practice, we construct matrices that approximate the row and
column space of $\matr{A}_{ij}$ and apply the ID separately through~\Cref{thm:prelim:id}.
This construction is described in \Cref{ssc:schemes:proxy}.

\section{Schemes for Coupling QBX with a Direct Solver}
\label{sc:schemes}

We briefly review the construction of a direct solver by taking advantage of the
block separability of $\matr{A}$. In particular, if $\matr{A}$ is block separable
in the sense of \Cref{def:prelim:separable}, we can write
\begin{equation} \label{eq:schemes:skeletonization}
  \matr{A} \approx \matr{A}_\epsilon \coloneqq \matr{D} + \matr{L} \matr{S} \matr{R},
\end{equation}
where $\matr{D}, \matr{L}$ and $\matr{R}$ are block diagonal matrices corresponding
to the diagonal of $\matr{A}$ and the interpolation matrices constructed from
the ID, as described in \Cref{thm:prelim:id}, i.e.,
\[
\matr{D} =
\begin{bmatrix}
\matr{A}_{00} & & \\
& \ddots & \\
& & \matr{A}_{NN}
\end{bmatrix}, \quad
\matr{L} =
\begin{bmatrix}
\matr{L}_{0} & & \\
& \ddots & \\
& & \matr{L}_{N}
\end{bmatrix}, \quad
\matr{R} =
\begin{bmatrix}
\matr{R}_{0} & & \\
& \ddots & \\
& & \matr{R}_{N}
\end{bmatrix}.
\]

The ``skeleton'' matrix $\matr{S} \in \mathbb{R}^{k \times k}$, where
$k \coloneqq k_0 + \cdots + k_N$, has zero diagonal blocks and the off-diagonal
elements consist of entries of the original dense matrix $\matr{A}$,
i.e.,
\[
\matr{S} =
\begin{bmatrix}
0 & \matr{S}_{01} & \cdots & \matr{S}_{0N} \\
\matr{S}_{10} & 0 & \cdots & \matr{S}_{1N} \\
\vdots & \vdots & \ddots & \vdots \\
\matr{S}_{N0} & \matr{S}_{N1} & \cdots & 0
\end{bmatrix},
\]
where $\matr{S}_{ij}$ is a subset of $\matr{A}_{ij}$ given
by~\Cref{thm:prelim:id}. Furthermore, $\matr{S}$ is also block separable,
which allows constructing a bottom-up recursive decomposition, as described
in~\Cref{ssc:schemes:skeletonization} and shown in~\Cref{fig:schemes:recursion}.
Using the above decomposition, we can construct a \reva{fast} approximate direct
solver (see~\cite{Gillman2012}). We introduce the extended system
\[
\begin{bmatrix}
    \matr{D} & \matr{L} \matr{S} \\
    -\matr{R} & \matr{I}
\end{bmatrix}
\begin{bmatrix}
    \vect{\sigma} \\
    \hat{\vect{\sigma}}
\end{bmatrix}
=
\begin{bmatrix}
    \vect{b} \\
    \vect{0}
\end{bmatrix}
\]
and applying a Schur complement to solve the smaller system for $\hat{\vect{\sigma}}$
instead. Multiplying the first row by $\matr{R} \matr{D}^{-1}$ gives the following
modified system
\[
\begin{bmatrix}
    \matr{D} & \matr{L} \matr{S} \\
    \matr{R} & \hat{\matr{D}}^{-1} \matr{S} \\
    -\matr{R} & \matr{I}
\end{bmatrix}
\begin{bmatrix}
    \vect{\sigma} \\
    \hat{\vect{\sigma}}
\end{bmatrix}
=
\begin{bmatrix}
    \vect{b} \\
    \hat{\vect{b}} \\
    \vect{0}
\end{bmatrix},
\]
where $\hat{\vect{b}} \coloneqq \matr{R} \matr{D}^{-1} \vect{b}$ and
$\hat{\matr{D}} \coloneqq (\matr{R} \matr{D}^{-1} \matr{L})^{-1}$. Note that
$\hat{\matr{D}} \in \mathbb{R}^{k \times k}$ is also a block diagonal matrix
with blocks $\hat{\matr{D}}_{ii} \coloneqq (\matr{R}_i \matr{D}_{ii}^{-1}
\matr{L}_i)^{-1} \in \mathbb{R}^{k_i \times k_i}$. Then, by adding the second
and the third rows, we can solve the (dense) reduced system
\begin{equation} \label{eq:schemes:reduced}
(\hat{\matr{D}} + \matr{S}) \hat{\vect{\sigma}} = \hat{\matr{D}} \hat{\vect{b}}
\end{equation}
and obtain the solution to the original system from
\begin{equation} \label{eq:schemes:original}
\matr{D} \vect{\sigma} = \vect{b} - \matr{L} \matr{S} \hat{\vect{\sigma}},
\end{equation}
which only requires the solution of a block diagonal linear system. This
compressed solver (summarized in \Cref{alg:schemes:single_level}) can be shown
to have an approximate cost of~\cite[Remark 3.1]{Gillman2012}
\begin{equation} \label{eq:schemes:single_level_complexity}
T = c_{\text{offline}} + c_{\text{solve}} + c_{\text{apply}} =
  O\left(\frac{n^2}{N^2} k + \frac{n^3}{N^2}\right)
  + O\left(k^3\right)
  + O\left(\frac{n^2}{N} + k^2\right),
\end{equation}
where $c_{\text{offline}}$ is the preprocessing cost of constructing $\matr{L},
\matr{R}$ and $\hat{\matr{D}}$ (using the methods from~\Cref{ssc:schemes:proxy}),
$c_{\text{solve}}$ is the cost of solving the reduced system~\eqref{eq:schemes:reduced},
and $c_{\text{apply}}$ is the cost of applying the compressed matrices
in~\eqref{eq:schemes:original}. We have assumed here
that the ID for each block can be computed in $O(n_i^2 k_i)$ as described
in~\cite[Remark 5]{Cheng2005}. If solved to a given tolerance $\epsilon$,
the compressed approximate direct solver is asymptotically faster than an exact
LU-based direct solver.

\begin{center}
\vspace{-0.75cm}
\begin{minipage}{0.7\linewidth}
\begin{algorithm}[H]
\caption{\sc Single-Level Compressed Inverse}
\label{alg:schemes:single_level}
\begin{algorithmic}[1]
\For{$i = 0, \dots N$}
  \State $\hat{\vect{b}}_i \gets
    \hat{\matr{D}}_{ii} \matr{R}_i \matr{D}_{ii}^{-1} \vect{b}_i$
\EndFor
\State $\hat{\vect{\sigma}} \gets
  \operatorname{Solve}(\hat{\matr{D}} + \matr{S}, \hat{\matr{D}} \hat{\vect{b}})$
  \Comment{Solve \eqref{eq:schemes:reduced}}
\For{$i, j = 0, \dots, N, i \neq j$}
  \State $\bar{\vect{\sigma}}_i \gets
    \bar{\vect{\sigma}}_i + \matr{S}_{ij} \hat{\vect{\sigma}}_j$
    \Comment{Prepare for \eqref{eq:schemes:original}}
\EndFor
\For{$i = 0, \dots, N$}
  \State $\vect{\sigma}_i \gets
    \matr{D}_{ii}^{-1}(\vect{b}_i + \matr{L}_i \bar{\vect{\sigma}}_i)$
    \Comment{Solve \eqref{eq:schemes:original}}
\EndFor
\end{algorithmic}
\end{algorithm}
\end{minipage}
\end{center}

\subsection{Cluster Construction}
\label{ssc:schemes:clusters}

To construct the single-level compression scheme described in the previous section,
we first require a method to construct the index tuples $\{I_i\}_{i = 0}^N$
and $\{J_i\}_{i = 0}^N$. The index tuples are typically selected based on geometric
proximity. We use a quadtree and an octree in two and three dimensions, respectively,
to obtain such subsets. This is a convenient choice that straightforwardly generalizes
to higher dimensions. However, binary trees have also been
used successfully in two dimensions~\cite{Gillman2012}. In the discretization
of~\eqref{eq:prelim:lp}, we consider the target points $X$ and the source points
$Y$, which correspond to the discretization of the geometry $\Sigma$ according
to~\Cref{ssc:prelim:qbx}. From the index tuples, we obtain a partition
$\{X_i\}_{i = 0}^N$ and $\{Y_i\}_{i = 0}^N$ of all the nodes. Via a renumbering,
the index tuples $\{I_i\}$ and $\{J_i\}$ are chosen to be contiguous in order
to obtain the block form presented in \Cref{ssc:prelim:ds}.

\begin{remark}
As discussed in~\Cref{ssc:schemes:qbx}, the sets $X$ and $Y$ are taken to be
the same to avoid a non-square matrix $\matr{A}$. Therefore, the initial partition
of the geometry is only performed once.
\end{remark}

The construction is similar to~\cite{Greengard2009} and \cite{Ho2012}. We bin
sort the centroids of each element in the geometry using the corresponding tree.
The tree is refined until a maximum predefined number of centroids is found in
each leaf. We require that adjacent nodes in the tree have
side lengths that differ by at most a factor of $2$ (known as 2:1 balancing),
to ensure that neighboring clusters are not too far apart in the tree hierarchy.
Each index tuple $I_i$ (or $J_i$) is then constructed from a numbering of the
nodes of the elements whose centroids are contained in a leaf. As such, the
number of index tuples $N + 1$ corresponds to the number of leaves in the
quadtree (or octree).

\subsection{Proxy Skeletonization}
\label{ssc:schemes:proxy}

The remaining requirement for performing the single-level compression from the
previous section is a procedure to construct the interpolation matrices
$\matr{L}$ and $\matr{R}$. We follow here the construction based on
proxy skeletonization shown in~\cite{Ho2012}. The basic idea behind this procedure
can be performed on a block-by-block basis as described below.

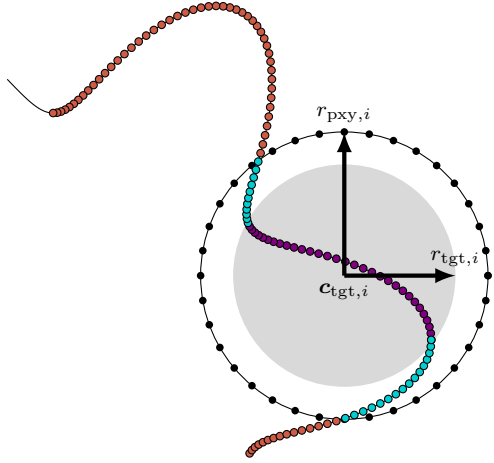
\begin{figure}[ht!]
\centering
\begin{tikzpicture}
\coordinate (Center) at (3.85, 0.85);
\fill[white] (Center) circle [radius=1.9];
\fill[gray!30] (Center) circle [radius=1.47];
\draw (Center) circle [radius=1.9];

\draw[domain=-30:100,samples=256] plot (
    {(4+cos(6*\x)) * cos(\x)},
    {(4+cos(6*\x)) * sin(\x)});

\foreach \x in {0, 10, ..., 360} {
    \fill[qProxy] ({3.85 + 1.9*cos(\x)}, {0.85 + 1.9*sin(\x)}) circle [radius=0.05];
}
\foreach \x in {0, 1, ..., 30} {
    \draw[fill=qCluster] (
    {(4+cos(6*\x)) * cos(\x)},
    {(4+cos(6*\x)) * sin(\x)}) circle [radius=0.05];
}
\foreach \x in {31, 32, ..., 41} {
    \draw[fill=qNear] (
    {(4+cos(6*\x)) * cos(\x)},
    {(4+cos(6*\x)) * sin(\x)}) circle [radius=0.05];
}
\foreach \x in {-15, -14, ..., 0} {
    \draw[fill=qNear] (
    {(4+cos(6*\x)) * cos(\x)},
    {(4+cos(6*\x)) * sin(\x)}) circle [radius=0.05];
}
\foreach \x in {42, 43, ..., 90} {
    \draw[fill=qFar] (
    {(4+cos(6*\x)) * cos(\x)},
    {(4+cos(6*\x)) * sin(\x)}) circle [radius=0.05];
}

\foreach \x in {-30, -29, ..., -16} {
    \draw[fill=qFar] (
    {(4+cos(6*\x)) * cos(\x)},
    {(4+cos(6*\x)) * sin(\x)}) circle [radius=0.05];
}

\draw[ultra thick,-latex] (Center) -- ++(right:1.47) node [above] {$\rtgt$};
\draw[ultra thick,-latex] (Center) -- ++(up:1.9) node [above] {$\rpxy$};
\node[below] at (Center) {$\ctgt$};
\end{tikzpicture}
\caption{Target points $X_i$ (\textcolor{qCluster}{purple}) inside the gray disk
  of radius $\rtgt$ centered at $\ctgt$, near source point
  $Y^{(\text{near})}_i$ (\textcolor{qNear}{cyan}) inside the proxy ball of radius
  $\rpxy$, far source points $Y^{(\text{far})}_i$ (\textcolor{qFar}{red}) outside
  the proxy ball, and proxy points $P_i$ (\textcolor{qProxy}{black}) on the circle.}
\label{fig:schemes:proxy}
\end{figure}

Locally, the geometry we are interested in is depicted in~\Cref{fig:schemes:proxy}.
For simplicity, we focus the initial discussion on the construction of the
$\matr{L}_i$ interpolation matrix, while to corresponding construction for
$\matr{R}_i$ is defined, in a sense, dual to it. We refer to the construction of
$\matr{L}_i$ and $\matr{R}_i$ as ``target skeletonization'' and ``source skeletonization'',
respectively. We consider a cluster of $n_i$ target points $X_i$, corresponding
to the row tuple $I_i$ (constructed as described in~\Cref{ssc:schemes:clusters}),
with a centroid $\ctgt$ and radius $\rtgt \coloneqq \max_{\vect{x} \in X_i}
\|\ctgt - \vect{x}\|_2$. The set of $q$ proxy points $\{\vect{p}_i \mid 0 \le k < q\}$
have no canonical construction. A common choice is a sphere around the cluster $X_i$, i.e.,
\begin{equation} \label{eq:schemes:proxy_surface}
\Sigma_{\text{pxy}, i} \coloneqq \{
\vect{p} \in \mathbb{R}^d \mid \|\vect{p} - \ctgt\|_2 = \rpxy\},
\end{equation}
where $\rpxy$ is a desired proxy radius. In practice, we choose $\rpxy \coloneqq
\alpha\, \rtgt$, for a constant $\alpha \ge 1$. This choice is motivated by the
fact that the Green functions used in the boundary integral equations are also
radially symmetric. However, for more general scenarios, an optimized proxy point
construction can be found in~\cite{Xing2020a,Ye2020}. In~\Cref{sc:errors}, we
provide an error analysis that aids in the choice of proxy points for the
spherical case.

The source points are split into a near-field and a far-field as
\begin{equation} \label{eq:schemes:proxy_near_far}
\begin{aligned}
Y^{(\text{near})}_i & \coloneqq
\{\vect{y} \in Y \setminus Y_i \mid \|\vect{y} - \ctgt\|_2 \le \rpxy\}, \\
Y^{(\text{far})}_i & \coloneqq Y \setminus (Y^{(\text{near})}_i \cup Y_i).
\end{aligned}
\end{equation}

\begin{remark}
As shown below, we are only interested in $Y^{(\text{near})}_i$ in the construction
of $\matr{L}_i$. This set can be determined efficiently with area queries based on
the existing quadtree or octree used in~\Cref{ssc:schemes:clusters}
(see \cite[Appendix A]{Wala2019}).
\end{remark}

For target skeletonization, the goal of this construction is to approximate the
interaction between the current point cluster $X_i$ and $Y^{(\textrm{far})}_i$.
Therefore, we seek a matrix $\matr{B}_i$ that captures the range of the matrix
block, i.e.,\ approximately,
\[
\operatorname{Range}(\matr{A}_{ij}) \approx \operatorname{Range}(\matr{B}_i),
\]
for all $i \ne j$. In general, it is not possible to ensure that the range of
$\matr{A}_{ij}$ will be a subset of the range of $\matr{B}_i$. In the following,
we construct the matrix $\matr{B}_i$ by considering the interactions between
the current cluster $X_i$ and the union of the proxy points $P_i$ and the
near-field $Y^{(\text{near})}_i$ (see~\Cref{fig:schemes:proxy}). An error
estimate for this approximation is provided in~\Cref{sc:errors}, by means of a
multipole expansion.

We introduce the notation $\matr{A}_{ij} \equiv
\matr{A}(X_i, Y_j)$, i.e.\ the component $(k, l)$ of the matrix block
$\vect{A}_{ij}$ is given by the interaction between the target point
$\vect{x}_k \in X_i$ and the source points $\vect{y}_l \in Y_j$. Using this
notation, we can write
\[
\matr{A}(X_i, Y_j) = \matr{K}(X_i, Y_j) \matr{W}(Y_j),
\]
where $\matr{K}(X_i, Y_j)$ is the (dense) matrix of kernel interactions (mediated
through the QBX local expansions) and $\matr{W}(Y_j)$ is a diagonal matrix comprised
of the quadrature weights and area elements on the source geometry at $Y_j$.
\reva{As an illustration, for the specific case of the single-layer kernel, a concrete
expression for $\matr{K}(X_i, Y_j)$ can be found in \eqref{eq:matrix_entries_slp}.}
We aim to construct proxy-based approximations that maintain a similar structure to
the operators.

\begin{method}[Target Skeletonization] \label{mm:schemes:target_skeletonization}
For target skeletonization, we construct the proxy matrix as follows
\begin{equation} \label{eq:schemes:proxy_target_matrix}
\matr{B}_i \coloneqq
\begin{bmatrix}
    \matr{G}(X_i, P_i) \matr{W}(P_i) &
    \matr{G}(X_i, Y^{(\text{near})}_i) \matr{W}(Y^{(\text{near})}_i)
\end{bmatrix}
\in \mathbb{R}^{n_i \times (p_i + n_i^{(\text{near})})}
\end{equation}
where $\matr{G}(X_i, P_i)$ are the target-proxy interactions with the underlying
Green's function for $\matr{K}$ and $\matr{W}(Y^{(\text{near})}_i)$ represent the
quadrature weights at the near-field points. The $\matr{W}(P_i)$ weight matrix
is chosen as a constant diagonal matrix
\[
\matr{W}(P_i) \coloneqq \|\matr{W}(Y^{(\text{near})}_i)\|_2 \matr{I}.
\]

The matrix $\matr{B}_i$ is then decomposed using~\Cref{thm:prelim:id} to obtain
\[
\matr{B}_i \approx \matr{L}_i
\begin{bmatrix}
  \matr{G}(X^{(\text{s})}_i, P_i) \matr{W}(P_i) &
  \matr{G}(X^{(\text{s})}_i, Y^{(\text{near})}_i) \matr{W}(Y^{(\text{near})}_i)
\end{bmatrix},
\]
where $X^{(\text{s})}_i \subset X_i$ are the remaining target skeleton points
(as chosen by the ID).
\end{method}

To improve the range approximation between the near-field and the far-field, we
have introduced the weight matrix $\matr{W}(P_i) \in \mathbb{R}^{n_i \times p_i}$
above. Its purpose is to ensure that the two blocks are weighted equally in a chosen
norm (here the spectral norm). This addition appears to not currently exist in the literature
and is not rigorously motivated, but it can lead to noticeable improvements in practice, as
shown in~\Cref{sssc:result:model:target_weight_matrix}.

\begin{method}[Source Skeletonization] \label{mm:schemes:source_skeletonization}
For source skeletonization, we construct the proxy matrix as follows
\begin{equation} \label{eq:schemes:proxy_source_matrix}
\matr{B}_j \coloneqq
\begin{bmatrix}
\matr{K}(P_j, Y_j) \matr{W}(Y_j) \\
\matr{K}(X^{(\text{near})}_j, Y_j) \matr{W}(Y_j)
\end{bmatrix}
\in \mathbb{R}^{(p_j + n_j^{(\text{near})}) \times n_j},
\end{equation}
where $\matr{K}(P_j, Y_j)$ represents the proxy-source interactions of the layer
potential kernel. The proxy points $P_j$ and the near-field target points
$X^{(\text{near})}_j$ are constructed analogously to~\reva{\eqref{eq:schemes:proxy_target_matrix}}.
The matrix $\matr{B}_j$ is then decomposed using~\Cref{thm:prelim:id} to obtain
\[
\matr{B}_j \approx
\begin{bmatrix}
\matr{K}(P_j, Y^{(\text{s})}_j) \matr{W}(Y^{(\text{s})}_j) \\
\matr{K}(X^{(\text{near})}_j, Y^{(\text{s})}_j) \matr{W}(Y^{(\text{s})}_j)
\end{bmatrix} \matr{R}_j,
\]
where $Y^{(\text{s})}_j \subset Y_j$ are the remaining source skeleton points
(as chosen by the ID).
\end{method}

The procedures above are repeated for each block. The choices made for the construction
of the proxy approximation matrix $\matr{B}$ (for both source and target
skeletonization) will be further discussed in the context of the error model
in~\Cref{sc:errors}.

\subsection{Recursive Skeletonization}
\label{ssc:schemes:skeletonization}

The decomposition from~\eqref{eq:schemes:skeletonization} can be performed
recursively. For this, we will introduce additional notation. First, we denote
the index tuples at each level of the recursion corresponding to the target
and source points by $\{I^{(\ell)}_i\}$ and $\{J^{(\ell)}_i\}$, where
$\ell \in \{0, \dots, N_{\text{levels}} - 1\}$ corresponds to the tree level,
where $0$ refers to the leaf level and $N_{\text{levels}} - 1$
refers to the root level. At the leaf level, $\{I^{(0)}_i\}$ and $\{J^{(0)}_i\}$
coincide with the index sets defined in~\Cref{ssc:prelim:ds} and constructed in
\Cref{ssc:schemes:clusters} for the single-level skeletonization. Given the
index sets at level $\ell$, we obtain the index sets at level $\ell + 1$ by
concatenating all the index tuples that have the same parent in the tree. This
ensures that the clusters at all levels maintain their geometric proximity
and that the clustered blocks in the matrix $\matr{A}$ remain adjacent.

Then, at each level $\ell$ we construct interpolation matrices $\matr{L}^{(\ell)}$
and $\matr{R}^{(\ell)}$ for the corresponding index tuples $\{I_i^{(\ell)}\}$ and
$\{J_i^{(\ell)}\}$, respectively, using the methods described
in~\Cref{ssc:schemes:proxy}. This gives rise to the telescoping factorization
\begin{equation} \label{eq:schemes:rec_skeletonization}
\matr{A} \approx \matr{D}^{(0)} + \matr{L}^{(0)} \Big(
  \matr{D}^{(1)} + \matr{L}^{(1)} (\cdots
  \matr{D}^{(\ell)} + \matr{L}^{(\ell)} (\cdots) \matr{R}^{(\ell)}
  \cdots) \matr{R}^{(1)}
\Big) \matr{R}^{(0)},
\end{equation}
where the root level just contains the coarsest $\matr{S}^{(N_{\text{levels}} - 1)}$
matrix. By definition, at the coarsest level we only have one index pair
$(I_0^{(N_{\text{levels}} - 1)}, J_0^{(N_{\text{levels} - 1})})$ that
describes the matrix $\matr{S}$ as a subset of $\matr{A}$ \reva{with a zero diagonal,
as in~\Cref{sc:schemes} (see also~\Cref{fig:schemes:recursion})}.

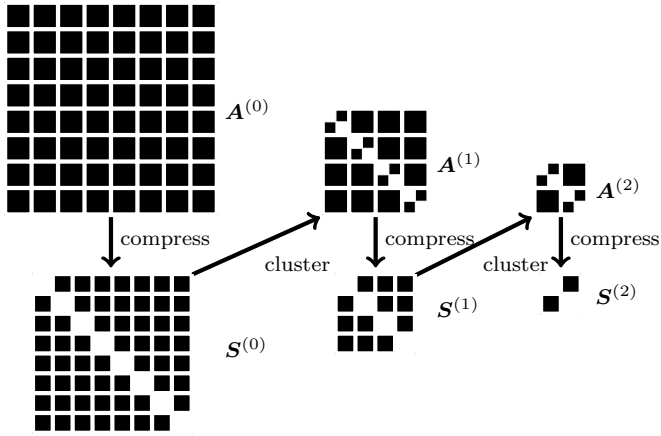
\begin{figure}[ht!]
\centering
\begin{tikzpicture}[scale=0.35]
\draw [fill=black] (0, 0) rectangle (8, 8);
\draw [ultra thick, white] (0, 0) grid (8, 8);
\draw [ultra thick, ->] (4, -0.1) -- node [right] {compress} (4, -2);
\draw [ultra thick, ->] (7, -2.3) -- node [below right] {cluster} (11.9, -0.1);
\node at (8, 4) [right] {$\matr{A}^{(0)}$};
\node at (8, -5) [right] {$\matr{S}^{(0)}$};

\draw [fill=black] (12, 0) rectangle (16, 4);
\draw [ultra thick, white] (12, 0) grid (16, 4);
\draw [ultra thick, ->] (14, -0.1) -- node [right] {compress} (14, -2);
\draw [ultra thick, ->] (15.5, -2.3) -- node [below right] {cluster} (19.9, -0.1);
\foreach \x in {12,12.5,...,15.5}{
    \draw [fill=white, white] (\x, 16 - \x) rectangle (\x + 0.5, 15.5 - \x);
}
\node at (16, 2) [right] {$\matr{A}^{(1)}$};
\node at (16, -3.5) [right] {$\matr{S}^{(1)}$};

\draw [fill=black] (20, 0) rectangle (22, 2);
\draw [ultra thick, white] (20, 0) grid (22, 2);
\draw [ultra thick, ->] (21, -0.1) -- node [right] {compress} (21, -2);
\foreach \x in {20,20.5,...,21.5}{
    \draw [fill=white, white] (\x, 22 - \x) rectangle (\x + 0.5, 21.5 - \x);
}
\node at (22, 1) [right] {$\matr{A}^{(2)}$};
\node at (22, -3) [right] {$\matr{S}^{(2)}$};

\begin{scope}[scale=0.75,xshift=-17pt]
\draw [fill=black] (2, -11) rectangle (10, -3);
\draw [ultra thick, white] (2, -11) grid (10, -3);
\foreach \x in {2,...,9}{
    \draw [fill=white, white] (\x, -2 - \x) rectangle (\x + 1, -1 - \x);
}
\end{scope}

\begin{scope}[scale=0.75,xshift=-10pt]
\draw [fill=black] (17, -7) rectangle (21, -3);
\draw [ultra thick, white] (17, -7) grid (21, -3);
\foreach \x in {17,...,20}{
    \draw [fill=white, white] (\x, 14 - \x) rectangle (\x + 1, 13 - \x);
}
\end{scope}

\begin{scope}[scale=0.75, xshift=0pt]
\draw [fill=black] (27, -5) rectangle (29, -3);
\draw [ultra thick, white] (27, -5) grid (29, -3);
\foreach \x in {27, ..., 28}{
    \draw [fill=white, white] (\x, 24 - \x) rectangle (\x + 1, 23 - \x);
}
\end{scope}
\end{tikzpicture}
\caption{Multilevel compression and clustering. Black blocks are unchanged and
  low-rank, while white blocks are full rank and correspond to the diagonals
  $\matr{D}^{(\ell)}$ in~\eqref{eq:schemes:rec_skeletonization} (reproduced
  from Figure~2 in \cite{Martinsson2005}).}
\label{fig:schemes:recursion}
\end{figure}

A fully recursive approximate direct solver can be constructed using the steps
provided in \Cref{alg:schemes:single_level} and the hierarchy provided by the
quadtree (or octree) described in~\Cref{ssc:schemes:clusters}. Abstractly, it
consists of an offline step that constructs the $(\matr{D}^{(\ell)},
\matr{L}^{(\ell)}, \matr{R}^{(\ell)})$ matrices at each level and an online
step that solves the approximate system for a given right-hand side $\vect{b}$.
As shown in \Cref{fig:schemes:recursion}, we also define the series of matrices
$\matr{A}^{(\ell)}$ and $\matr{S}^{(\ell)}$ representing the input matrix and
the skeleton matrix at each level. By construction, each \reva{off-diagonal}
entry of $\matr{A}^{(\ell)}$ is given by a re-indexing or the original matrix
$\matr{A}_{I_i^{(\ell)}, J_j^{(\ell)}}$, \reva{while the diagonal entries are
clustered from the previous level and have a zero sub-diagonal}. This offline
construction step is detailed in~\Cref{alg:schemes:skeletonization}. \reva{We
note that the matrices $\matr{S}^{(\ell)}$ and $\matr{A}^{(\ell)}$ are never
constructed explicitly, and only the entries required for
\Cref{mm:schemes:target_skeletonization},
\Cref{mm:schemes:source_skeletonization}, and computing $\matr{D}^{(\ell)}$ are
evaluated.}

\begin{center}
\vspace{-0.7cm}
\begin{minipage}{0.75\linewidth}
\begin{algorithm}[H]
\caption{\sc Proxy-Based Recursive Compression}
\label{alg:schemes:skeletonization}
\textbf{Inputs}: Index tuples $\{I^{(0)}_i\}$ and $\{J^{(0)}_i\}$ \\
\textbf{Outputs}: Matrices $\matr{L}^{(\ell)}, \matr{D}^{(\ell)}, \matr{R}^{(\ell)},
  \matr{S}^{(N_{\text{levels} - 1})}$.
\begin{algorithmic}[1]
  \For{$\ell = 0, \dots, N_{\text{levels}} - 1$}
    \State $\matr{D}^{(\ell)} \gets \operatorname{block-diag} (\matr{A}^{(\ell)})$\;
    \State $(\hat{I}^{(\ell + 1)}, \matr{L}^{(\ell)})$ are computed using
      \Cref{mm:schemes:target_skeletonization} for $I^{(\ell)}$ and $\matr{A}^{(\ell)}$\;
    \State $(\hat{J}^{(\ell + 1)}, \matr{R}^{(\ell)})$ are computed using
      \Cref{mm:schemes:source_skeletonization} for $J^{(\ell)}$ and $\matr{A}^{(\ell)}$\;
    \State Compute $(I^{(\ell + 1)}, J^{(\ell + 1)}, \matr{A}^{(\ell + 1)})$
      by clustering $(\hat{I}^{\ell + 1}, \hat{J}^{(\ell + 1)}, \matr{S}^{(\ell)})$
      from child to parent based on the quadtree (octree) hierarchy\;
  \EndFor
\end{algorithmic}
\end{algorithm}
\end{minipage}
\end{center}

The offline step should only be performed once and can be reused to solve for
multiple right-hand sides $\vect{b}$. Following the offline step, we also define
the ``compressed'' right-hand sides $\vect{b}^{(\ell)}$ and the solution vectors
$\vect{x}^{(\ell)}$ at each level based on re-indexing the leaf level
$\vect{b}^{(0)} \equiv \vect{b}$ and $\vect{x}^{(0)} \equiv \vect{x}$ with the
corresponding index tuples. Then, the online step from~\Cref{alg:schemes:multi_level}
can take the following the steps
\begin{enumerate}
  \item ``Compress'' the right-hand side vector $\vect{b}$ to the root level,
  using the $\vect{R}_i^{(\ell)}$ matrices to obtain $\vect{b}^{(\ell + 1)}$
  from $\vect{b}^{\ell}$.
  \item Solve a small dense linear system with the root matrix
  $\matr{S}^{(N_{\text{levels} - 1})}$ and $\vect{b}^{(N_{\text{levels}} - 1)}$.
  \item ``Uncompress'' the solution vector $\vect{x}$ to the leaf level, using
  the $\matr{L}_i^{(\ell)}$ matrices to obtain $\vect{x}^{(\ell - 1)}$
  from $\vect{x}^{(\ell)}$.
\end{enumerate}

\begin{center}
\vspace{-0.7cm}
\begin{minipage}{0.7\linewidth}
\begin{algorithm}[H]
\caption{\sc Multi-Level Solve}
\label{alg:schemes:multi_level}
\textbf{Inputs}: Matrices $\matr{L}^{(\ell)}, \matr{D}^{(\ell)}, \matr{R}^{(\ell)},
  \matr{S}^{(N_{\text{levels} - 1})}$ and $\vect{b}$ \\
\textbf{Inputs}: Index tuples $\{I^{(\ell)}_i\}$ and $\{J^{(\ell)}_i\}$ \\
\textbf{Outputs}: Approximate solution $\vect{x}_\epsilon$
\begin{algorithmic}[1]
  \State $\vect{b}^{(0)} \gets \vect{b}$\;
  \For{$\ell = 0, \dots, N_{\text{levels}} - 2$}
    \For{$i \in 0, \dots, N^{(\ell)}$}
      \State $\hat{\vect{b}}_i^{(\ell)} \gets
        \hat{\matr{D}}_{ii}^{(\ell)} \matr{R}_{ii}^{(\ell)}
        \left(\matr{D}_{ii}^{(\ell)}\right)^{-1} \vect{b}_i^{(\ell)}$\;
    \EndFor
      \State $\vect{b}^{(\ell + 1)} \gets \hat{\vect{b}}^{(\ell)}$\;
  \EndFor

  \State $\vect{x}^{(N_{\text{levels} - 1})} \gets
    (\matr{S}^{(N_{\text{levels}} - 1)})^{-1} \vect{b}^{(N_{\text{levels} - 1})}$\;

  \For{$\ell = N_{\text{levels}} - 2, \dots, 0$}
    \State $\hat{\vect{x}}^{(\ell)} \gets \vect{x}^{(\ell + 1)}$\;
    \For{$i = 0, \dots, N^{(\ell)}$}
      \State $\vect{x}^{(\ell)}_i \gets
        \left(\matr{D}_{ii}^{(\ell)}\right)^{-1} \left(
        \vect{b}_i^{(\ell)} - \matr{L}_{ii}^{(\ell)} \hat{\vect{b}}_i^{(\ell)}
        + \matr{L}_{ii}^{(\ell)} \hat{\matr{D}}_{ii}^{(\ell)} \hat{\vect{x}}_i^{(\ell)}
        \right)$\;
    \EndFor
  \EndFor
\end{algorithmic}
\end{algorithm}
\end{minipage}
\end{center}

\revb{For the steps described in~\Cref{alg:schemes:skeletonization} and
\Cref{alg:schemes:multi_level}, \reva{the} scheme has an asymptotic cost
of~\cite{Gillman2012,Ho2012} (extended from~\eqref{eq:schemes:single_level_complexity})
\begin{equation} \label{eq:schemes:scaling}
\begin{aligned}
T^{(\text{2D})} & =
  c_{\text{offline}} + c_{\text{solve}} =
  O(n) + O(n), \\
T^{(\text{3D})} & =
  c_{\text{offline}} + c_{\text{solve}} =
  O(n^{\frac{3}{2}}) + O(n \log n),
\end{aligned}
\end{equation}
}as $n \to \infty$, under appropriate rank assumptions in~\Cref{thm:prelim:id}.
The three-dimensional solver can also be made into a solver with linear time complexity
by recursively compressing the interpolation matrices, as discussed in~\cite{Gillman2012}.
This is mainly an algorithmic improvement, which is not incorporated in the current
study.

\subsection{QBX Direct Solver}
\label{ssc:schemes:qbx}

Coupling the recursive direct solver construction from~\Cref{ssc:schemes:skeletonization}
with the adaptive QBX method described in~\Cref{ssc:prelim:qbx} requires special
care. We address the choices made for this work in the following.

\emph{Choice of discretization for constructing the matrix $\matr{A}$}. In the
context of the direct solver, the choice of density (source) and result (target)
discretizations is subject to additional considerations. First, to avoid a
potentially expensive refactorization of the HSS matrix, we choose to avoid
the inclusion of interpolation steps. In addition, from a linear-algebraic
point of view, it is convenient for the system matrix to be square (see~\cite{Ho2014}).
This leads to the choice of a single discretization that serves as both the target and
the source discretization. In view of \Cref{ssc:prelim:qbx}, we use a version of
the \stagetwo{} (i.e.,\ interpolatory) discretization that provides accurate
resolution to meet the accuracy goals for integration of the kernel given
in~\cite{Kloeckner2013}.

\emph{Impact of the QBX method on proxy ball selection}. The proxy balls
$\Sigma_{\text{pxy}, i}$ must be selected in such a way that the accuracy of
the local expansions used in the QBX method is not affected. To address this,
we ensure that the proxy radius $\rpxy$ is always larger than the cluster radius
containing the QBX expansion balls. This is achieved by adding the maximum local
QBX expansion radius to the cluster radius $\rtgt$. Taking a larger $\rpxy$ also
presents a trade-off in efficiency, as it increases the size of the
$Y^{(\text{near})}_i$ set used in the construction of the proxy matrix
$\matr{B}_i$ from~\Cref{ssc:schemes:proxy}. However, it also results in better
accuracy, as discussed in~\Cref{sc:errors}.

\emph{Impact of the QBX method in proxy matrix construction}.
\reva{Due to the use of the QBX method, the construction of the proxy matrices
is asymmetric between the target and source proxy skeletonization, as shown
in \eqref{eq:schemes:proxy_target_matrix} and \eqref{eq:schemes:proxy_source_matrix}.
Specifically, the source skeletonization makes use of the $\matr{K}(P_j, Y_j)$
entries, i.e. the QBX regularized kernel, when considering the neighbor and proxy
interactions, while the target skeletonization makes use of the unregularized
kernel $\matr{G}(X_i, P_i)$ entries. We have found that, empirically, this
modification consistently decreases the error.}

\section{Estimates of Skeletonization Error}
\label{sc:errors}

We are now in a position to discuss error bounds for a direct solver
for boundary integral equations discretized using the QBX method. The solver
is constructed using the method described in~\Cref{sc:schemes}, where the
interpolation matrices $\matr{L}$ and $\matr{R}$ are obtained by applying the
proxy skeletonization method from \Cref{ssc:schemes:proxy} on the \stagetwo{}
discretization of the QBX method, as described in \Cref{ssc:schemes:qbx}. For
simplicity, we analyze the single-level direct solver~\eqref{eq:schemes:skeletonization}
using the same number of proxy points $q$ and the proxy radius factor
$\alpha$ in each cluster. Furthermore, we only consider the three-dimensional
single-layer operator~\eqref{eq:prelim:lp}, while the two-dimensional case can be
handled analogously. The results can also be extended to the case where the
proxy points $q_i$ and the proxy radius factor $\alpha_i$ are defined per cluster.

In recent years, there have been several renewed efforts at providing more rigorous
error estimate for proxy-based skeletonization direct solvers
\cite{Ye2020,Xing2020a}. In this work, we present an error estimate for
complete implementations, such as~\cite{Gillman2012,Ho2012}, that extends to
the case where the kernel $K$ is regularized by the QBX method. The main ingredient
used in determining the proxy points $\{\vect{p}_k\}$ is a
quadrature rule on the sphere that serves as a
source of proxy points that are able to integrate spherical harmonics up to a
given order $2 p$ (see \Cref{lemma:errors:multipole}). For a given quadrature
rule we define $q = Q(p)$ as the number of points corresponding to the order $p$.
There exist in the literature high-order accurate quadrature rules on the
sphere having equal weights, termed \emph{spherical designs}
\cite[Section 1.2]{Womersley2018}. According to numerical and theoretical results there,
one may expect that $q = Q(2 p) = 2 p (p + 1) + O(1)$, as $p \to \infty$. On account of
the equality of weights, we chose these rules as a source of proxy points.

The significant parameters in this analysis include the number of proxy points
$q$ and the ID tolerance $\epsilon_{\text{id}}$. We start by stating an error
bound treating both as independent parameters in~\Cref{thm:errors:ids}
and~\Cref{cor:errors:global}. Then, \Cref{cor:errors:proxy_count} provides
guidance in choosing $q$ so that the error is minimized (by balancing the terms
in~\eqref{eq:errors:global}).

\begin{theorem} \label{thm:errors:ids}
Consider a source point cluster $Y_j$ and target points in $X^{(\text{d})}_j \coloneqq
\cup_{i \ne j} X_i$ stemming from a discretized three-dimensional geometry. Let the
cluster $Y_j$ have a centroid $\vect{c}_{\text{src}, j}$ and a radius $r_{\text{src},
j} \coloneqq \max_{\vect{y} \in Y_j} \|\vect{y} - \vect{c}_{\text{src}, j}\|_2$. Let
$\Sigma_{\text{pxy}, j}$ be the proxy sphere, centered at $\vect{c}_{\text{src},
j}$ with a radius $r_{\text{pxy}, j} \coloneqq \alpha r_{\text{src}, j}$, for
$\alpha > 1$, which is sampled at a set of $q$ proxy points $\{\vect{p}_k\}_{k = 0}^{q
- 1}$. Then, the source skeletonization error from the cluster $Y_j$ to points in
$X^{(\text{d})}_j$ is given by
\begin{equation} \label{eq:errors:ids}
\begin{aligned}
& \|\vect{A}_{:, j} - (\matr{D}_{:, j} + \matr{S}_{:, j} \matr{R}_j)\|_2 \\
=\,\, &
\|\matr{K}(X^{(\text{d})}_j, Y_j) \matr{W}(Y_j)
  - \matr{K}(X^{(\text{d})}_j, Y_j^{(\text{s})}) \matr{W}(Y_j^{(\text{s})}) \matr{R}_j\|_2 \\
\le\,\, &
\left(1 + c_0 \frac{4 \pi r^2_{\text{pxy}, j}}{q}\right) \epsilon_{\text{id}}
+ c_1 \frac{1}{4 \pi r_{\text{pxy}, j}} \frac{1}{\alpha - 1}
  \alpha^{\left\lceil \frac{1}{2} Q^{-1}(q) \right\rceil},
\end{aligned}
\end{equation}
where the constants $c_0$ (see~\eqref{eq:errors:constants} and \Cref{prop:errors:t_estimate})
and $c_1$ (see \eqref{eq:errors:constants} and \Cref{lemma:errors:multipole})
depend on the geometry and size of the point sets $Y_j$ and $X_j$, but not on proxy
surface parameters $(q, r_{\text{pxy}, j})$ and tolerance $\epsilon_{\text{id}}$.
\end{theorem}

The proof of this result is given in~\Cref{ssc:errors:global} and
\Cref{ssc:errors:multipole}. An analogous result can be obtained for the target
skeletonization error $\|\matr{A}_{i, :} - (\matr{D}_{i, :}) + \matr{L}_i
\matr{S}_{i, :})\|_2$, where only the constants differ, as stated
in~\eqref{eq:errors:constants}. The two can be combined to obtain an estimate of
the global error for the skeletonization of a three-dimensional single-layer
potential as follows.

\begin{corollary} \label{cor:errors:global}
Under the assumptions of \Cref{thm:errors:ids}, we have that
\begin{equation} \label{eq:errors:global}
\|\matr{A} - (\matr{D} - \matr{L} \matr{S} \matr{R})\|_2 \le
c \left\{
\left(1 + c_0 \frac{4 \pi R^2_{\text{pxy}}}{q}\right) \epsilon_{\text{id}}
+ c_1 \frac{1}{4 \pi R_{\text{pxy}}} \frac{1}{\alpha - 1} \frac{1}{\alpha^p}
\right\},
\end{equation}
where $R_{\text{pxy}} \coloneqq \max r_{\text{pxy}, i}$ and $c_0, c_1 \in
\mathbb{R}$ are constants that do not depend on $(q, R_{\text{pxy}})$ and
tolerance $\epsilon_{\text{id}}$.
\end{corollary}

The result from~\Cref{thm:errors:ids} can be used to provide an estimate for
the proxy radius $r_{\text{pxy}, j}$ and the number of proxy points $q$ required
to achieve a given tolerance. A good choice of the proxy point count $q$, for a
specified tolerance $\epsilon_{\text{id}}$, can be obtained by ensuring the
two terms in the estimate from~\Cref{cor:errors:global} are of equal magnitude.

\begin{corollary} \label{cor:errors:proxy_count}
Under the assumptions of~\Cref{thm:errors:ids}, the global error estimate
from~\eqref{eq:errors:global} is minimized by balancing the two terms. This is
achieved by choosing
\begin{equation} \label{eq:errors:proxy_count}
p \sim -\frac{1}{\log \alpha} \log \left[
  (\alpha - 1) \frac{4 \pi R_{\text{pxy}} (1 + 4 \pi c_0 R^2_{\text{pxy}})}{c_1}
  \epsilon_{\text{id}}
\right],
\end{equation}
for constants $c_0, c_1 \in \mathbb{R}_+$. Then, we take the proxy count $q$ to be
$q \sim 2 p (p + 1)$.
\end{corollary}

In practice, the constants $c_0, c_1$ can be determined starting from
\Cref{thm:errors:ids} in a more precise manner. However, it is unclear if this
is desired, as it may be prohibitively expensive to provide good estimates
for the geometry terms inherent in their definition~\eqref{eq:errors:constants}.

\begin{remark}
\reva{The presented results and the following analysis can be applied per-level to
construct estimates for the full hierarchical direct solver. In particular, the
results from \Cref{thm:errors:ids} can be applied to the skeletonized source and target
point clusters $X^{(\text{s})}_i$ and $Y^{(\text{s})}_j$, respectively, with no
modifications. Therefore, the asymptotic dependence on the proxy radius and
proxy count will remain the same.}
\end{remark}

\subsection{Global Matrix Bounds from Cluster Interactions}
\label{ssc:errors:global}

Without a-priori knowledge of the skeletonization procedure, we can make standard
estimates for the forward and solution error inherent in the approximation. For
example, we have that~\cite[Section 5]{Ho2012}
\[
\begin{aligned}
\frac{\|\vect{b} - \vect{b}_\epsilon\|_2}{\|\vect{b}\|_2} & \coloneqq
\frac{\|\vect{b} - \matr{A}_\epsilon \vect{\sigma}\|_2}{\|\vect{b}\|_2}
\le \epsilon_{\text{id}} \kappa(\matr{A}), \\
\frac{\|\vect{\sigma} - \vect{\sigma}_\epsilon\|_2}{\|\vect{\sigma}\|_2}
& \coloneqq
\frac{\|\vect{\sigma} - \matr{A}^{-1}_\epsilon \vect{b}\|_2}{\|\vect{\sigma}\|_2}
\le 2 \epsilon_{\text{id}} \frac{\kappa(\matr{A})}{1 - \epsilon \kappa(\matr{A})},
\end{aligned}
\]
where $\kappa(\matr{A})$ is the $\ell^2$-norm condition number of the matrix $\matr{A}$.
These errors are especially important for the ill-conditioned single-layer case, where
the condition number is high. We can also show that a good approximation of
$\matr{A}_\epsilon$ does not necessarily imply a good approximation of $\matr{A}^{-1}$
due to~\cite[Section 6.4]{Gillman2012}
\[
\begin{aligned}
&
\matr{A}^{-1} - \matr{A}^{-1}_\epsilon
= -\matr{A}^{-1} (\matr{A} - \matr{A}_\epsilon) \matr{A}^{-1}_\epsilon \\
\implies &
\|\matr{A}^{-1} - \matr{A}^{-1}_\epsilon\|_2
\le \|\matr{A}^{-1}\|_2 \|\matr{A}^{-1}_\epsilon\|_2
  \|\matr{A} - \matr{A}_\epsilon\|_2,
\end{aligned}
\]
i.e.\ the error in the inverse can be large if the norms of the true or the approximate
inverse are large. From our numerical experiments in~\Cref{sc:results}, the dominant
source of error in using the approximate inverse to solve the equations appears
to be the conditioning of the original operator. However, \cite[Figure 6]{Gillman2012}
shows that the approximation error of the inverse can become significantly larger
than the approximation error of the operator itself for specific domains.

For a more specific error estimate, we consider the case of the single-level
proxy-based skeletonization approximant from~\Cref{sc:schemes}. We have that
\[
\begin{aligned}
\|\matr{A} - \matr{A}_\epsilon\|_2 =
\|\matr{A} - (\matr{D} + \matr{L} \matr{S} \matr{R})\|_2,
\end{aligned}
\]
where the diagonal blocks are always evaluated exactly and do not contribute
to the error. To simplify notation below, we introduce the notation $\matr{A}^{(\text{d})}
\coloneqq \matr{A} - \matr{D}$, i.e.,\ the matrix $\matr{A}$ with zero diagonal blocks.
Analogously, we define $\matr{K}^{(\text{d})}(X, Y)$ as the kernel interactions without
the diagonal blocks.

\emph{Separating source and target errors}. A first step consists of reducing
the block-wise error to one that separates the source and target skeletonization.
This is required to account for the fact that the source and target
skeletonization matrices are obtained from different constructions, as detailed
in~\Cref{ssc:schemes:proxy}. This gives
\[
\begin{aligned}
\|\matr{A} - \matr{A}_\epsilon\|_2 & =
\|
  \matr{K}^{(\text{d})}(X, Y) \matr{W}(Y)
  - \matr{L} \matr{K}^{(\text{d})}(X^{(\text{s})}, Y^{(\text{s})})
    \matr{W}(Y^{(\text{s})}) \matr{R}
\|_2 \\
& \le
\frac{1}{2} (1 + \|\matr{R}\|_2)
\|
  \matr{K}^{(\text{d})}(X, Y) \matr{W}(Y)
  - \matr{L} \matr{K}^{(\text{d})}(X^{(\text{s})}, Y) \matr{W}(Y)
\|_2 \\
& +
\frac{1}{2} (1 + \|\matr{L}\|_2)
\|
  \matr{K}^{(\text{d})}X, Y) \matr{W}(Y)
  - \matr{K}^{(\text{d})}(X, Y^{(\text{s})}) \matr{W}(Y^{(\text{s})}) \matr{R}
\|_2,
\end{aligned}
\]
where we have introduced additional terms and applied the triangle inequality
to obtain the second bound. By applying the triangle inequality again, we can
further decompose the error into block-wise terms
\[
\begin{aligned}
\|\matr{A} - \matr{A}_\epsilon\|_2 & \le
\frac{1}{2} (1 + \|\matr{R}\|_2) \sum_{i = 0}^N
\|
  \matr{K}^{(\text{d})}(X_i, Y) \matr{W}(Y)
  - \matr{L}_i \matr{K}^{(\text{d})}(X_i^{(\text{s})}, Y) \matr{W}(Y)
\|_2 \\
& +
\frac{1}{2} (1 + \|\matr{L}\|_2) \sum_{j = 0}^N
\|
  \matr{K}^{(\text{d})}(X, Y_j) \matr{W}(Y_j)
  - \matr{K}^{(\text{d})}(X, Y_j^{(\text{s})}) \matr{W}(Y_j^{(\text{s})}) \matr{R}_j
\|_2 \\
& \coloneqq
\frac{1}{2} (1 + \|\matr{R}\|_2) \sum_{i = 0}^N \|\matr{E}_{\text{tgt}, i}(X_i, Y)\|_2
+
\frac{1}{2} (1 + \|\matr{L}\|_2) \sum_{j = 0}^N \|\matr{E}_{\text{src}, j}(X, Y_j)\|_2.
\end{aligned}
\]

We can then focus on the problem of skeletonizing a single row or column for
target or source skeletonization, respectively. Taking a $\|\matr{E}_{\text{src}, j}\|_2$ term,
we further separate the near-field and far-field interactions to obtain
\[
\begin{aligned}
&
\|
  \matr{K}^{(\text{d})}(X, Y_j) \matr{W}(Y_j)
  - \matr{K}^{(\text{d})}(X, Y_j^{(\text{s})}) \matr{W}(Y_j^{(\text{s})}) \matr{R}_j
\|_2 \\
\le \,\, &
\|
  \matr{K}(X_j^{(\text{near})}, Y_j) \matr{W}(Y_j)
  - \matr{K}(X_j^{(\text{near})}, Y_j^{(\text{s})}) \matr{W}(Y_j^{(\text{s})}) \matr{R}_j
\|_2 \\
+ \,\, &
\|\matr{K}(X_j^{(\text{far})}, Y_j) \matr{W}(Y)_j
  - \matr{K}(X_j^{(\text{far})}, Y_j^{(\text{s})}) \matr{W}(Y_j^{(\text{s})}) \matr{R}_j\|_2 \\
\le \,\, &
\epsilon_{\text{id}}
+ \|\matr{K}(X_j^{(\text{far})}, Y_j) \matr{W}(Y_j)
    - \matr{K}(X_j^{(\text{far})}, Y_j^{(\text{s})}) \matr{W}(Y_j^{(\text{s})}) \matr{R}_j
  \|_2,
\end{aligned}
\]
where the first term is bounded by the error in the interpolative decomposition
due to being a (row) subset of the factorized matrix in~\Cref{mm:schemes:source_skeletonization}.
The $X^{(near)}_j$ near-field and $X^{(far)}_j$ far-field target subsets are
constructed as in~\eqref{eq:schemes:proxy_near_far}.

\emph{Introducing the proxy points}. The second step consists of introducing the
proxy points into the error incurred from the far-field interactions above. This
results in a model of the error that takes into account the number of proxy points
$q$ and the proxy radius $r_{\text{pxy}, j}$. We introduce an operator $\matr{T}$
such that, pointwise,
\[
\matr{K}(X_j^{(\text{far})}, \vect{y})
= \matr{T}(X_j^{(\text{far})}, P_j) \matr{K}(P_j, \vect{y})
  + \matr{E}_{\text{src}, j}(\vect{y}; q, r_{\text{pxy}, j}),
\]
for all $\vect{y} \in Y_j$ and an error matrix $\matr{E}_{\text{src}, j}$ that
depends on $(q, r_{\text{pxy}, j})$. This operator is based on the Poisson
kernel~\eqref{eq:prelim:poisson} and the error term is made explicit in
\Cref{ssc:errors:multipole}. Then, the far-field term can be bounded by
\[
\begin{aligned}
&
\|\matr{K}(X_j^{(\text{far})}, Y_j) \matr{W}(Y_j)
  - \matr{K}(X_j^{(\text{far})}, Y_j^{(\text{s})}) \matr{W}(Y_j^{(\text{s})}) \matr{R}_j\|_2 \\
\le \,\, &
\|\matr{T}(X_j^{(\text{far})}, P_j)\|_2 \|
\matr{K}(P_j, Y_j) \matr{W}(Y_j) - \matr{K}(P_j, Y_j^{(\text{s})}) \matr{W}(Y_j^{(\text{s})}) \matr{R}_j\|_2 \\
+ & (1 + \|\matr{R}_j\|_2) \|\matr{W}(Y_j)\|_2 \|\matr{E}_{\text{src}, j}(Y_j)\|_2 \\
\le \,\, &
\|\matr{T}(X_j^{(\text{far})}, P_j)\|_2 \epsilon_{\text{id}}
+
(1 + \|\matr{R}_j\|_2) \|\matr{W}(Y_j)\|_2 \|\matr{E}_{\text{src}, j}(Y_j)\|_2,
\end{aligned}
\]
where the first term can be found in the proxy matrix from
\Cref{mm:schemes:source_skeletonization}. As such, its error can also be
bounded by the tolerance of the Interpolative Decomposition. Therefore, it
remains to provide a bound for the error $\matr{E}_{\text{src}, j}(Y_j)$
introduced by the operator $\matr{T}$ in the above construction.

For target skeletonization, we can analogously obtain
\[
\begin{aligned}
&
\|\matr{K}(X_i, Y_i^{(\text{far})}) \matr{W}(Y_i^{(\text{far})})
  - \matr{L}_i \matr{K}(X_i^{(\text{s})}, Y_i^{(\text{far})}) \matr{W}(Y_i^{(\text{far})})\|_2 \\
\le \,\, &
\|\matr{W}^{-1}(P_i)\|_2 \|\matr{T}(P_i, Y_i^{(\text{far})})\|_2 \|\matr{W}(Y_i^{(\text{far})})\|_2 \\
\times \,\, &
\|\matr{K}(X_i, P_i) \matr{W}(P_i)
  - \matr{L}_i \matr{K}(X_i^{(\text{s})}, P_i) \matr{W}(P_i)\|_2 \\
+ \,\, &
(1 + \|\matr{L}_i\|_2) \|\matr{W}(Y_i^{(\text{far})})\|_2
  \|\matr{E}_{\text{tgt}, i}(X_i)\|_2,
\end{aligned}
\]
where we can see that the first term corresponds to the proxy matrix construction
from~\Cref{mm:schemes:target_skeletonization}. Putting the two results together,
we have that
\begin{equation} \label{eq:errors:clusterwise}
\begin{aligned}
\|\matr{E}_{\text{tgt}, i}(X_i, Y)\|_2
& \le
  (1 + a_{0, i} \|\matr{T}(P_i, Y_i^{(\text{far})})\|_2)
  \epsilon_{\text{id}}
+ a_{1, i} \|\matr{E}_{\text{tgt}, i}(X_i)\|_2, \\
\|\matr{E}_{\text{src}, j}(X, Y_j)\|_2
& \le (1 + b_{0, j} \|\matr{T}(X_j^{(\text{far})}, P_j)\|_2)
  \epsilon_{\text{id}}
+ b_{1, j} \|\matr{E}_{\text{src}, j}(Y_i)\|_2, \\
\end{aligned}
\end{equation}
where
\[
\begin{aligned}
\|\matr{E}_{\text{tgt}, i}(X_i)\|_2 & =
\|\matr{K}(X_i, Y_i^{(\text{far})})
  - \matr{K}(X_i, P_i) \matr{T}(P_i, Y_i^{(\text{far})})\|_2, \\
\|\matr{E}_{\text{src}, j}(Y_j)\|_2 & =
\|\matr{K}(X_j^{(\text{far})}, Y_j)
  - \matr{T}(X_j^{(\text{far})}, P_j) \matr{K}(P_j, Y_j)\|_2, \\
\end{aligned}
\]
and the constants are given by
\begin{equation} \label{eq:errors:constants}
\begin{aligned}
a_{0, i} & \coloneqq
  \|\matr{W}^{-1}(P_i)\|_2 \|\matr{W}(Y_i^{(\text{far})})\|_2, \\
b_{0, j} & \coloneqq 1, \\
c_0 & \coloneqq \max_{0 \le i \le N} (a_{0, i}, b_{0, i}),
\end{aligned}
\qquad
\begin{aligned}
a_{1, i} & \coloneqq
  (1 + \|\matr{L}_i\|_2) \|\matr{W}(Y_i^{(\text{far})})\|_2, \\
b_{1, j} & \coloneqq
  (1 + \|\matr{R}_j\|_2) \|\matr{W}(Y_j)\|_2, \\
c_1 & \coloneqq \max_{0 \le j \le N} (a_{1, i}, b_{1, i}).
\end{aligned}
\end{equation}

Then, it remains to find an estimate for the source and target far-field
errors, denoted as $\|\matr{E}_{\text{src}, j}(Y_i)\|_2$ and
$\|\matr{E}_{\text{tgt}, i}(X_j)\|_2$, respectively, and the norm of the
$\matr{T}$ operators. Bounds on the remaining quantities are given by the
ID~\Cref{thm:prelim:id} ($\matr{L}$ and $\matr{R}$) and the geometry (the
diagonal $\matr{W}$ matrices).

\subsection{Cluster Multipole-based Error Estimates}
\label{ssc:errors:multipole}

In this section, we give an estimate for the source error $\varepsilon_{s, j}$
and the target error $\varepsilon_{t, i}$ for an arbitrary cluster. This is
achieved by making use of multipole expansions of the kernels in question.
For the Laplace case, we direct the reader to~\cite{Greengard1987} for the
two-dimensional expansions, to~\cite{Greengard1997} for the corresponding
three-dimensional expansions, and to~\cite{Wala2020} for
extensions to the expansion of the QBX regularized kernels. The estimates focus
on the source error $\epsilon_{s, j}$ in the three-dimensional case, but the
target error $\epsilon_{t, j}$ and the two-dimensional cases can be obtained
analogously.

We consider the Poisson Integral Formula~\eqref{eq:prelim:poisson_integral},
written for the harmonic function $g_{\vect{y}}(\vect{x}) \coloneqq K(\vect{x},
\vect{y})$ as
\begin{equation} \label{eq:errors:poisson}
K(\vect{x}, \vect{y}) = \int_{\Sigma_{\text{pxy}, i}}
  P(\vect{x}, \vect{p}) K(\vect{p}, \vect{y}) \dx[\vect{p}],
\end{equation}
where $\vect{x} \in X_i^{(\text{far})}$ and $\vect{y} \in Y_i$. By construction,
we have that $r_{\text{src}, j} < r_{\text{pxy}, j} < r_{\text{far}, i}$,
where $r_{\text{far}, j} \coloneqq \min_{\vect{x} \in X^{(\text{far})}} \|\vect{x}\|_2$.
Therefore, the expression above is well-defined, and we can construct a discretization
of~\eqref{eq:errors:poisson} that defines the $\vect{T}$ operator. An error
estimate based on this bound is given below.

\begin{lemma} \label{lemma:errors:multipole}
Consider a cluster $j$ with $\vect{x} \in X_i^{(\text{far})}, \vect{y} \in Y_j$
with a proxy surface $\Sigma_{\text{pxy}, j}$ centered at $\vect{c}_{\text{src}, j}$
with a radius $r_{\text{pxy}, j} \coloneqq \alpha \max_{\vect{y} \in Y_j} \|\vect{y} -
\vect{c}_{\text{src}, j}\|_2$, as defined in~\Cref{ssc:schemes:proxy}. Let
$(w_k, \vect{p}_k)$ be a quadrature rule with $q$ points on $\Sigma_{\text{pxy}, j}$
that integrates spherical harmonics of order $2 p$ exactly on the proxy surface. Then,
we have that the pointwise source error $\epsilon_{s, j}$ is given by
\[
\left|
K(\vect{x}, \vect{y}) - \sum_{k = 0}^{q - 1} T(\vect{x}, \vect{p}_k) K(\vect{p}_k, \vect{y})
\right| \le \frac{1}{4 \pi r_{\text{pxy}, j}} \frac{1}{\alpha - 1} \frac{1}{\alpha^p},
\]
with $T(\vect{x}, \vect{p}_k) \coloneqq P(\vect{x}, \vect{p}_k) w_k$, where
$P(\vect{x}, \vect{p}_k)$ is the pointwise evaluation of the Poisson
kernel~\eqref{eq:prelim:poisson}.
\end{lemma}

\begin{proof}
We start by choosing a quadrature scheme on the sphere $\Sigma_{\text{pxy}, j}$,
centered at $\vect{c}_{\text{src}, j}$ with radius $r_{\text{pxy}, j}$ that can
exactly integrate spherical harmonics up to order $2 p$. We assume that the
quadrature rule is a ``spherical design'', such that the weights are given by
\[
w_k \equiv w \coloneqq \frac{4 \pi r^2_{\text{pxy}, j}}{q}.
\]

An example for such a spherical quadrature rule is given in~\cite{Womersley2018},
with a number of nodes that grows quadratically in the required order $p$. Note that,
on a circle, the standard trapezoidal rule satisfies all the requirements. Using
the quadrature rule given by such a spherical design, we can write
\[
\begin{aligned}
K(\vect{x}, \vect{y}) & =
  \sum_{k = 0}^{q - 1} P(\vect{x}, \vect{p}_k) K(\vect{p}_k, \vect{y}) w_k
  + E_{\text{quad}}(\vect{x}, \vect{y}) \\
  & =
  \sum_{k = 0}^{q - 1} T(\vect{x}, \vect{p}_k) K(\vect{p}_k, \vect{y})
  + E_{\text{quad}}(\vect{x}, \vect{y}),
\end{aligned}
\]
where we denote the residual as $E_{\text{quad}}(\vect{x}, \vect{y})$. We now
introduce the expansions of the Laplace kernel $K$ and the Poisson kernel $P$ given
by~\eqref{eq:prelim:laplace_expansion} and~\eqref{eq:prelim:poisson_expansion},
respectively. Using an order $p$ expansion, we can write
\[
\begin{aligned}
& \sum_{k = 0}^{q - 1} P(\vect{x}, \vect{p}_k) K(\vect{p}_k, \vect{y}) w_k
  + E_{\text{quad}}(\vect{x}, \vect{y}) \\
= &
\sum_{k = 0}^{q - 1} \left(
\sum_{l = 0}^p \sum_{n = -l}^l
  \frac{r_{\text{pxy}, j}^{l - 1}}{r_x^{l + 1}}
  Y^n_l(\hat{\vect{x}}) Y^{-n}_l(\hat{\vect{p}}_k)
\right)
\left(
\sum_{j = 0}^p \sum_{m = -j}^j
  \frac{1}{2 j + 1}
  \frac{r_y^j}{r_{\text{pxy}, j}^{j + 1}}
  Y^m_j(\hat{\vect{p}}_k) Y^{-m}_j(\hat{\vect{y}})
\right) w_k \\
+ & E_{\text{quad}}(\vect{x}, \vect{y}) + E_{\text{trunc}}(\vect{x}, \vect{y}),
\end{aligned}
\]
where the error due to truncation is denoted as $E_{\text{trunc}}(\vect{}, \vect{y})$.
Expanding the sums and re-arranging the terms gives
\[
\begin{aligned}
&
\sum_{l = 0}^p \sum_{n = -l}^l
\sum_{j = 0}^p \sum_{m = -j}^j
  \frac{1}{2 j + 1}
  \frac{r_{\text{pxy}, j}^{l - 1}}{r_{\text{pxy}, j}^{j + 1}}
  \frac{r_y^j}{r_x^{l + 1}}
  Y^n_l(\hat{\vect{x}}) Y^{-m}_j(\hat{\vect{y}})
  \left(
  \sum_{k = 0}^{q - 1} Y^{-n}_l(\hat{\vect{p}}_k) Y^m_j(\hat{\vect{p}}_k) w_k
  \right) \\
= &
  \sum_{l = 0}^p \sum_{m = -l}^l
  \frac{1}{2 l + 1}
  \frac{r_y^l}{r_x^{l + 1}}
  Y^m_l(\hat{\vect{x}}) Y^{-m}_l(\hat{\vect{y}}),
\end{aligned}
\]
where we have used the fact that $(\vect{p}_k, w_k)$ can exactly integrate
the spherical harmonics up to order $2 p$. This allowed the use of the orthogonality
of the spherical harmonics, i.e.,
\[
  \sum_{k = 0}^{q - 1} Y^{-n}_l(\hat{\vect{p}}_k) Y^m_l(\hat{\vect{p}}_k) w_k
  = r_{\text{pxy}, j}^2 \delta_{l, l} \delta_{m, n}.
\]

We can see that we are left with the standard truncated expansion of the
Laplace kernel. Applying~\cite[Theorem 3.2]{Greengard1997} gives the desired
bound. Therefore, we have that
\[
|E_{\text{quad}}(\vect{x}, \vect{y}) + E_{\text{trunc}}(\vect{x}, \vect{y})| \le
  \frac{1}{4 \pi (r_x - r_y)} \left(\frac{r_y}{r_x}\right)^{p + 1},
\]
or, given that $r_{\text{pxy}, j} < \min_{\vect{x} \in X_j^{\text{far}}}
\|\vect{x} - \vect{c}_{\text{src}, j}\|_2$ by definition of $X_j^{(\text{far})}$,
\[
|E_{\text{quad}}(\vect{x}, \vect{y}) + E_{\text{trunc}}(\vect{x}, \vect{y})|
  \le \frac{1}{4 \pi r_{\text{pxy}, j}} \frac{\alpha^{-p}}{\alpha - 1}.
\]
\qed
\end{proof}

This completes the proof of the multipole-based error estimate for
$\|\vect{E}_{\text{src}, j}(Y_i)\|$ in~\eqref{eq:errors:clusterwise}. It remains
to provide a bound for the source operator $T(\vect{x}, \vect{y})$, which is based
on the chosen quadrature rule $(\vect{p}_k, w_k)$ as shown below.

\begin{proposition} \label{prop:errors:t_estimate}
Consider the construction from~\Cref{lemma:errors:multipole}.
Under the assumptions of~\Cref{lemma:errors:multipole}, the operator $\vect{T}$
used in source error estimate~\eqref{eq:errors:clusterwise} is bounded by
\[
\|\vect{T}(X^{(\text{far})}_j, P_j)\|_2 \le C \frac{4 \pi r^2_{\text{pxy}, j}}{q}.
\]
\end{proposition}

\begin{proof}
From \Cref{lemma:errors:multipole}, we have that $T(\vect{x}, \vect{p}_k)
\coloneqq w_k P(\vect{x}, \vect{p}_k)$. As the Poisson kernel is bounded
when $r_x > r_{\text{pxy}, j}$, we have
\[
\|\vect{T}(X^{(\text{far})}_j, P_j)\|_2
  \le w \|\vect{P}(X^{(\text{far})}_j, P_j)\|_2
  \le C \frac{4 \pi r^2_{\text{pxy}, j}}{q},
\]
for some constant $C > 0$. Note that, in practice, the far-field points are
close to the proxy surface, so the constant will be large. This can be controlled
by increasing the number of proxy points $q$, i.e.,\ the quadrature resolution.
\qed
\end{proof}

This completes the proof of \Cref{thm:errors:ids}. In~\Cref{ssc:errors:global},
we have decomposed the global error of the direct solver into cluster-level errors
and appropriate constants independent of the proxy point selection. Then, \Cref{lemma:errors:multipole} and \Cref{prop:errors:t_estimate} above give
appropriate bounds for the remaining terms and clarify the dependence on the
number of proxy points. In the case of the source skeletonization terms from~\eqref{eq:errors:clusterwise},
this leads to
\[
\begin{aligned}
& \|
  \matr{K}^{(\text{d})}(X_i, Y) \matr{W}(Y)
  - \matr{L}_i \matr{K}^{(\text{d})}(X_i^{(\text{s})}, Y) \matr{W}(Y)
\|_2 \\
\le\,\, &
  \left(1 + a_{0, i} \frac{4 \pi r^2_{\text{pxy}, j}}{q}\right)
  \epsilon_{\text{id}}
+ a_{1, i} \frac{1}{4 \pi r_{\text{pxy}, j}} \frac{1}{\alpha - 1} \frac{1}{\alpha^p},
\end{aligned}
\]
with constants defined by~\eqref{eq:errors:constants}. Bounds for the
target skeletonization terms in~\eqref{eq:errors:clusterwise} can be obtained
analogously to the analysis above.
Then, to construct the global estimate from~\Cref{thm:errors:ids}, we directly
take a maximum over all the clusters. Note that a two-dimensional estimate can be
obtained in an analogous fashion by using the appropriate kernel expansions
(see \cite{Greengard1987}) and the uniform trapezoidal quadrature rule.

\section{Numerical Results}
\label{sc:results}

We present below a series of experiments to verify the error estimate
presented in~\Cref{sc:errors}. We also show the accuracy and cost scaling behavior
of the resulting solver for a pair of standard Laplace boundary value problems in
two and three dimensions. In two dimensions, we use a geometry given by
\begin{equation} \label{eq:results:starfish}
(x, y) = [1 + A \sin (n + 1) \theta] \left(\cos \theta, \sin \theta \right),
\end{equation}
where $A = 0.25$ and $n = 16$ (see~\Cref{fig:results:geometry}a).
In three dimensions we use a standard torus described by
\begin{equation} \label{eq:results:torus}
(x, y, z) = ((a + b \cos \theta) \cos \phi, (a + b \cos \theta) \sin \phi, b \sin \theta),
\end{equation}
where $a = 10$ and $b = 2$ (see~\Cref{fig:results:geometry}b). Both of these
geometries take advantage of the adaptive apparatus of the QBX method described
in~\Cref{ssc:prelim:qbx} and allow testing the different parts of the direct
solver. To discretize the geometry, we use a Gauss-Legendre quadrature
with $4$ nodes in two dimensions and a Vioreanu--Rokhlin quadrature with 15
nodes~\cite{Vioreanu2014} in three dimensions on the \stageone{} and \stagetwo{}
discretizations (see~\Cref{ssc:prelim:qbx}). The \stagequad{} discretization is
not used in the following tests. The QBX local expansion is performed using
$p_{qbx} = 4$, in the manner described in~\cite{Wala2019}. For the two-dimensional
geometry, this gives \num{10240} degrees of freedom for the \stagetwo{}
discretization. In three dimensions, we have \num{24000} degrees of freedom for
the \stagetwo{} discretization. Using the hierarchical clustering described in
\Cref{ssc:schemes:clusters}, this results in a 6-level and a 4-level
skeletonization in two and three dimensions, respectively.

\begin{figure}[ht!]
\begin{subfigure}{0.5\linewidth}
\centering
\includegraphics[width=0.65\linewidth]{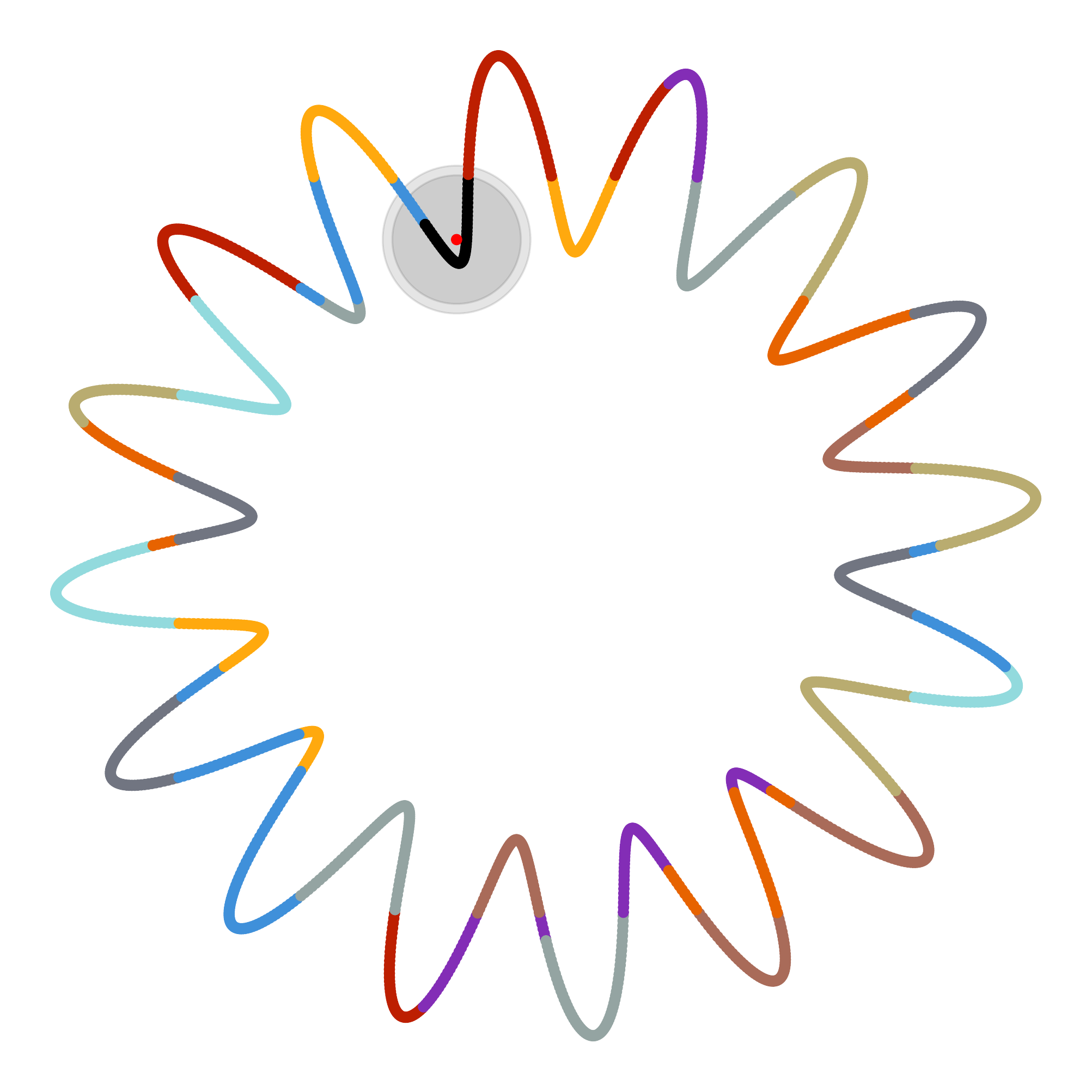}
\caption{}
\end{subfigure}%
\begin{subfigure}{0.5\linewidth}
\centering
\includegraphics[width=0.65\linewidth]{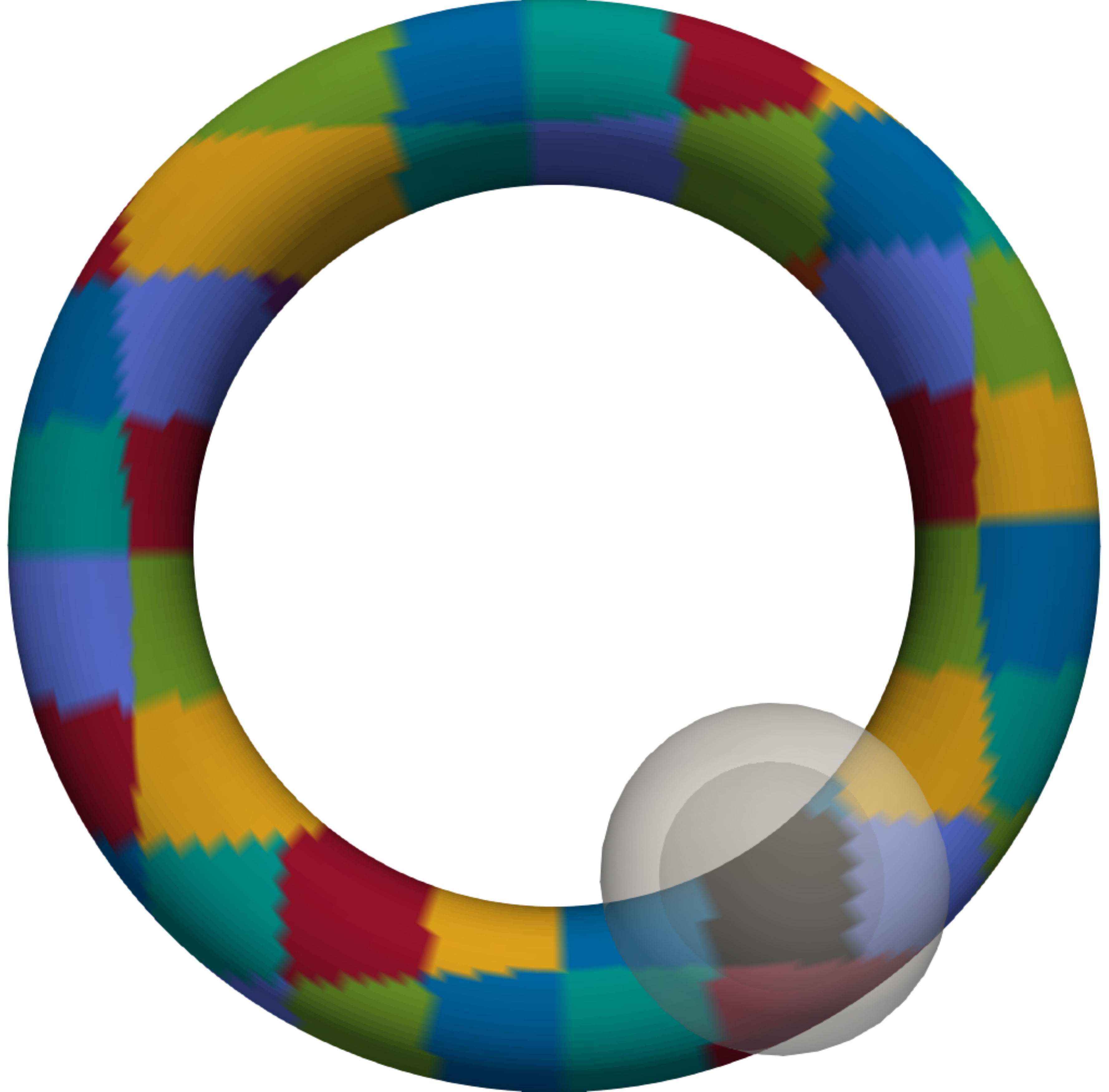}
\vspace{1em}
\caption{}
\end{subfigure}
\caption{Geometry in (a) 2D and (b) 3D. Different colors (not unique) denote different
  clusters and a representative proxy ball (gray) \revb{with a center
  $\vect{c}_{\text{src}}$ (red dot)} is shown on each geometry. The inner ball
  denotes the cluster bounding ball and the lighter outer circle denotes the proxy ball.}
\label{fig:results:geometry}
\end{figure}

Based on the quadtree/octree decomposition described
in~\Cref{ssc:schemes:skeletonization}, we observe cluster sizes at the leaf
level as shown in~\Cref{fig:results:clusters}. Notably, clusters
can vary in size by an order of magnitude due to the spatial structure of the
tree decomposition. While both the cluster size distribution and the geometry
have an impact on the error model developed in~\Cref{sc:errors}, they appear in
the various constants and have no effect on the asymptotic behavior that is
being verified here. Finally, we take a proxy radius factor of $\alpha = 1.15$,
as described in~\Cref{ssc:schemes:proxy}, with respect to a cluster radius that
also includes the QBX expansion balls.

To approximate the operator norms of the error from~\Cref{sc:errors} in a
computationally feasible manner for the matrix sizes involved, we generate a
uniformly random vector $\vect{\sigma}$ on $[-1, 1]$ and compute a reference
solution $\vect{b} = \matr{A} \vect{\sigma}$ by direct QBX-mediated point-to-point
evaluation (without FMM acceleration). Then, the relative errors
$E_{2, rel}(\vect{\sigma})$ are reported in the standard $\ell_2$ norm as
\begin{equation} \label{eq:results:rel_error_empirical}
E_{2, rel}(\vect{\sigma}) =
\frac{\|\vect{b} - \matr{A}_\epsilon \vect{\sigma}\|_2}{\|\vect{\sigma}\|_2},
\quad \vect{\sigma} \in \mathbb{R}^n,
\end{equation}
which, for each $\vect{\sigma}$, provide a lower bound for the operator norm of
$\matr{A} - \matr{A}_\epsilon$.

\begin{figure}[ht!]
\begin{subfigure}{0.5\linewidth}
\centering
\includegraphics[width=0.65\linewidth]{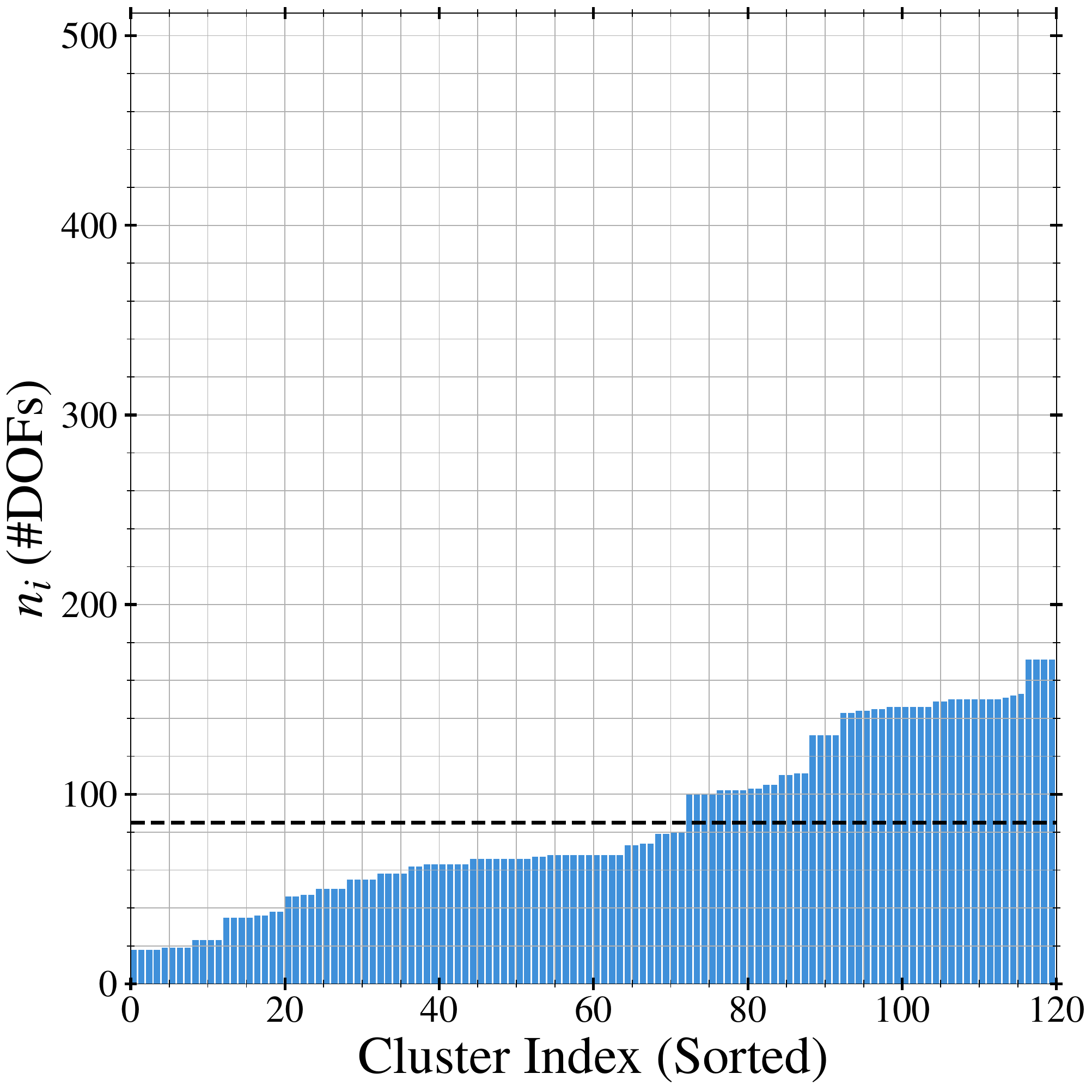}
\caption{2D.}
\end{subfigure}%
\begin{subfigure}{0.5\linewidth}
\centering
\includegraphics[width=0.65\linewidth]{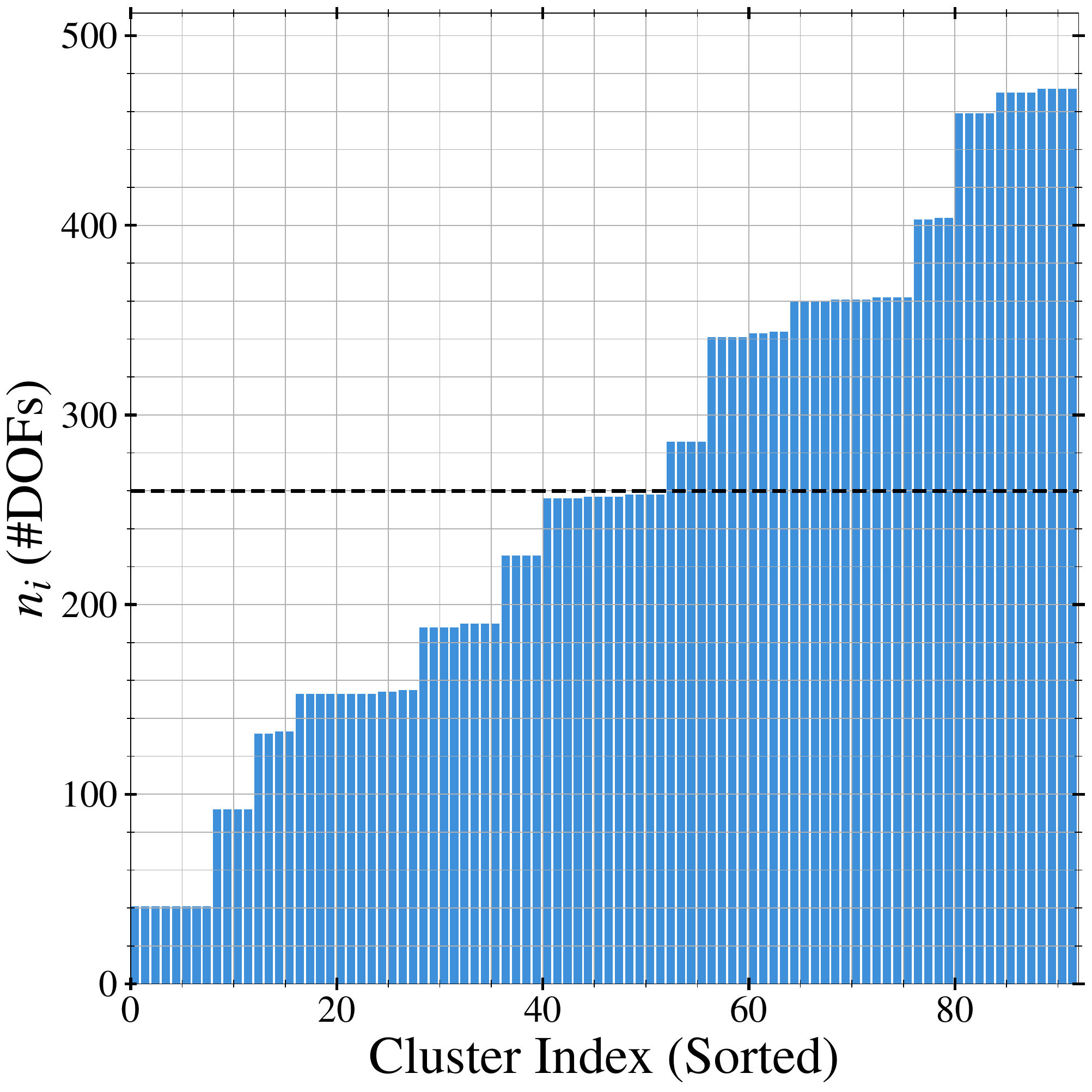}
\caption{3D}
\end{subfigure}
\caption{Sorted cluster size, where the dashed line denotes the mean cluster size.}
\label{fig:results:clusters}
\end{figure}

The convergence of the method with grid spacing is briefly shown
in~\Cref{ax:convergence}. In the following, we focus on showcasing the performance
of the direct solver and the analysis performed in \Cref{sc:errors}.

\subsection{Error Model Experiments}
\label{ssc:result:model}

We first present a series of experiments to compare the error model derived in
\Cref{sc:errors} with the empirical error data obtained from numerical simulations.
These experiments will focus on analyzing the error as a function of the proxy count $q$
and the proxy radius factor $\alpha$, defined by $\rpxy = \alpha\, \rtgt$. The
relevant theoretical results are given by~\Cref{thm:errors:ids}
and~\Cref{cor:errors:proxy_count}.

\subsubsection{Target-Proxy Weight Matrix}
\label{sssc:result:model:target_weight_matrix}

We start by examining the effect of introducing the weight matrix $\matr{W}(P)$
described in~\Cref{mm:schemes:target_skeletonization}. For this, we perform an
experiment that varies the ID tolerance $\epsilon_{\text{id}}$, while keeping the
remaining parameters fixed. We consider the three-dimensional case and take
$q = 192$ and $\alpha = 1.15$.

\begin{figure}[ht!]
\begin{subfigure}{0.5\linewidth}
\centering
\includegraphics[width=0.7\linewidth]{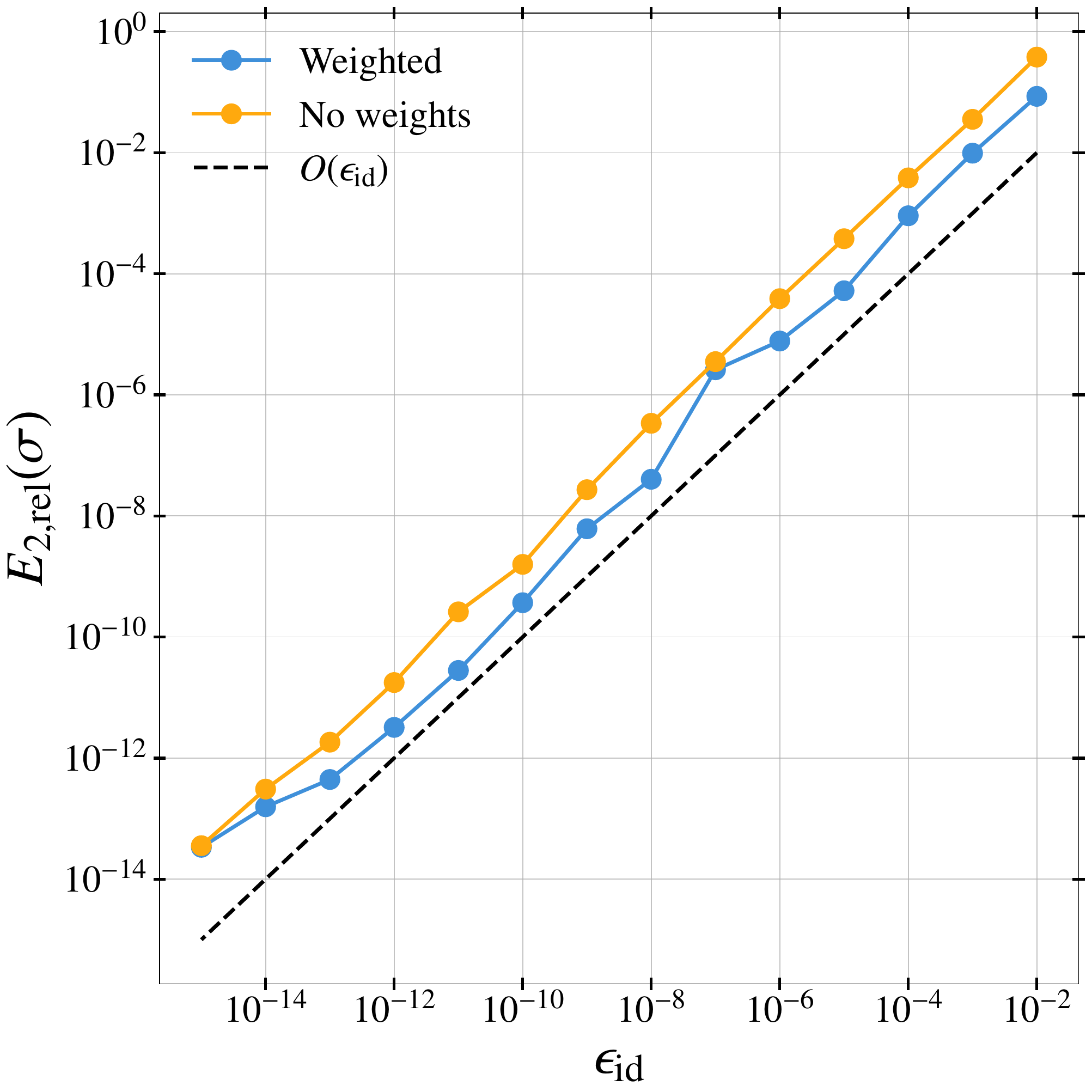}
\caption{3D single-layer ($\mathcal{S}$)}
\end{subfigure}%
\begin{subfigure}{0.5\linewidth}
\centering
\includegraphics[width=0.7\linewidth]{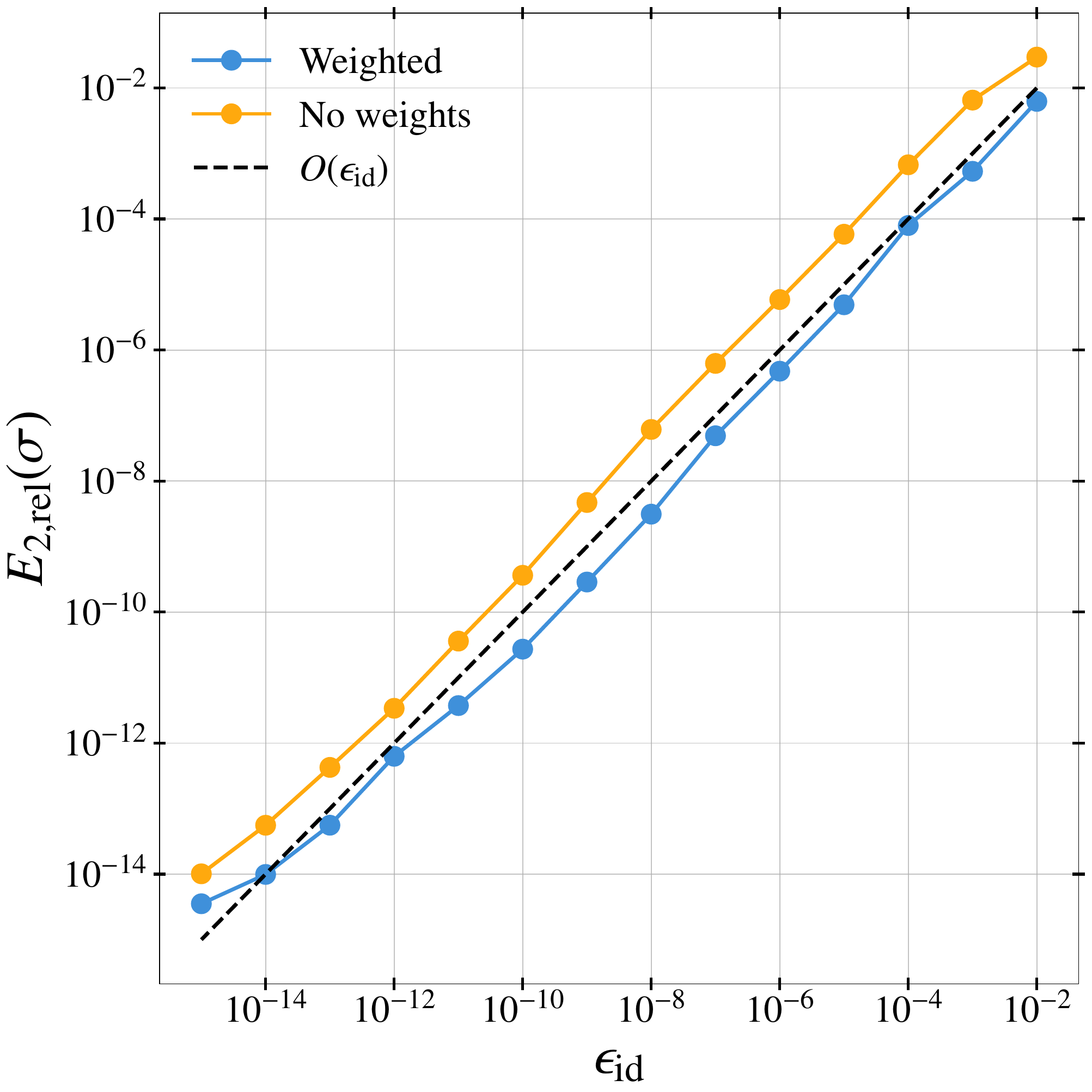}
\caption{3D double-layer ($-\sfrac{1}{2} \mathcal{I} + \mathcal{D}$)}
\end{subfigure}
\caption{Comparison of a direct solver using the additional weight matrix $\matr{W}(P)$
  when performing target skeletonization. All other weight matrices are kept in
  both cases. Results are shown for (a) a single-layer potential and (b) a
  double-layer potential on the chosen 3D geometry~\eqref{eq:results:torus}.}
\label{fig:results:target_weight_matrix}
\end{figure}

The results can be seen in \Cref{fig:results:target_weight_matrix} for a single-layer
and a double-layer potential. In the single-layer case, the difference is at most an
order of magnitude, but this gradually diminishes at small tolerances. In the
double-layer case we have a clearer improvement that is maintained at smaller
tolerances. However, in both cases, the additional benefit of $\matr{W}(P)$ is
clear. As such, this weight is added in all subsequent experiments.

\subsubsection{Error Estimates as a Function of Proxy Parameters}
\label{sssc:result:model:q_vs_error}

In this experiment, we vary the proxy count $q$ and the proxy radius factor $\alpha$,
while keeping the ID tolerance fixed. This removes the effects of the ID in the
proxy-based skeletonization from~\Cref{ssc:schemes:proxy} and focuses on the
multipole-based error estimates from~\Cref{ssc:errors:multipole}. As such, we
set $\epsilon_{\text{id}} = 10^{-15}$, i.e.,\ close to machine epsilon for
IEEE double precision floating point numbers. The empirical error data is obtained
from a skeletonization of a double-layer potential.

\begin{table}[ht]
\centering
\caption{Empirically determined constants $C_0, C_1$ in the error model.}
\label{tbl:results:constants}
\begin{tabular}{
  l
  |S[scientific-notation=true,round-mode=places,round-precision=3]
  |S[scientific-notation=true,round-mode=places,round-precision=3]} \toprule
& {$C_0$} & {$C_1$} \\\midrule
2D (single-layer) & 3.689422244809e+00 & 1.147006002096e-02 \\
2D (double-layer) & 1.005122937622e+00 & 4.678342024328e-04 \\
3D (single-layer) & 7.719297274962e-03 & 2.146727387609e-03 \\
3D (double-layer) & 5.314763845850e-03 & 1.581941072958e-05 \\
\bottomrule
\end{tabular}
\end{table}

In general, we cannot expect the data to reflect the constants from
\Cref{thm:errors:ids}, as they are upper bounds based on worst-case estimates.
Therefore, the modeled error shown in \Cref{fig:results:model_error} is constructed
from \Cref{cor:errors:global} as
\begin{equation} \label{eq:results:rel_error_model}
E_{2, model}(q) =
\frac{1}{2} (2 + \|\matr{L}\|_2 + \|\matr{R}\|_2) \left\{
\left(1 + C_0 c_0 \frac{4 \pi R_{\text{pxy}}^2}{q}\right) \epsilon_{\text{id}}
+ C_1 c_1 \frac{1}{4 \pi R_{\text{pxy}}} \frac{1}{\alpha - 1} \frac{1}{\alpha^p}
\right\},
\end{equation}
where the constants $C_0, C_1$ are determined empirically to match the magnitude
of the empirical error calculated using~\eqref{eq:results:rel_error_empirical}.
For this experiment, we determined the constants by a least squares fit over a
range of $\epsilon_{\text{id}}$ (as shown in \Cref{fig:results:model_proxy}),
with values reported in~\Cref{tbl:results:constants}. The results for this
experiment are shown in~\Cref{fig:results:model_error}. We observe that both
in two and three dimensions, the slope of the empirical data (dashed lines) follows
the theoretical model (full lines) in the asymptotic regime.

\begin{figure}[ht!]
\begin{subfigure}{0.5\linewidth}
\centering
\includegraphics[width=0.7\linewidth]{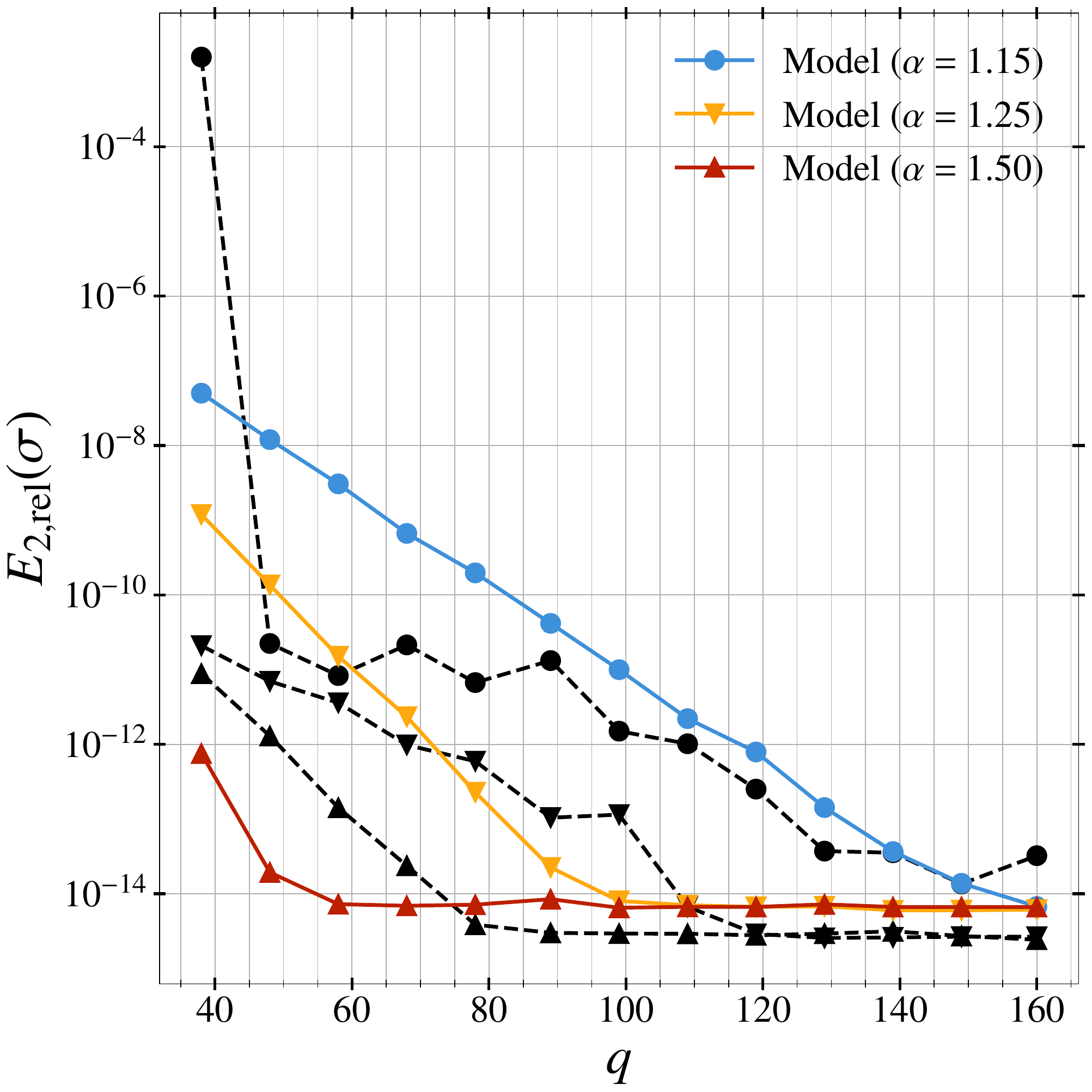}
\caption{2D double-layer ($-\sfrac{1}{2} \mathcal{I} + \mathcal{D}$)}
\end{subfigure}%
\begin{subfigure}{0.5\linewidth}
\centering
\includegraphics[width=0.7\linewidth]{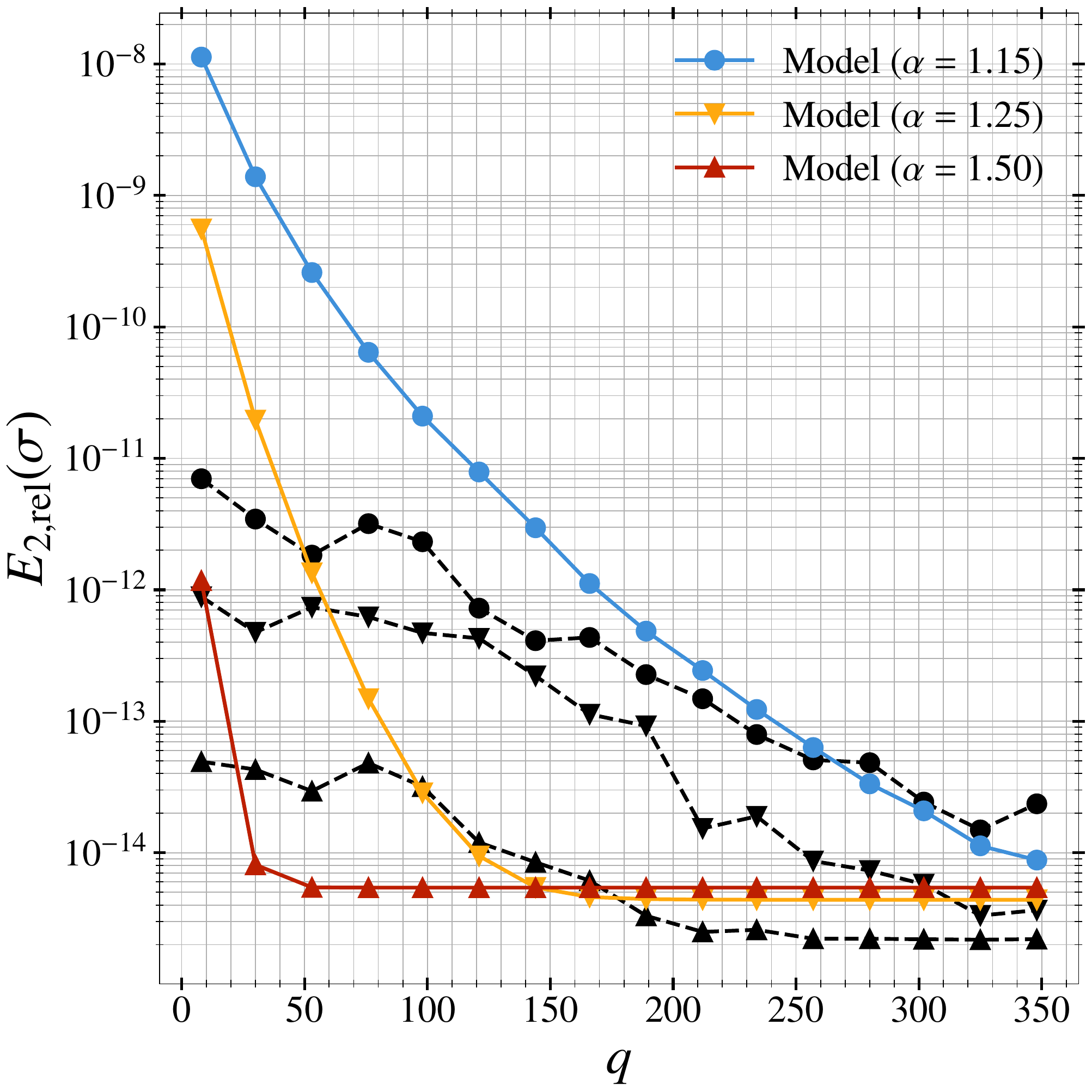}
\caption{3D double-layer ($-\sfrac{1}{2} \mathcal{I} + \mathcal{D}$)}
\end{subfigure}
\caption{Comparison of empirical error data (dashed) and the error model from
  \Cref{thm:errors:ids} (full) on a (a) two-dimensional and a (b)
  three-dimensional double-layer potential at $\epsilon_{\text{id}} = 10^{-15}$
  for varying proxy counts $q$ and proxy radius factors $\alpha$.}
\label{fig:results:model_error}
\end{figure}

\subsubsection{Proxy Estimates as a Function of ID Tolerance}
\label{sssc:result:model:eps_vs_q}

In this experiment, we examine the estimates for the proxy count $q$ as a function
of $\epsilon_{\text{id}}$ described by~\Cref{cor:errors:proxy_count}. The
proxy count is taken as $q = p (p + 1)$, in which the expansion order $p$ is
approximated by~\eqref{eq:errors:proxy_count}.

The constants from~\eqref{eq:errors:proxy_count} are also multiplied by the values
given in~\Cref{tbl:results:constants}, analogously to the previous section and
consistent with \eqref{eq:results:rel_error_model}. The empirical results in this
case are obtained by iteration: we start with a $q_{\text{empirical}} = 8$ and
increase it until the empirical error matches the tolerance $10 \epsilon_{\text{id}}$.
The tolerance is increased because the empirical error cannot always achieve
the exact value, as seen in~\Cref{fig:results:target_weight_matrix}.

\begin{figure}[ht!]
\begin{subfigure}{0.5\linewidth}
\centering
\includegraphics[width=0.7\linewidth]{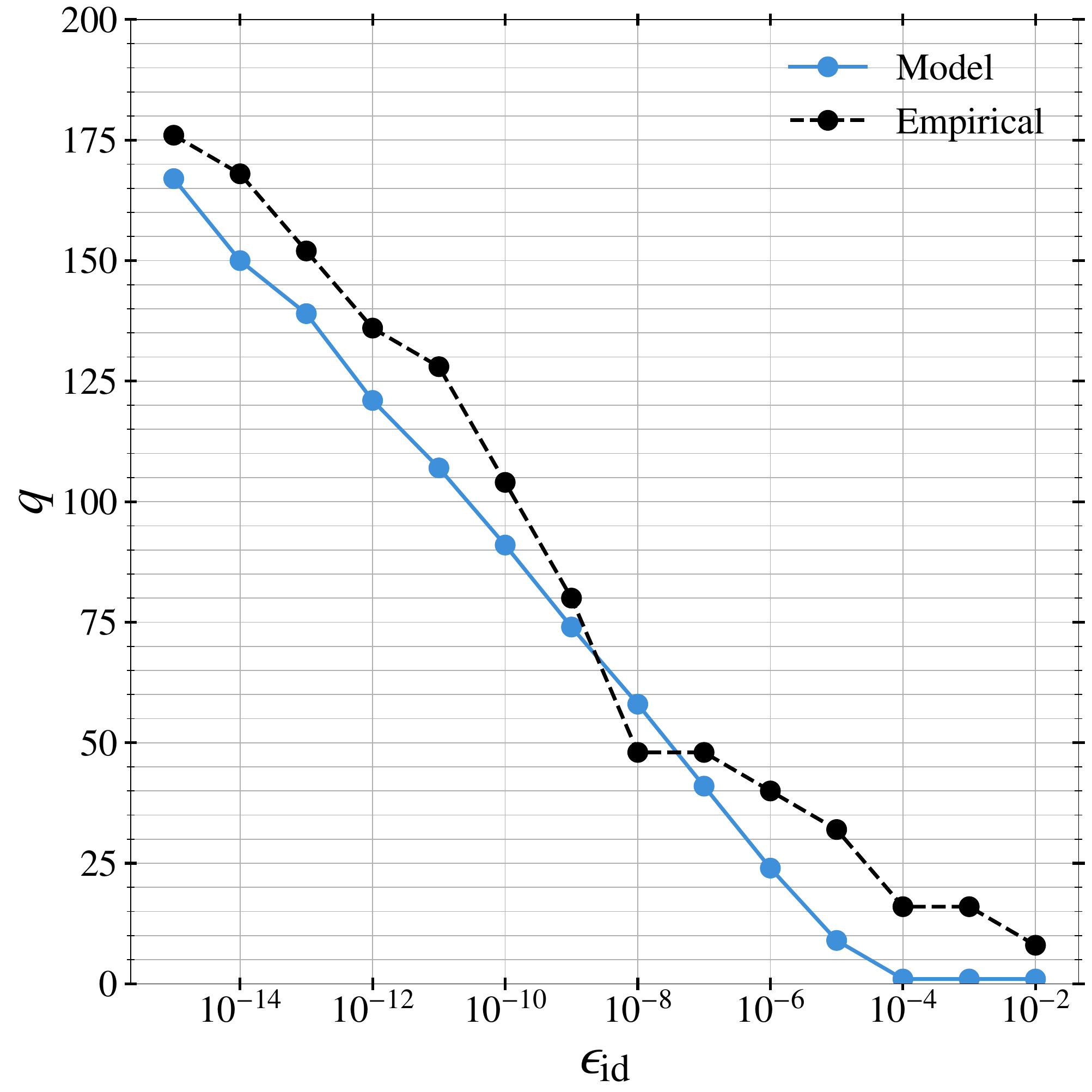}
\caption{2D single-layer ($\mathcal{S}$).}
\end{subfigure}%
\begin{subfigure}{0.5\linewidth}
\centering
\includegraphics[width=0.7\linewidth]{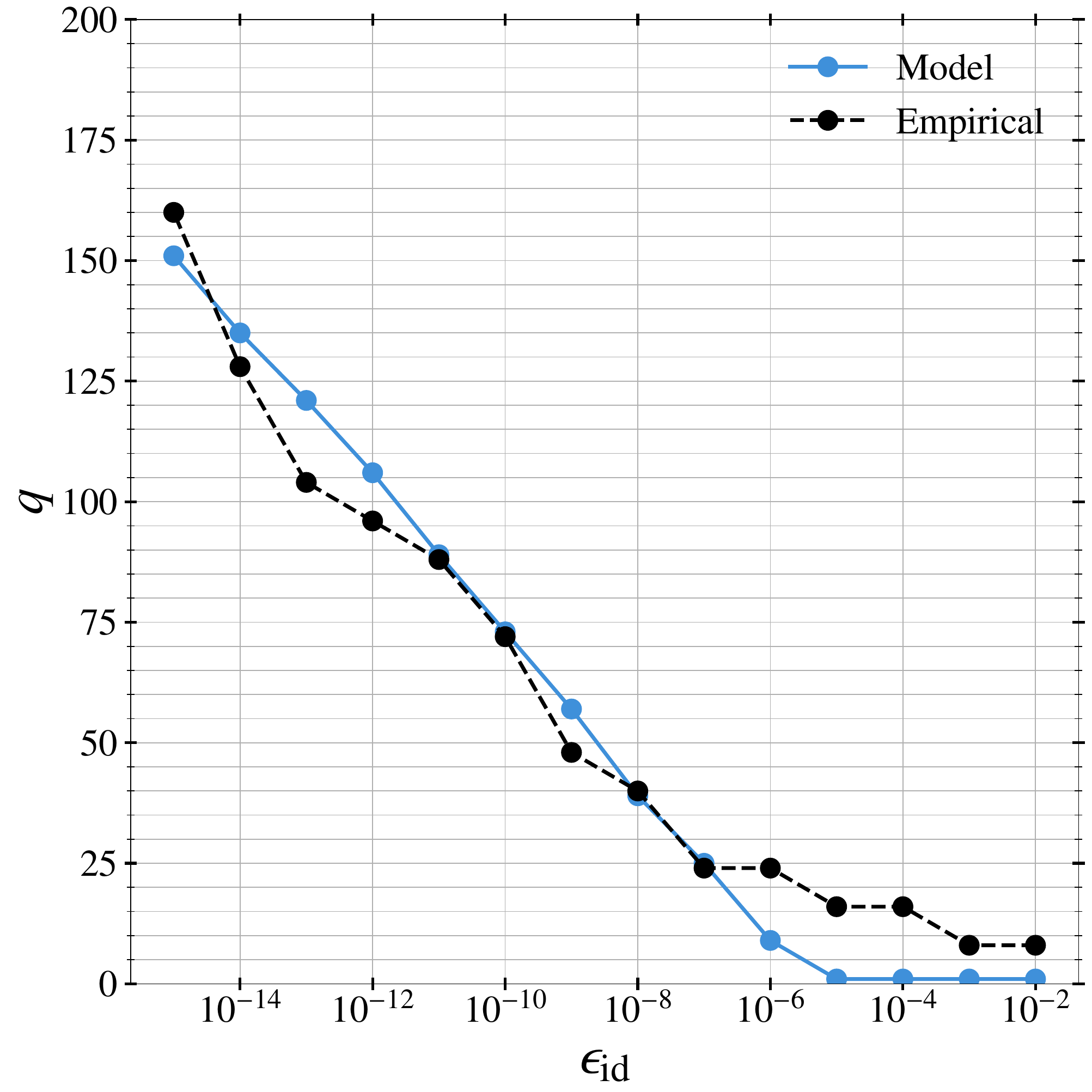}
\caption{2D double-layer ($-\sfrac{1}{2} \mathcal{I} + \mathcal{D}$).}
\end{subfigure}

\begin{subfigure}{0.5\linewidth}
\centering
\includegraphics[width=0.7\linewidth]{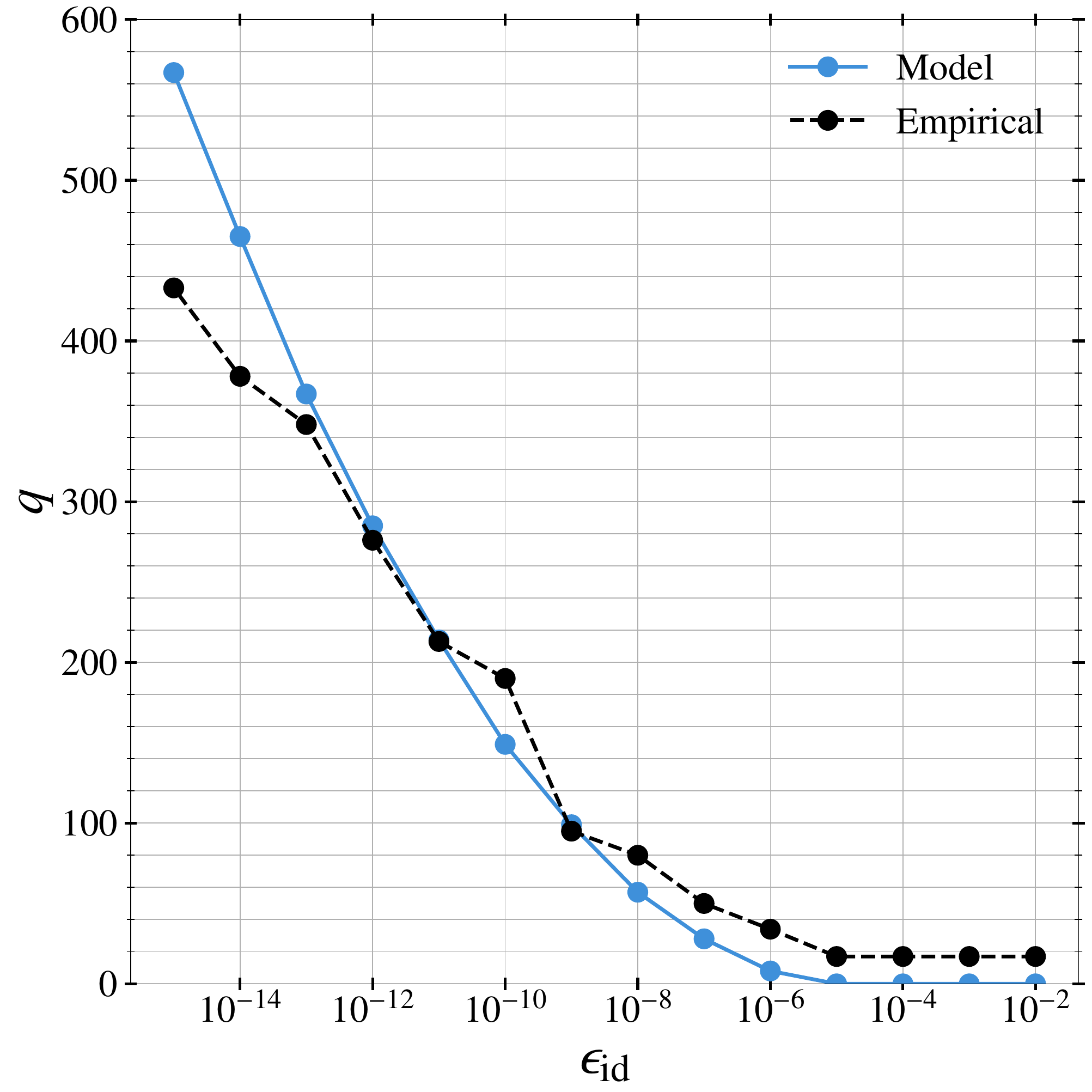}
\caption{3D single-layer ($\mathcal{S}$).}
\end{subfigure}%
\begin{subfigure}{0.5\linewidth}
\centering
\includegraphics[width=0.7\linewidth]{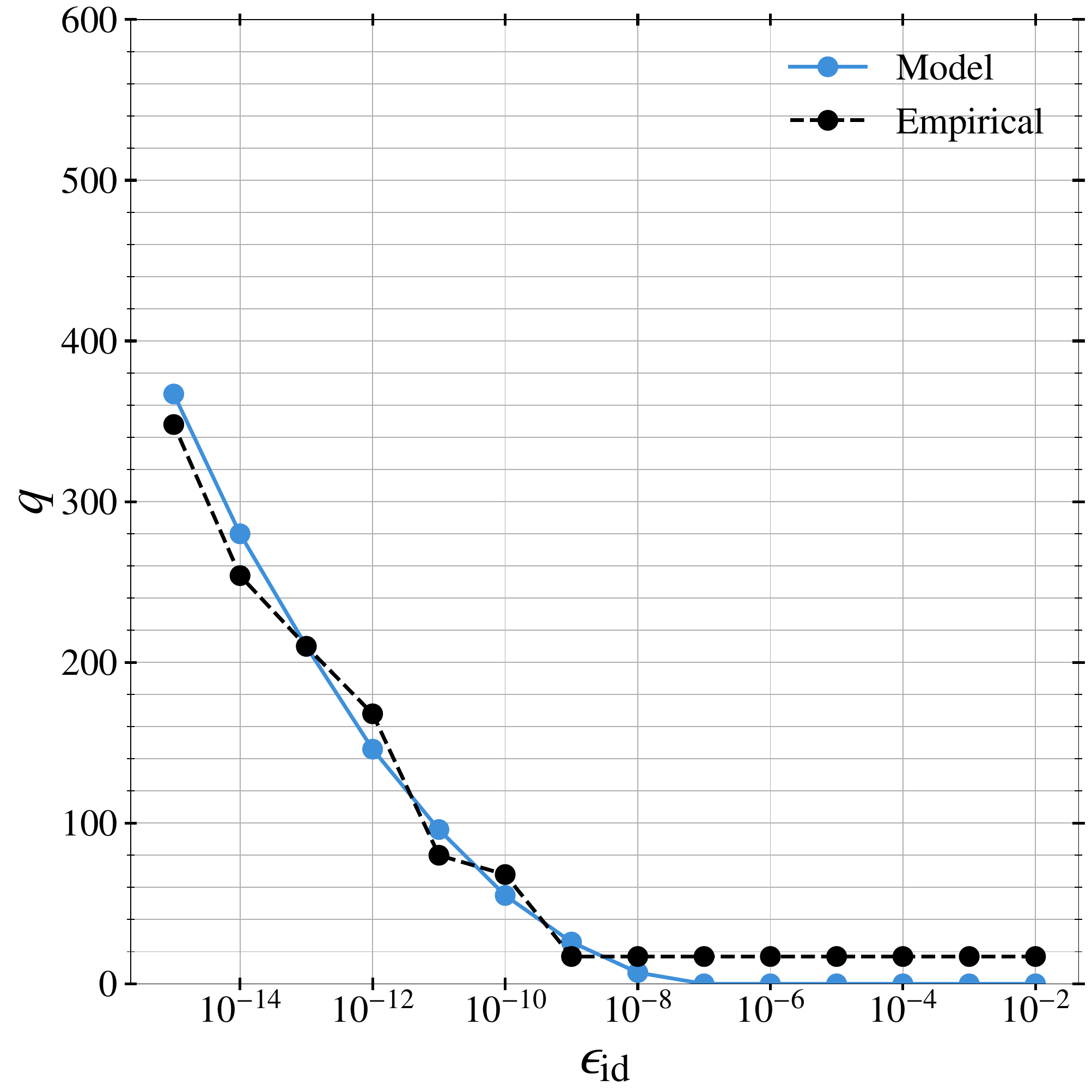}
\caption{3D double-layer ($-\sfrac{1}{2} \mathcal{I} + \mathcal{D}$).}
\end{subfigure}
\caption{Proxy count estimation based on the model from~\Cref{cor:errors:proxy_count} as
  a function of the ID tolerance for a (a,c) single-layer and (b,d) double-layer
  potential. The constants in the model are modified according
  to~\Cref{tbl:results:constants}.}
\label{fig:results:model_proxy}
\end{figure}

\Cref{fig:results:model_proxy} shows the results for both a single-layer and a
double-layer potential. We can see that in both cases the model gives a good
estimate of the required number of proxy points. This is true even for the
case of the double-layer potential, which was not technically included in the
analysis from~\Cref{sc:errors}.

\subsection{Forward Accuracy}
\label{ssc:result:forward}

We continue by verifying the asymptotic behavior of the forward error with respect
to the ID tolerance $\epsilon_{\text{id}}$. Given the ID tolerance, we compute
the proxy count using~\eqref{eq:errors:proxy_count} and the
constants from~\Cref{tbl:results:constants} as before. The proxy radius factor
is fixed at $\alpha = 1.15$. We use the relative error measure
\eqref{eq:results:rel_error_empirical}. As before, the solution vector
$\vect{\sigma}$ is randomly generated and the right-hand side $\vect{b}$ is
computed from $\vect{b} = \matr{A} \vect{\sigma}$ by direct point-to-point evaluation
of the QBX expansion.

\begin{figure}[ht!]
\begin{subfigure}{0.5\linewidth}
\centering
\includegraphics[width=0.7\linewidth]{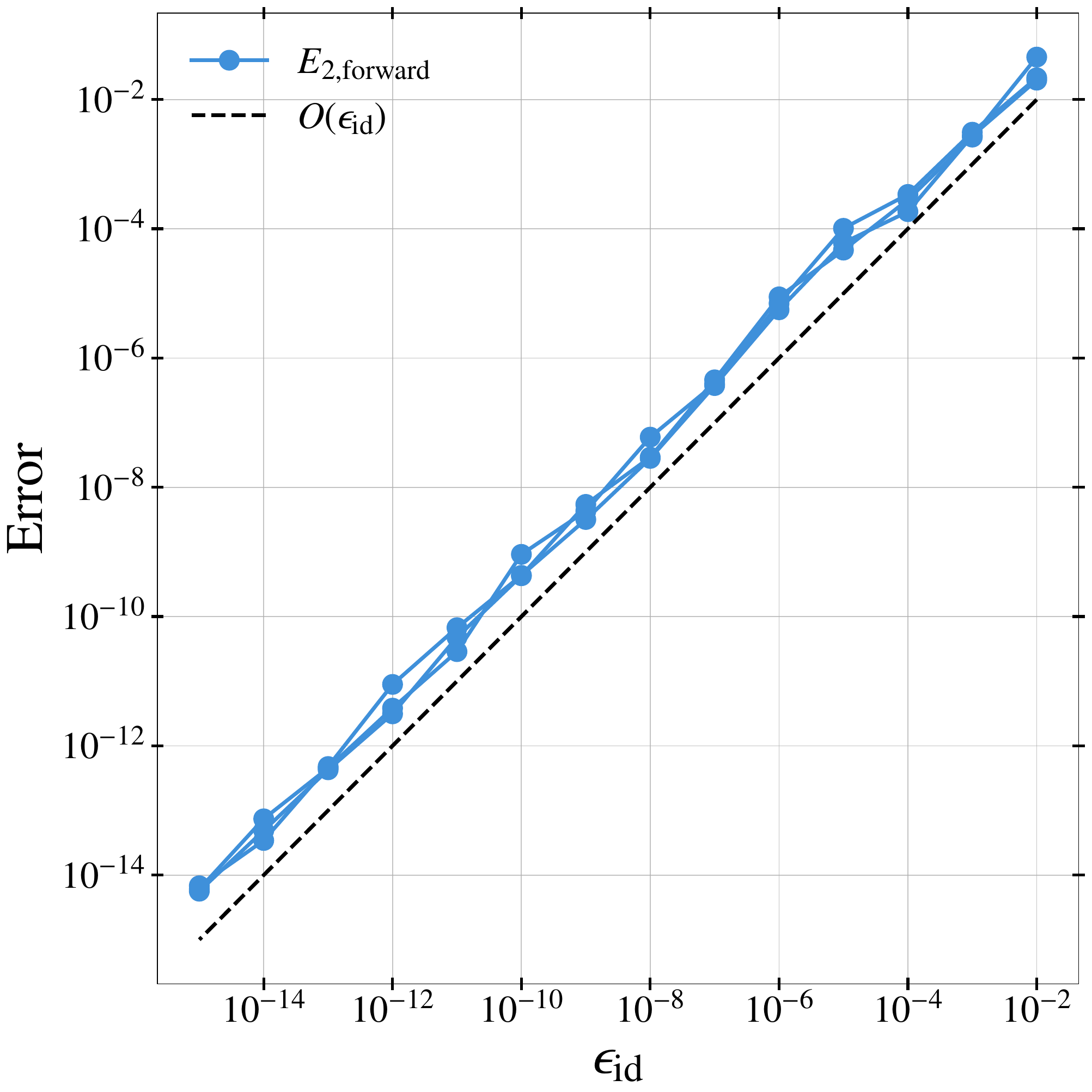}
\caption{2D single-layer ($\mathcal{S}$).}
\end{subfigure}%
\begin{subfigure}{0.5\linewidth}
\centering
\includegraphics[width=0.7\linewidth]{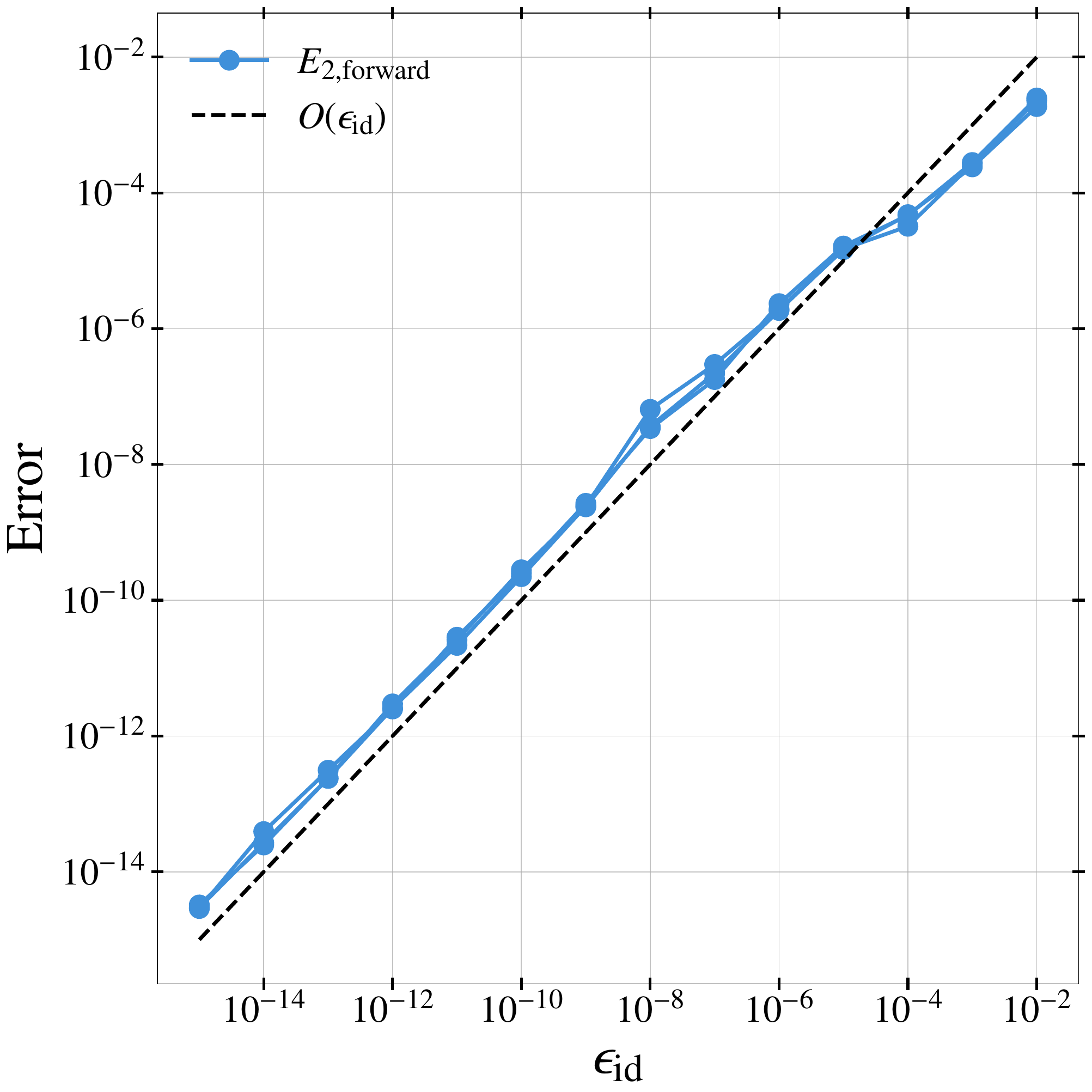}
\caption{2D double-layer ($-\sfrac{1}{2} \mathcal{I} + \mathcal{D}$).}
\end{subfigure}

\begin{subfigure}{0.5\linewidth}
\centering
\includegraphics[width=0.7\linewidth]{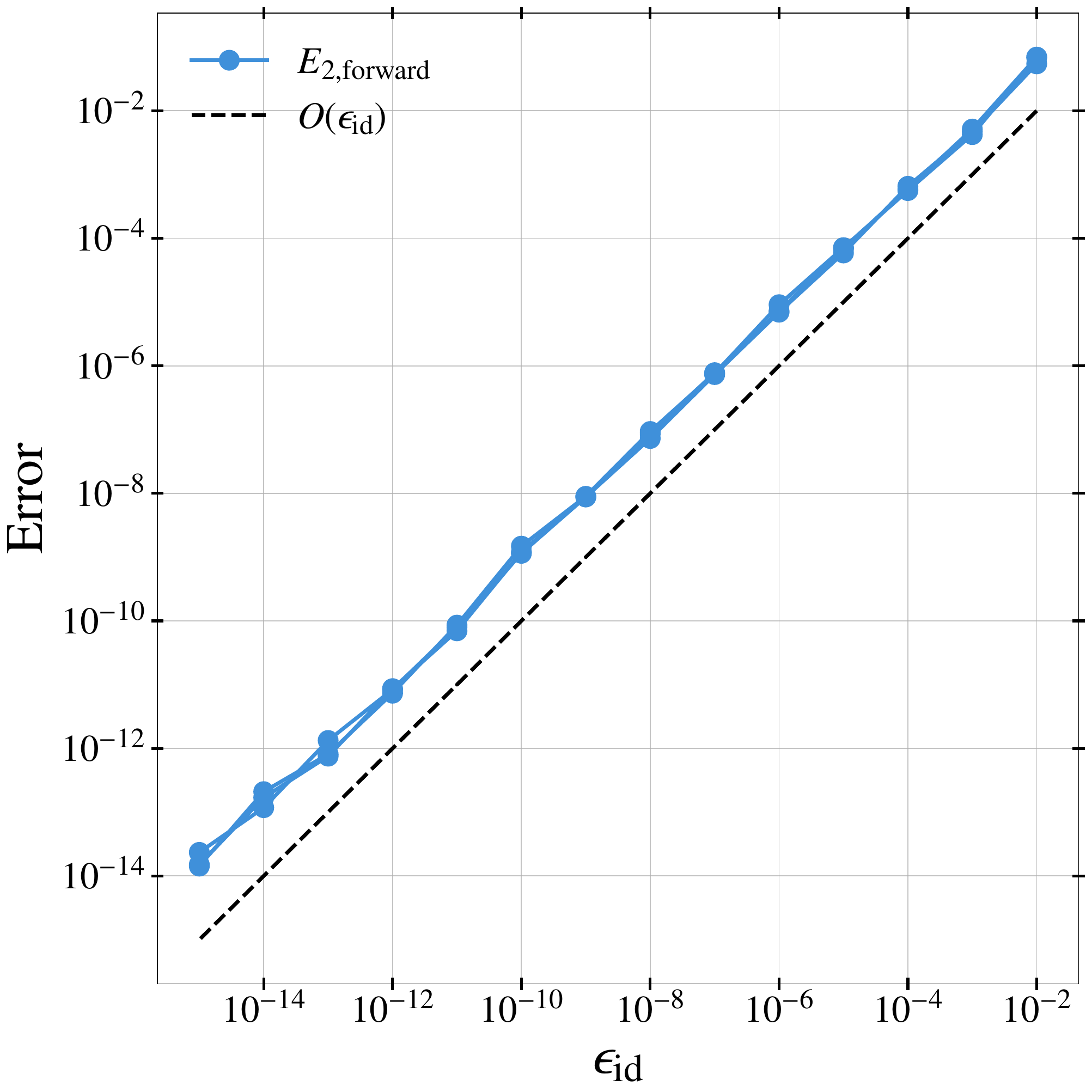}
\caption{3D single-layer ($\mathcal{S}$).}
\end{subfigure}%
\begin{subfigure}{0.5\linewidth}
\centering
\includegraphics[width=0.7\linewidth]{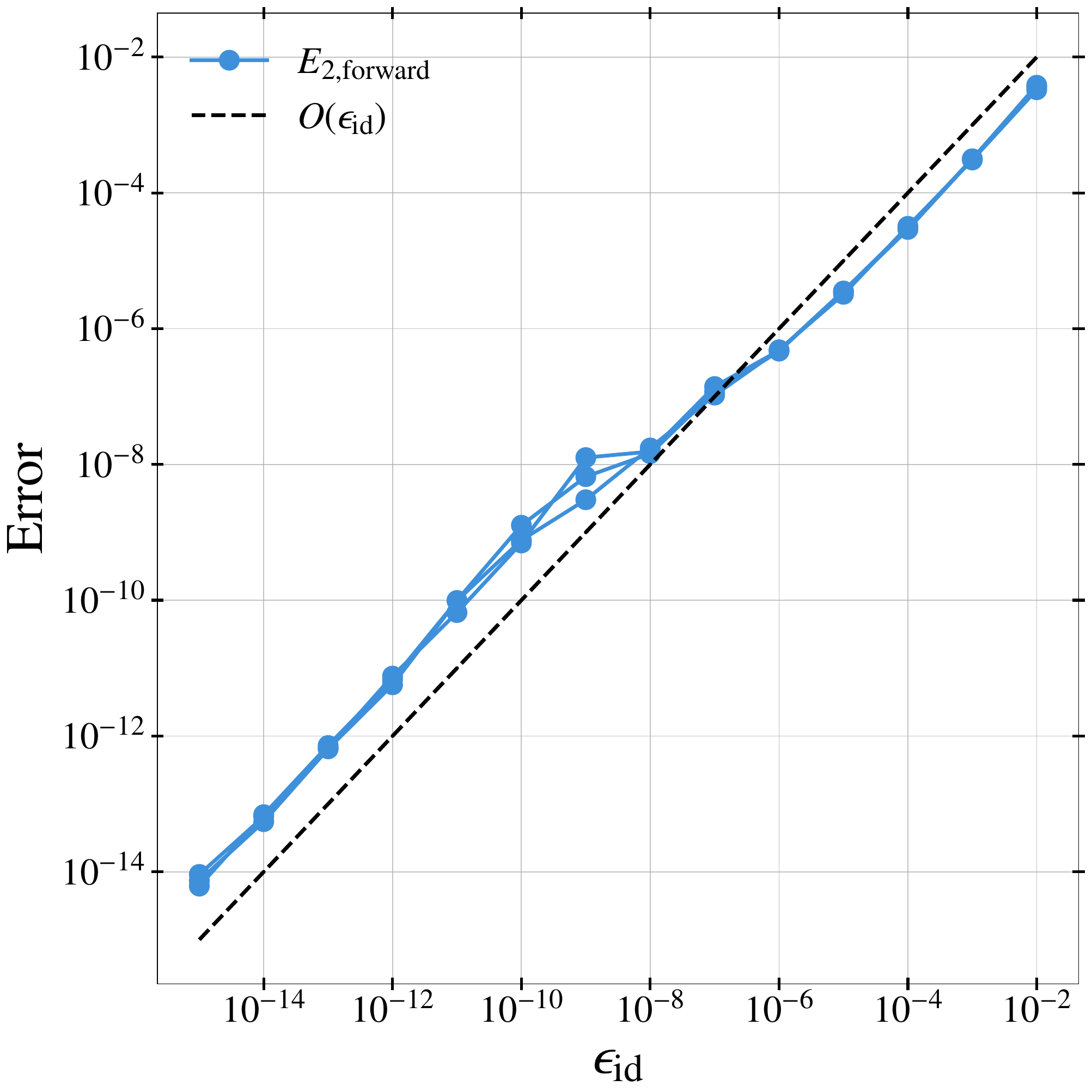}
\caption{3D double-layer ($-\sfrac{1}{2} \mathcal{I} + \mathcal{D}$).}
\end{subfigure}
\caption{Forward error for a Laplace (a,c) single-layer potential and a
  (b,d) double-layer potential from~\eqref{eq:prelim:lp}. Each (blue) line curve
  represents a different randomly generated solution vector $\vect{\sigma}$ and
  right-hand side $\vect{b}$ for the same approximate operator $\matr{A}_\epsilon$.}
\label{fig:results:forward}
\end{figure}

We then skeletonize the single-layer and double-layer potentials in two and
three dimensions for the range of $\epsilon_{\text{id}} \in \{10^{-1}, \dots, 10^{-15}\}$.
As can be seen in \Cref{fig:results:forward}, an error proportional to the ID
tolerance is achieved at every value of the parameter. Furthermore, the
approximation can achieve close to machine epsilon for IEEE double precision
floating point numbers.

\subsection{Solution Accuracy}
\label{ssc:result:solution}

In this final test of the direct solver, we verify the
accuracy obtained when solving a linear system, i.e.,\ the solution accuracy.
As in the forward case, we consider the relative error given by
\[
E_{2, \text{solution}}(\epsilon_{\text{id}}) =
  \frac{\|\vect{\sigma} - \matr{A}^{-1}_{\epsilon_{\text{id}}} \vect{b}\|_2}
  {\|\vect{b}\|_2},
\]
where $\vect{\sigma}$ is randomly generated and $\vect{b} = \matr{A} \vect{\sigma}$.
As shown in~\Cref{sc:errors}, the error in the approximation of the inverse
operator is not necessarily bounded in the same way as that of the forward operator.
Furthermore, the conditioning of the operator will play a large part in the
observed solution error. This is especially the case for the single-layer potential,
which is known to be compact ($0$ is an accumulation point of the spectrum)
with an undefined (``infinite'') condition number that makes itself felt with a
finer discretization. For the double-layer representation of the Dirichlet boundary
value problem, the inverse and any approximating sequence are better behaved
according to Anselone's Theorem (see~\cite[Theorem 10.12]{Kress2013}).

\begin{figure}[ht!]
\begin{subfigure}{0.5\linewidth}
\centering
\includegraphics[width=0.7\linewidth]{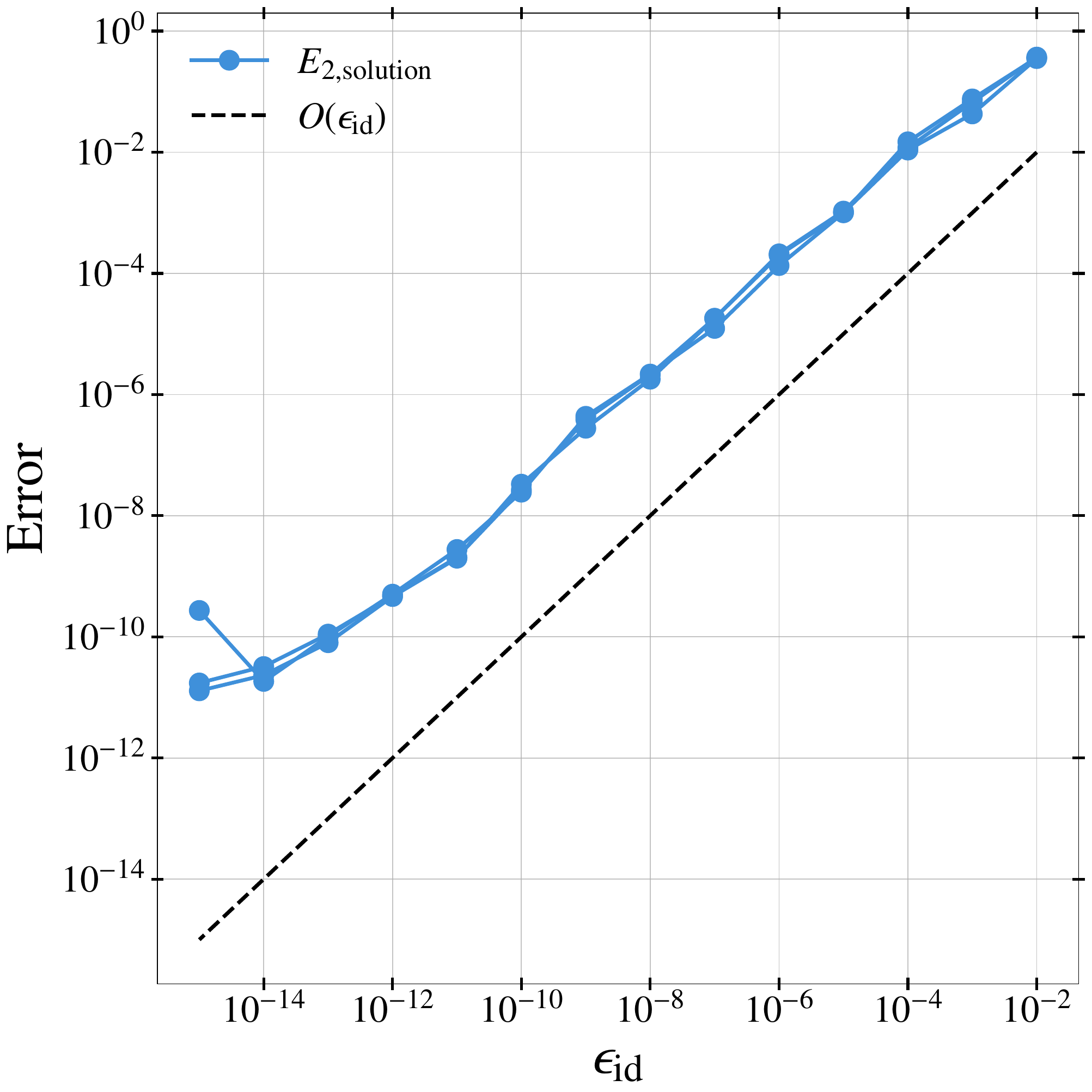}
\caption{2D single-layer ($\mathcal{S}$).}
\end{subfigure}%
\begin{subfigure}{0.5\linewidth}
\centering
\includegraphics[width=0.7\linewidth]{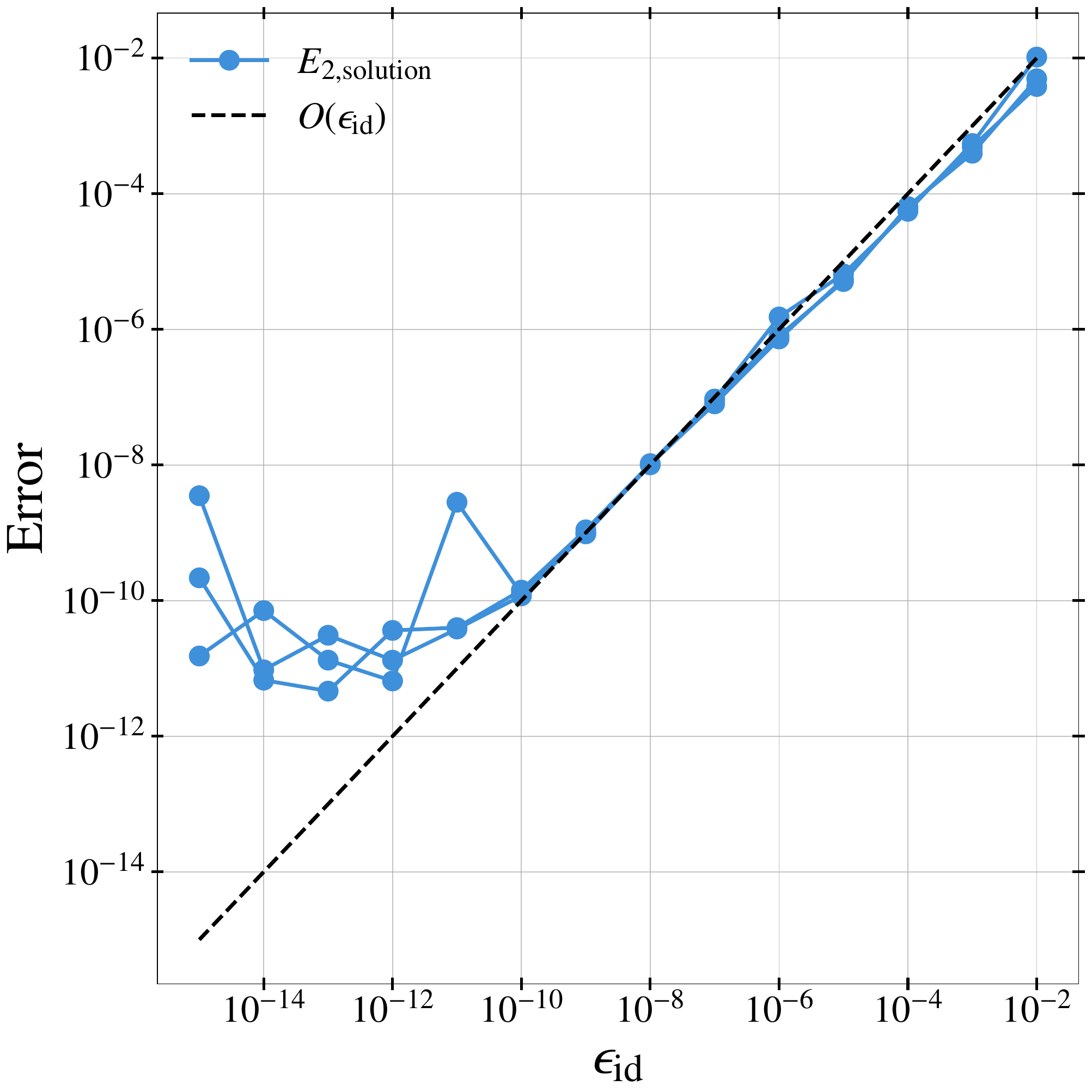}
\caption{2D double-layer ($-\sfrac{1}{2} \mathcal{I} + \mathcal{D}$).}
\end{subfigure}

\begin{subfigure}{0.5\linewidth}
\centering
\includegraphics[width=0.7\linewidth]{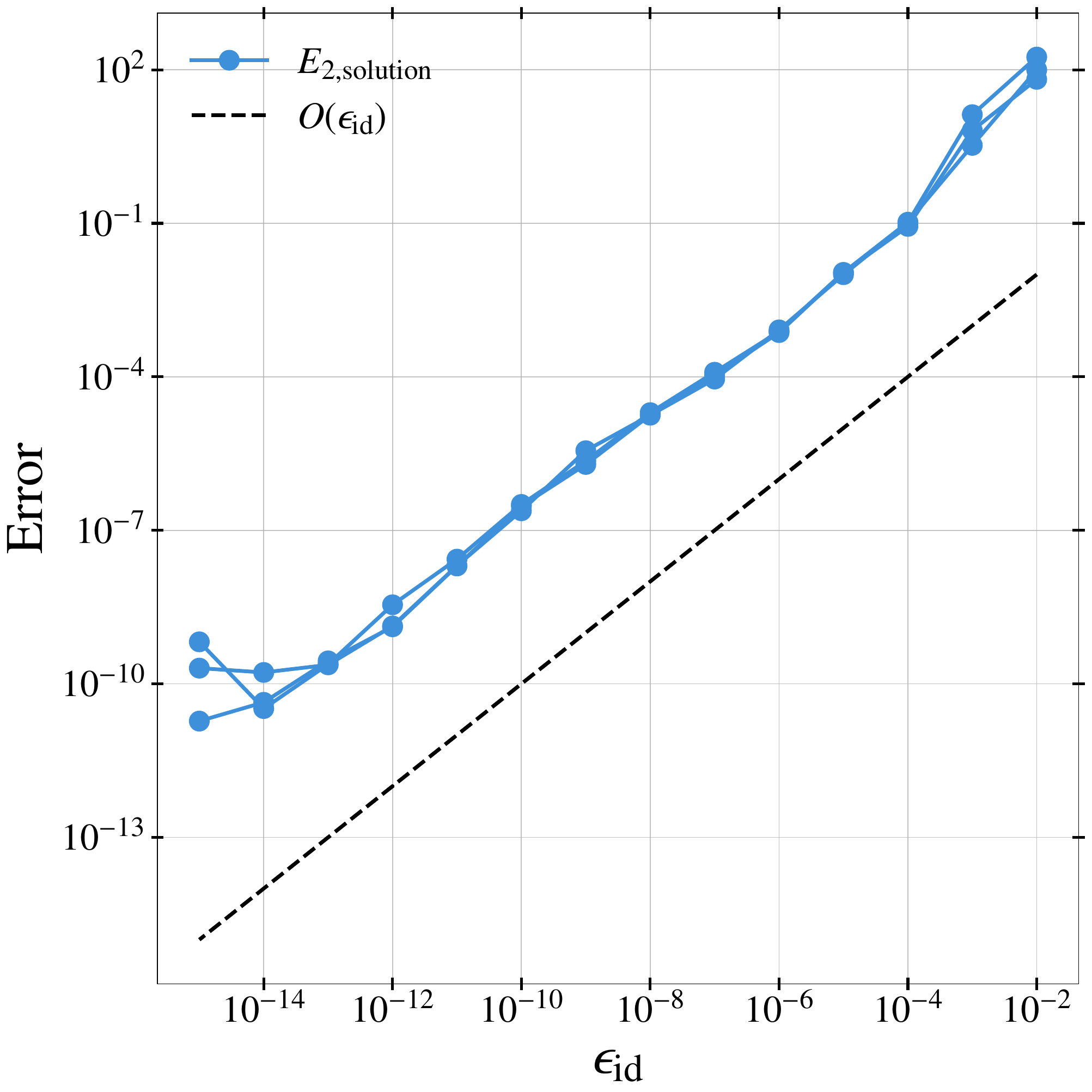}
\caption{3D single-layer ($\mathcal{S}$).}
\end{subfigure}%
\begin{subfigure}{0.5\linewidth}
\centering
\includegraphics[width=0.7\linewidth]{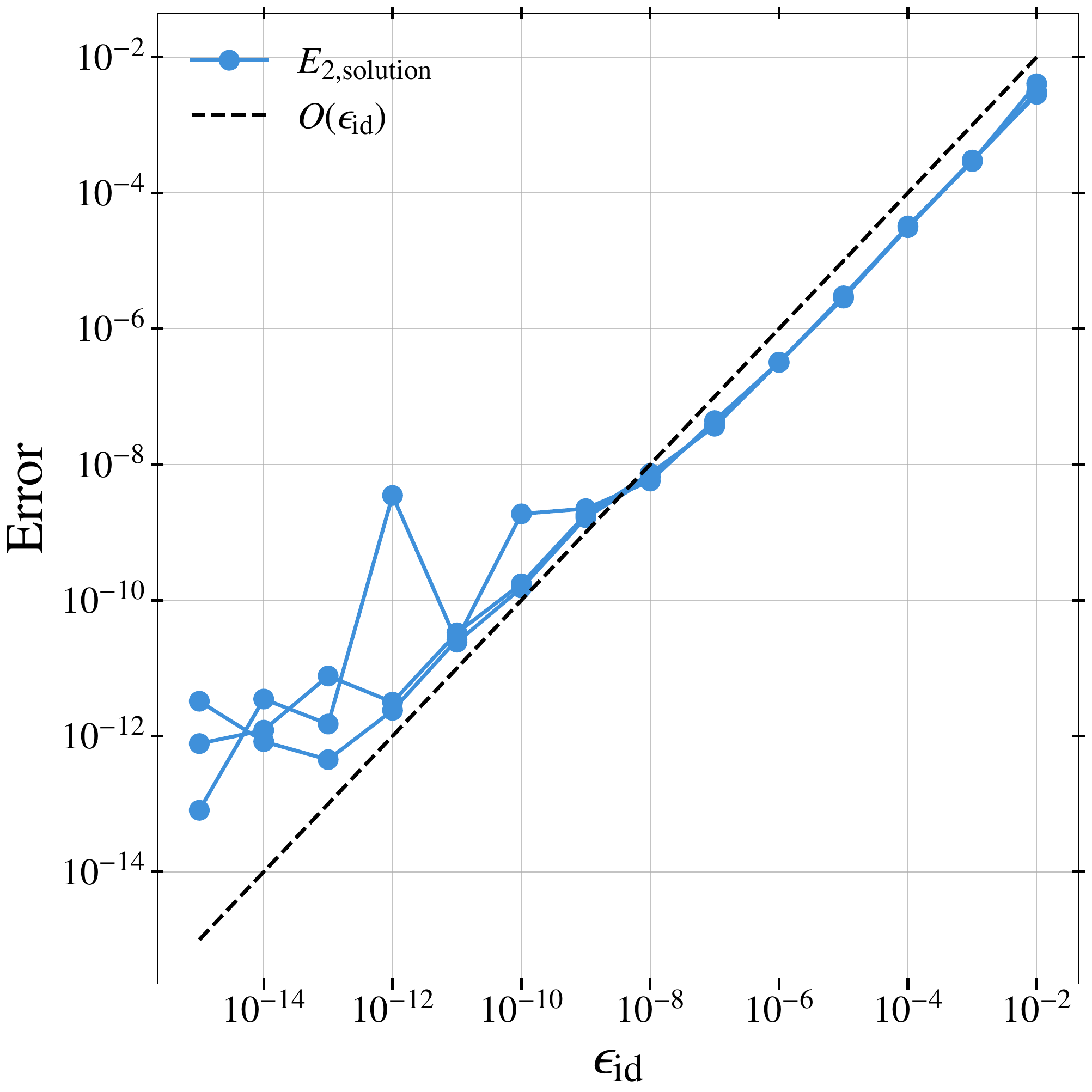}
\caption{3D double-layer ($-\sfrac{1}{2} \mathcal{I} + \mathcal{D}$).}
\end{subfigure}
\caption{Solution error for a Laplace (a,c) single-layer potential and a
  (b,d) double-layer potential from~\eqref{eq:prelim:lp}. Each (blue) line curve represents a different
  randomly generated solution vector $\vect{\sigma}$ and right-hand side $\vect{b}$
  for the same approximate operator $\matr{A}_\epsilon$.}
\label{fig:results:solution}
\end{figure}

As with the forward accuracy from the previous section, we check that the
solution error is achieved for a range of $\epsilon_{\text{id}}$ using the proxy
count approximations from~\Cref{fig:results:model_proxy}. The results can be
seen in~\Cref{fig:results:solution} for the range of operators under consideration.
As expected, the double-layer representation is well-behaved in both 2D and 3D for
the chosen setup. However, the error in the single-layer potential is offset by
a constant factor that appears to depend on the condition number of the discrete
operator. This matches the theory presented in~\Cref{sc:errors} and shows that
the error model allows accurate prediction of the approximation error of the
inverse skeletonized operators.

\subsection{Scaling of Computational Cost}
\label{ssc:result:scaling}

Another important benefit of the direct solver is its asymptotic efficiency, as
compared to a standard application of an LU-based factorization. As discussed
in~\cite{Gillman2012}, we expect that the implementation presented
in~\Cref{ssc:schemes:skeletonization} will result in a scaling from
\eqref{eq:schemes:scaling}
\[
T =
\begin{cases}
O(n), & \quad d = 2, \\
\displaystyle
O(n^{\frac{3}{2}}), & \quad d = 3.
\end{cases}
\]

We verify that the scaling of the implementation matches this expectations
in~\Cref{fig:results:scaling_single} and~\Cref{fig:results:scaling_double} for
the single-layer and double-layer potentials, respectively. In two dimensions,
we take $\{384, 608, 1024, 2048, 3072, 4096\}$ elements in the
\stagetwo{} discretization, which results in $\{1920, 3040, 5120, 10240,
\penalty0 15360, 20480\}$ total degrees of freedom. Similarly, in three dimensions we
take a spherical geometry (as opposed to the torus used previously) and
$\{6990, 8415,\penalty0 14775, 21180, 32955, 60435\}$ total degrees of freedom. The
spherical geometry is generated by \texttt{gmsh} and is chosen because it features
less regular element shapes and sizes. \reva{At a certain geometry size in 3D, the solve
has lower-than-expected cost, leading to a `downward spike' in the plot.
A detailed investigation revealed that this case exhibited an
unbalanced clustering of the geometry at the second level of the tree, leading to
the root-level solve having reduced cost.}

\begin{figure}[ht!]
\begin{subfigure}{0.5\linewidth}
\centering
\includegraphics[width=0.8\linewidth]{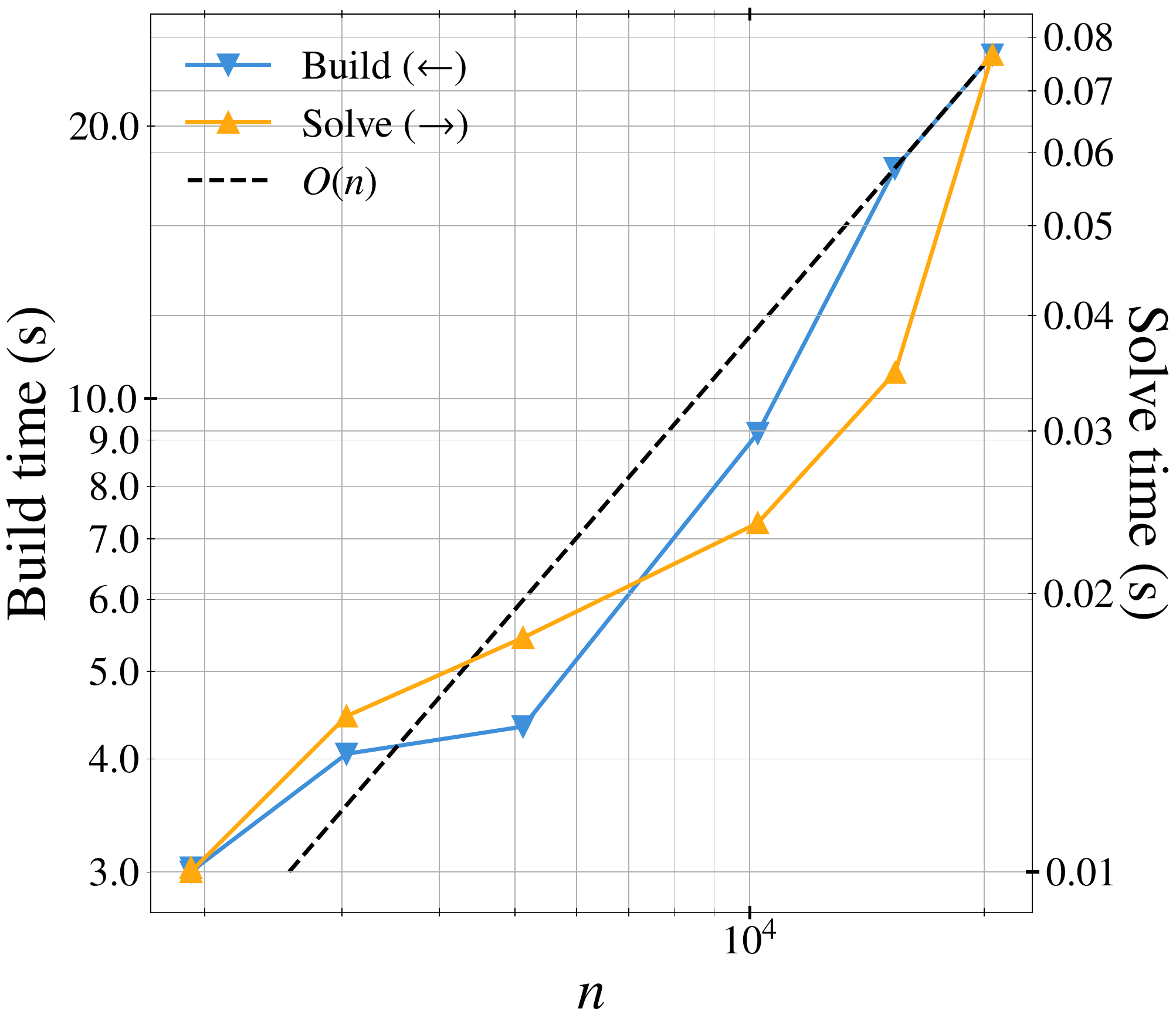}
\caption{2D single-layer.}
\end{subfigure}%
\begin{subfigure}{0.5\linewidth}
\centering
\includegraphics[width=0.8\linewidth]{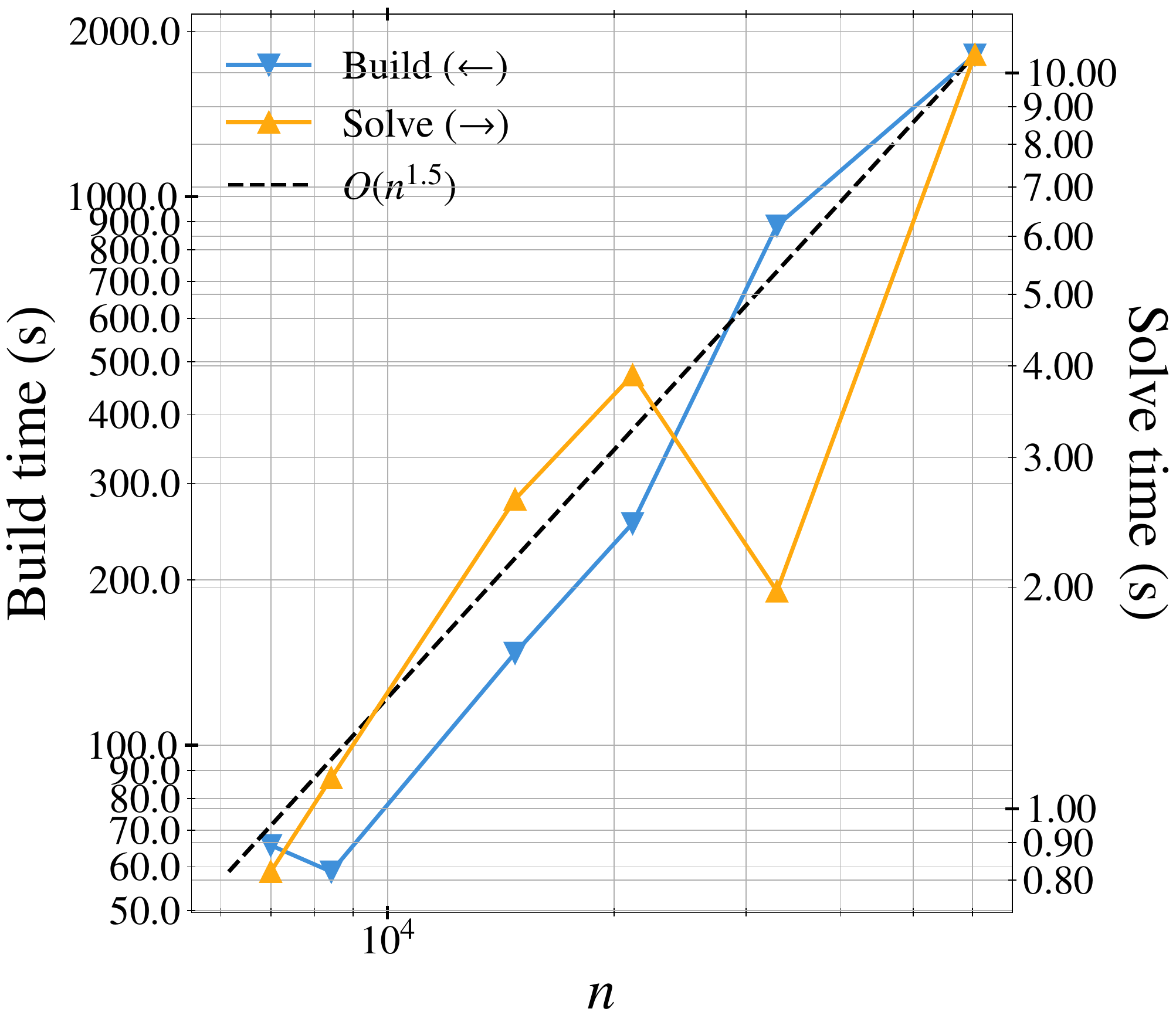}
\caption{3D single-layer.}
\end{subfigure}
\caption{Timing (in seconds) for a solving a single-layer potential using the
  skeletonized direct solver in a logarithmic scale.}
\label{fig:results:scaling_single}
\end{figure}

\begin{figure}[ht!]
\begin{subfigure}{0.5\linewidth}
\centering
\includegraphics[width=0.8\linewidth]{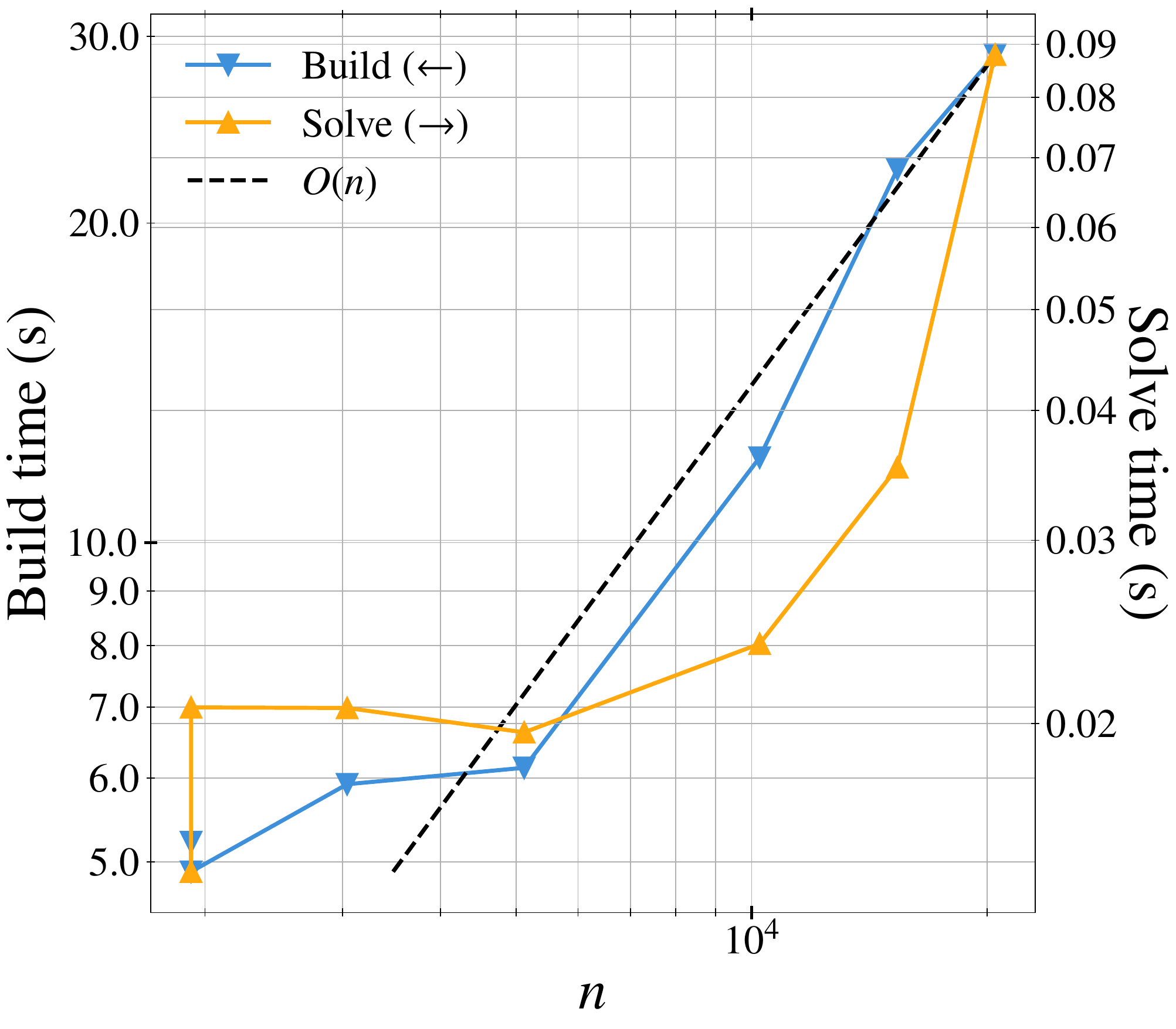}
\caption{2D double-layer.}
\end{subfigure}%
\begin{subfigure}{0.5\linewidth}
\centering
\includegraphics[width=0.8\linewidth]{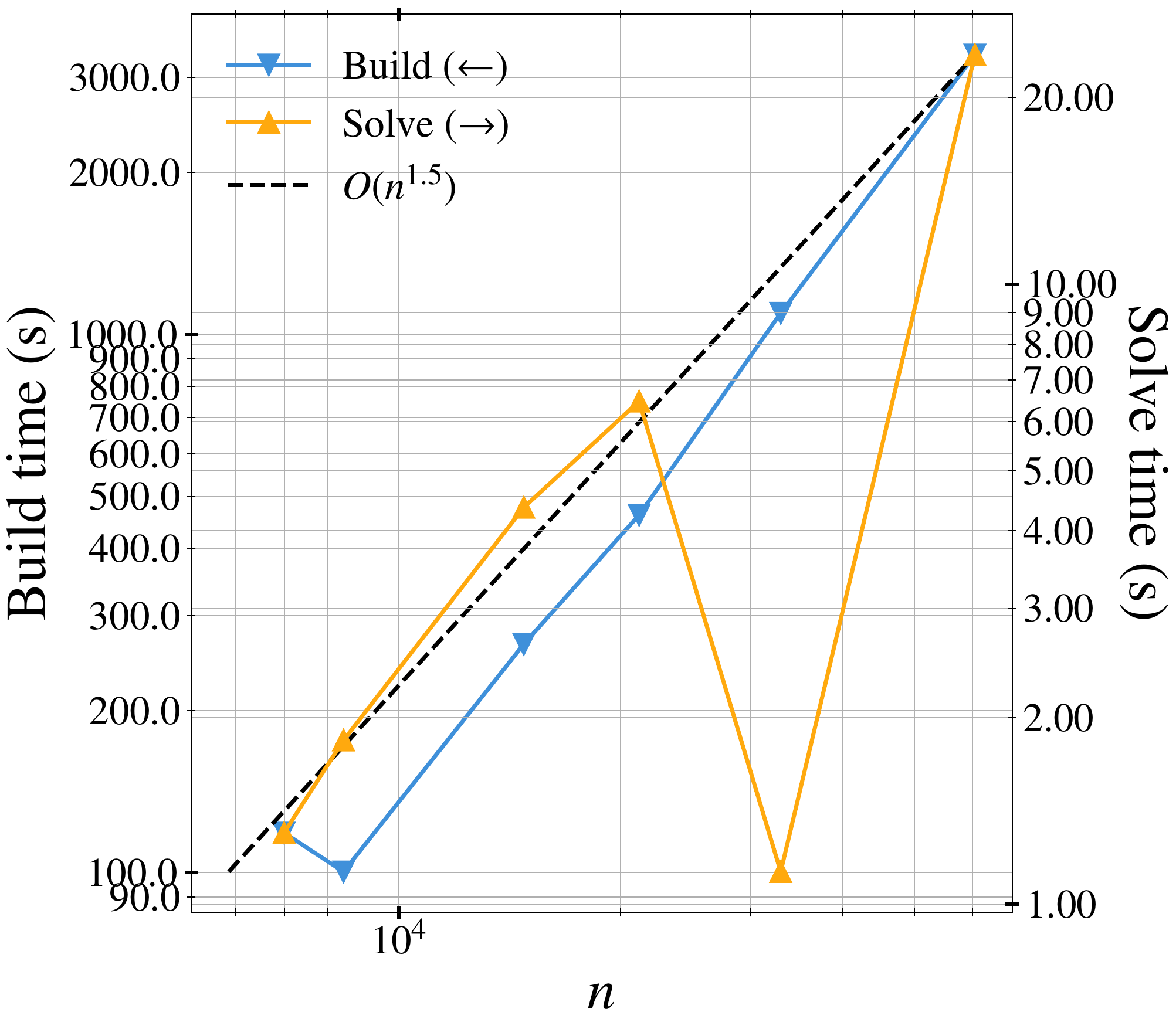}
\caption{3D double-layer.}
\end{subfigure}
\caption{
  Timing (in seconds) for a solving a double-layer potential using the skeletonized
  direct solver in a logarithmic scale. The presented results are a mean over multiple
  runs with one standard deviation shown as a shaded region.}
\label{fig:results:scaling_double}
\end{figure}

In both cases, the asymptotic behavior of the scaling matches the expected
results from~\eqref{eq:schemes:scaling}. We also observe that the most
time-consuming step of direct solver is the construction of the recursive skeletonization,
where the interpolative decomposition dominates when compared to the source-proxy
and source-target evaluations. Therefore, we expect that improvements in the
efficiency of the ID or a similar algorithm (e.g.,\ a CUR decomposition) can
significantly lower the cost of the direct solver construction step. However, the
solution process does not involve computation of IDs and is therefore
significantly cheaper. Additionally, it can be applied repetitively once
constructed, e.g.,\ as a preconditioner.

\section{Conclusions}
\label{sc:conclusions}

This work introduces a direct solver based on proxy skeletonization for homogeneous
elliptic operators discretized with the Quadrature by Expansion method. We establish
that the direct solver framework is compatible with the QBX method with only small
modifications to the construction and choice of parameters. This extension is
accompanied by a detailed error estimate that highlights the dependence on the
parameters of interest, such as the tolerance and the proxy count and radius of
the spherical proxy surface.

For a practical direct solver implementation, we provide means to control
the error introduced by the skeletonization (handled by the ID) and by the
proxy approximation. The main component of this result is an analysis of the
error introduced by using a proxy surface to approximate far field interactions.
This is done by defining the Poisson kernel on the sphere and constructing a
multipole-based estimate that applies up to the boundary of the proxy surface.
This model is then validated against empirical results on non-trivial geometries.
We find that it is in good agreement with respect to the asymptotic behavior,
but pessimistic with respect to the model constants.

Numerical experiments show that the resulting direct solver can achieve
close to design tolerance, based on the ID tolerance, for the matrix-vector
multiplication. The inverse operator can suffer from additional errors that are
likely due to conditioning. Furthermore, the solver can attain predicted
asymptotic scaling of computational cost for the different operators. The
theoretical analysis has been carried out in three dimensions,
with a straightforward generalization to the two-dimensional case. The associated
numerical tests have been carried out
in both two and three dimensions and have shown that a robust implementation of
the direct solver makes no distinction between the two cases.

Further optimizations to the scheme are the subject of future study. In particular,
it has been shown that the 3D direct solver for HSS matrices can also attain
linear scaling by further recursive skeletonization of the interpolation matrices
\cite{Ho2016}. Extensions of the error model to other families of elliptic equations,
such as Helmholtz and Yukawa, are also important and can highlight the dependence
on physical parameters, such as the wave number, that impact the efficiency of the
direct solver.

\appendix

\section{Convergence with Grid Size}
\label{ax:convergence}

\begin{figure}[ht!]
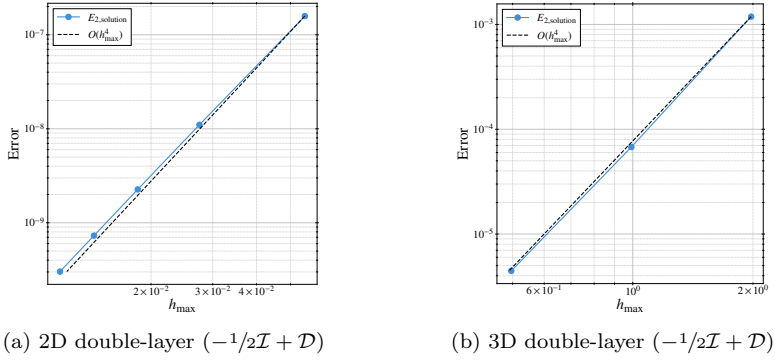

\begin{subfigure}{0.5\linewidth}
\centering
\includegraphics[width=0.7\linewidth]{ds-backward-convergence-2d}
\caption{2D double-layer ($-\sfrac{1}{2} \mathcal{I} + \mathcal{D}$)}
\end{subfigure}%
\begin{subfigure}{0.5\linewidth}
\centering
\includegraphics[width=0.7\linewidth]{ds-backward-convergence-3d}
\caption{3D double-layer ($-\sfrac{1}{2} \mathcal{I} + \mathcal{D}$)}
\end{subfigure}
\caption{Convergence of the direct solver with grid size $h_{\text{max}}$.}
\label{fig:ax:convergence}
\end{figure}

\reva{
To test the convergence of the QBX-based direct solver with respect to grid size,
we construct a known exact solution to~\eqref{eq:prelim:lp}. Following the results
from~\Cref{sc:results}, we focus on solving an exterior Dirichlet problem for the
Laplacian using the double-layer representation from~\eqref{eq:prelim:layer_potential}
on the geometries described in~\Cref{ssc:result:scaling}. This solution is constructed
by considering a set of charges $\gamma_j$, for $0 \le j < N_s$, located at points
$\vect{z}_j$ that are placed on a circle of radius $0.25$ and $0.5$ for the two- and
three-dimensional geometry, respectively. The charge strengths $\gamma_j$ are chosen
randomly on $[0, 1]$ and modified to have zero average. The boundary conditions are
evaluated directly from}
\[
b(\vect{x}_i) = \sum_{j = 0}^{N_s - 1} \gamma_j G(\vect{x_i}, \vect{z}_j).
\]
\reva{where $\vect{x}_i \in X$ are target points on the geometry $\Sigma$. To evaluate
the accuracy of the solution, we compare the result by evaluating the charges at a set
of exterior off-surface target points $\hat{\vect{z}}_i$, for $0 \le i < N_t$. Those target points are
placed on a circle of radius $3$ for both the two- and three-dimensional geometries. The
reference solution is computed as}
\[
x_{\text{ref}}(\hat{\vect{z}}_i) =
  \sum_{j = 0}^{N_s - 1} \gamma_j G(\hat{\vect{z}}_i, \vect{z}_i).
\]

\reva{For this test, we modify the quadrature orders used to discretize the
geometry to a Gauss--Legendre quadrature rule with 21 nodes in two dimensions and a
Vioreanu--Rokhlin quadrature rule with 153 nodes in three dimensions. This oversampling
is required to ensure convergence of the QBX method (see~\cite{Kloeckner2013}). Finally,
\Cref{fig:ax:convergence} shows the convergence results for the direct solver on this
known solution. In both cases, we can see that the error convergences with order
$p_{\text{qbx}}$. The presented errors are computed using the standard $\ell_2$ norm}
\[
E_{2, \text{solution}} =
  \frac{\|\vect{x} - \vect{x}_{\text{ref}}\|_2}{\|\vect{x}_{\text{ref}}\|_2}
\]
\reva{and the grid spacing $h_{\text{max}}$ is approximated from the
singular values of the parametrization Jacobian on the surface.}

\section*{Declarations}

\begin{acknowledgements}
This work was supported in part by West University of Timișoara, Romania, START
Grant No. 33580/25.05.2023 (Fikl). It was also supported by the US National
Science Foundation under award numbers DMS-1654756, SHF-1911019, OAC-1931577,
and DMS-2410943, as well as the Siebel School of Computing and Data Science at
the University of Illinois at Urbana-Champaign (Kloeckner).
\end{acknowledgements}

\noindent \textbf{Data Availability}.
The software developed and used in this study is openly available in the Zenodo repository at
\href{https://doi.org/10.5281/zenodo.15487041}{https://doi.org/10.5281/zenodo.15487041}.
This archive (version 2.0) includes all source code and documentation required to
reproduce the results, and is released under the MIT License.

\medskip

\noindent \textbf{Declaration of Interests}.
The authors have no competing interests to declare that are relevant to the content of this article.

\bibliography{qbx-direct}

\end{document}